\documentclass[a4paper, 10pt]{amsart}
\usepackage{amsmath}
\usepackage{amssymb, color, hyperref}

\theoremstyle{plain}
\newtheorem{theorem}{\bf Theorem}[section]
\newtheorem{proposition}[theorem]{\bf Proposition}
\newtheorem{lemma}[theorem]{\bf Lemma}
\newtheorem{corollary}[theorem]{\bf Corollary}

\theoremstyle{definition}

\newcommand{\N}{\mathbb N}
\newcommand{\Z}{\mathbb Z}
\newcommand{\R}{\mathbb R}
\newcommand{\Q}{\mathbb Q}

\renewcommand{\P}{\mathbb P}

\newcommand{\LK}{\,[\![}
\newcommand{\RK}{]\!]}

 \DeclareMathOperator{\ord}{ord}

 \DeclareMathOperator{\supp}{supp}

\renewcommand{\time}{\negthinspace \times \negthinspace}

\newcommand{\red}{{\text{\rm red}}}
\renewcommand{\t}{\, | \,}

\numberwithin{equation}{section}

\begin{document}

\title{A characterization of class groups via sets of lengths}

\author{Alfred Geroldinger  and Wolfgang A. Schmid}

\thanks{This work was supported by
the Austrian Science Fund FWF, Project Number P26036-N26, by the Austrian-French Amad{\'e}e Program FR03/2012, and  by the ANR Project Caesar, Project Number ANR-12-BS01-0011.}

\keywords{Krull monoids,    class groups,   arithmetical characterizations, sets of lengths, zero-sum sequences}

\subjclass[2010]{11B30, 11R27, 13A05, 13F05, 20M13}

\begin{abstract}
Let $H$ be a Krull monoid  with class group $G$ such that every class contains a prime divisor. Then every nonunit $a \in H$ can be written as a finite product of irreducible elements. If $a=u_1 \cdot \ldots \cdot u_k$, with irreducibles $u_1, \ldots,  u_k \in H$, then $k$ is called the length of the factorization and the set $\mathsf L (a)$ of all possible $k$ is  the set of lengths of $a$. It is well-known that the system $\mathcal L (H) = \{\mathsf L (a) \mid a \in H \}$ depends only on the class group $G$. We study the inverse question asking whether  the system $\mathcal L (H)$ is characteristic for the class group. Let $H'$ be a further Krull monoid  with class group $G'$ such that every class contains a prime divisor and suppose that $\mathcal L (H) = \mathcal L (H')$. We show that, if one of the groups $G$ and $G'$ is finite and has rank at most two, then $G$ and $G'$ are isomorphic (apart from two well-known exceptions).
\end{abstract}

\maketitle


\section{Introduction} \label{1}

Let $H$ be a  cancellative semigroup with unit element. If an element $a \in H$ can be written as a product of $k$ irreducible elements, say $a = u_1 \cdot \ldots \cdot u_k$, then $k$ is called the length of the factorization. The set $\mathsf L (a)$ of all possible factorization lengths is  the set of lengths of $a$, and $\mathcal L (H) = \{ \mathsf L (a) \mid a \in H \}$ is called the  system of sets of lengths of $H$. Clearly, if $H$ is factorial, then $|\mathsf L (a)|=1$ for each  $a \in H$. Suppose there is some $a \in H$ with $|\mathsf L (a)| > 1$, say $k, l \in \mathsf L (a)$ with $k < l $. Then, for every $m \in \N$, we observe that $\mathsf L (a^m) \supset \{km + \nu (l-k) \mid \nu \in [0,m] \}$ which shows that sets of lengths can become arbitrarily large. Under mild conditions on the ideal theory of $H$ every nonunit of $H$, has a factorization into irreducibles and  all sets of lengths are finite.

Sets of lengths (together with parameters controlling their structure) are the most investigated invariants in factorization theory. They occur in settings ranging from numerical monoids,  noetherian domains, monoids of ideals and of modules to maximal orders in central simple algebras (for recent progress see \cite{Ba-Wi13a, C-F-G-O16, Ge16c}).
The focus of the present paper is on Krull monoids  with finite class group such that every class contains a prime divisor. Rings of integers in algebraic number fields are such Krull monoids, and classical notions from algebraic number theory (dating back to the 19th century) state that the class group determines the arithmetic of the ring of integers. This idea has been formalized and justified. In the 1970s Narkiewicz posed the inverse question whether or not arithmetical phenomena (in other words, phenomena describing the non-uniqueness of factorizations) characterize the class group (\cite[Problem 32; page 469]{Na74}). Very quickly first affirmative answers were given by Halter-Koch, Kaczorowski, and Rush (\cite{Ka81b, HK83a,Ru83}). Indeed, it is not too difficult to show that the system of sets of factorizations determines the class group (\cite[Sections 7.1 and 7.2]{Ge-HK06a}).

These answers are not really satisfactory because the given characterizations are based on rather abstract arithmetical properties  which play only  little role in other parts of factorization theory. Since, on the other hand, sets of lengths are of central interest in factorization theory, it is natural to ask whether their structure is rich enough to do characterizations.

Let $H$ be a commutative Krull monoid with finite class group $G$ and suppose that every class contains a prime divisor (recall that an integral domain is a Krull domain if and only if its monoid of nonzero elements is a Krull monoid). It is classical that  $H$ is factorial if and only if $|G|=1$, and by a result due to  Carlitz in 1960 we know that all sets of lengths are singletons (i.e., $|L|=1$ for all $L \in \mathcal L (H)$) if and only if $|G| \le 2$. Let us suppose now that $|G| \ge 3$.
Then the monoid $\mathcal B (G)$ of zero-sum sequences over $G$ is again a Krull monoid with class group isomorphic to $G$, every class contains a prime divisor, and  $\mathcal L (H) = \mathcal L \big(B (G) \big)$ (as usual, we set $\mathcal L (G) = \mathcal L \big( \mathcal B (G) \big)\big)$. The Characterization Problem can be formulated as follows (\cite[Section 7.3]{Ge-HK06a}, \cite[page 42]{Ge-Ru09}, \cite{Sc09b}).

\begin{enumerate}
\item[]
{\it Given two finite abelian groups $G$ and $G'$ such that $\mathcal L(G) = \mathcal L(G')$.
Does it follow that $G \cong G'$?}
\end{enumerate}

The system of sets of lengths $\mathcal L  (G)$ for finite abelian groups is studied with methods from Additive Combinatorics. Zero-sum theoretical invariants, such as the Davenport constant, play a central role. Recall that, although the precise value of the Davenport constant is well-known for $p$-groups and for groups of rank at most two (see Proposition \ref{2.3}), its precise value is unknown in general (even for groups of the form $G=C_n^3$). Thus it is not surprising that all answers to the Characterization Problem so far have been restricted to very special groups including cyclic groups,  elementary $2$-groups, and groups of the form $C_{n} \oplus C_{n}$ (\cite{Ga-Ge00, Sc09c}). Apart from two well-known (trivial) pairings, the answer is always positive. Starting from $C_n \oplus C_n$,  Zhong studied the Characterization Problem for groups of the
form $C_n^r$ in a series of papers (\cite{Ge-Zh17b, Zh18a, Zh18b}).
The goal of the present paper is to settle the Characterization Problem for groups of rank at most two. Here is our main result.

\begin{theorem} \label{1.1}
Let $G$ be an abelian group such that $\mathcal L (G) = \mathcal L ( C_{n_1}  \oplus C_{n_2})$ where $n_1, n_2 \in \N$ with $n_1 \t n_2$ and $n_1+n_2 > 4$. Then $G \cong C_{n_1} \oplus C_{n_2}$.
\end{theorem}

The difficulty of the Characterization Problem stems from the fact that most sets of lengths over any finite abelian group are arithmetical progressions with difference $1$ (see Proposition \ref{3.2}.3, or   \cite[Theorem 9.4.11]{Ge-HK06a}). If $G$ and $G'$ are finite abelian groups with $G \subset G'$, then clearly $\mathcal L (G) \subset \mathcal L (G')$. Thus, in order to characterize a group $G$, we first have to find  distinctive sets of lengths for $G$ (i.e., sets of lengths which do occur in $\mathcal L (G)$, but in no other or only in a small number of further groups), and second we will have to show that certain sets  are not sets of lengths in $\mathcal L (G)$. These distinctive sets of lengths for  rank two groups are identified in Proposition \ref{6.5} which is the core of our whole approach, and Proposition \ref{5.1} provides sets which do not occur as sets of lengths for rank two groups.  After gathering some background material in Section \ref{2}, we summarize key results on the structure of sets of lengths in Propositions \ref{3.2}. \ref{3.3}, and \ref{3.4}. Furthermore, we provide some explicit constructions which will turn out to be crucial (Propositions \ref{3.5} -- \ref{3.8}).  After that, we are well-prepared for the main parts given in Sections \ref{5} and \ref{6}.

\section{The arithmetic of Krull monoids: Background} \label{2}

We gather the required tools from the algebraic and arithmetic theory of Krull monoids. Our notation and terminology are consistent with the monographs \cite{Ge-HK06a, Ge-Ru09, Gr13a}.
Let $\N$ denote the set of positive integers, $\P \subset \N$ the set of prime numbers and put $\N_0 = \N \cup \{0\}$. For real numbers $a,\,b \in \R$, we set $[a, b ] = \{ x \in \Z \mid a \le x \le b \}$. Let  $A, B \subset \mathbb Z$ be subsets of the integers.  We denote by $A+B = \{a+b \mid a \in A, b \in B\}$ their {\it sumset}, and by $\Delta (A)$ the {\it set of $($successive$)$ distances} of $A$ (that is, $d \in \Delta (A)$ if and only if $d = b-a$ with $a, b \in A$ distinct and $[a, b] \cap A = \{a, b\}$). For $k \in \N$, we denote by $k \cdot A = \{k a \mid a \in A \}$ the {\it dilation} of $A$ by $k$. If $A \subset \N$, then the {\it elasticity} of $A$ is defined as
\[
\rho (A) = \sup \Big\{ \frac{m}{n} \mid m, n \in A \Big\} = \frac{\sup A}{\min A} \in \Q_{\ge 1} \cup \{\infty\}
\ \text{and we set} \ \rho ( \{0\})=1 \,.
\]

\noindent
{\bf Monoids and Factorizations.} By a  {\it monoid},  we
mean a commutative semigroup with identity which satisfies
the cancellation law (that is, if $a,b ,c$ are elements of the
monoid with $ab = ac$, then $b = c$ follows). The multiplicative
semigroup of non-zero elements of an integral domain is a monoid.
Let $H$ be a monoid. We denote by $H^{\times}$ the set of invertible
elements of $H$, by  $\mathcal A (H)$ the set of atoms (irreducible elements) of $H$, and by  $H_{\red} = H/H^{\times}= \{ a H^{\times} \mid a \in H \}$
the associated reduced monoid of $H$.
A monoid $F$ is  free abelian, with basis $P \subset
F$, and we write $F =\mathcal F (P)$  if every $a \in F$ has a unique representation of the form
\[
a = \prod_{p \in P} p^{\mathsf v_p(a), } \quad \text{where} \quad
\mathsf v_p(a) \in \N_0 \ \text{ with } \ \mathsf v_p(a) = 0 \ \text{
for almost all } \ p \in P \,,
\]
and we   call
\[
|a|_F = |a| = \sum_{p \in P} \mathsf v_p (a) \quad \text{the \ {\it
length}} \ \text{of} \ a \,.
\]
The  monoid  $\mathsf Z (H) = \mathcal F \bigl(
\mathcal A(H_\red)\bigr)$  is called the  {\it factorization
monoid}  of $H$, and  the  homomorphism $\pi \colon \mathsf Z (H) \to H_{\red}$, defined by
$\pi (u) = u \ \text{for each} \ u \in \mathcal A(H_\red)$  is  the  {\it factorization homomorphism}  of $H$. For $a \in H$,
\[
\begin{aligned}
\mathsf Z_H (a) = \mathsf Z (a)  & = \pi^{-1} (aH^\times) \subset
\mathsf Z (H) \quad
\text{is the \ {\it set of factorizations} \ of \ $a$} \,,  \quad \text{and}
\\
\mathsf L_H (a) = \mathsf L (a) & = \bigl\{ |z| \, \bigm| \, z \in
\mathsf Z (a) \bigr\} \subset \N_0 \quad \text{is the \ {\it set of
lengths} \ of $a$}  \,.
\end{aligned}
\]
Thus $H$ is factorial if and only if $H_{\red}$ is free abelian (equivalently, $|\mathsf Z (a)|=1$ for all $a \in H$). The monoid $H$ is called {\it atomic} if $\mathsf Z (a) \ne \emptyset$ for all $a \in H$ (equivalently, every nonunit can be written as a finite product of irreducible elements). For the remainder of this work, we suppose that $H$ is atomic. Note that, $\mathsf L(a) = \{0\}$ if and only if  $a \in H^{\times}$, and $\mathsf L (a) = \{1\}$ if and only if $a \in \mathcal A (H)$.
We denote by
\[
\begin{aligned}
\mathcal L (H) & = \left\{ \mathsf L (a) \mid a \in H \right\} \qquad \text{the {\it system of sets of lengths} of} \ H \,, \quad \text{and by} \\
\Delta (H) & = \bigcup_{L \in \mathcal L (H)} \Delta ( L ) \ \subset
\N \qquad \text{the {\it set of distances} of} \ H \,.
\end{aligned}
\]
For $k \in \N$, we set $\rho_k (H) = k$ if $H=H^{\times}$, and
\[
\rho_k (H) = \sup \{ \sup L \mid L \in \mathcal L (H), k \in L \} \in \N \cup \{\infty\}, \quad \text{if} \quad H \ne H^{\times} \,.
\]
Then
\[
\rho (H) = \sup \{ \rho (L) \mid L \in \mathcal L (H) \} = \lim_{k \to \infty} \frac{\rho_k (H)}{k} \in \R_{\ge 1} \cup \{\infty\}
\]
is the {\it elasticity} of $H$. The monoid $H$ is said to be
\begin{itemize}
\item {\it half-factorial} if $\Delta (H) = \emptyset$. If $H$ is not half-factorial, then $\min \Delta (H) = \gcd \Delta (H)$.

\item {\it decomposable} if there exist submonoids $H_1, H_2$ with $H_i \not\subset H^{\times}$ for $i \in [1,2]$ such that $H = H_1 \times H_2$ (otherwise $H$ is called {\it indecomposable}).
\end{itemize}

For a free abelian monoid $\mathcal F (P)$, we introduce a distance function $\mathsf d \colon \mathcal F (P) \times \mathcal F (P) \to \N_0$, by setting
\[
\mathsf d (a,b) = \max \Big\{ \Big|\frac{a}{\gcd(a,b)}\Big|, \Big|\frac{b}{\gcd(a,b)}\Big| \Big\} \in \N_0 \quad \text{for } \ a, b \in \mathcal F (P) \,,
\]
and we note that $\mathsf d (a,b) = 0$ if and only if $a=b$. For a subset $\Omega \subset \mathcal F (P)$, we define the {\it catenary degree} $\mathsf c (\Omega)$ as the smallest $N \in \N_0 \cup \{\infty\}$ with the following property: for each $a, b \in \Omega$, there are elements $a_0, \ldots a_k \in \Omega$ such that $a=a_0, a_k=b$, and $\mathsf d (a_{i-1}, a_i) \le N$ for all $i \in [1,k]$. Note that $\mathsf c (\Omega) = 0$ if and only if $|\Omega| \le 1$.
For an element $a \in H$, we call $\mathsf c_H (a) = \mathsf c (a) = \mathsf c ( \mathsf Z_H (a))$ the catenary degree of $a$, and
$
\mathsf c (H) = \sup \{ \mathsf c (a) \mid a \in H \} \in \N_0 \cup \{\infty\}$
is the {\it catenary degree} of $H$. The monoid $H$ is factorial if and only if $\mathsf c (H) = 0$, and if $H$ is not factorial, then $2 + \sup \Delta (H) \le \mathsf c (H)$ (\cite[Theorem 1.6.3]{Ge-HK06a}).

\smallskip
\noindent
{\bf Krull monoids.} A monoid $H$ is {\it Krull} if it is completely integrally closed and satisfies the ascending chain condition on divisorial ideals. An integral domain $R$ is a Krull domain if and only if its
multiplicative monoid $R \setminus \{0\}$ is a Krull monoid, and this generalizes to Marot rings (\cite{Ge-Ra-Re15c}). The theory of Krull monoids is presented in detail in
 \cite{HK98, Ge-HK06a}, and for a survey  we refer to \cite{Ge16c}.

Much of the arithmetic of a Krull monoid can be seen in an associated monoid of zero-sum sequences. This is a Krull monoid again which can be studied with methods from Additive Combinatorics. To introduce the necessary concepts, let $G$ be an additively written abelian group, $G_0 \subset G$ a subset, and let $\mathcal F (G_0)$ be the free abelian monoid with basis $G_0$. In Additive Combinatorics, the elements of $\mathcal F(G_0)$ are called \ {\it sequences} over \
$G_0$. If $S = g_1 \cdot \ldots \cdot g_l \in \mathcal F (G_0)$, where $l \in \N_0$ and $g_1, \ldots , g_l \in G_0$, then $\sigma (S) = g_1+ \ldots + g_l$ is called the sum of $S$, and the monoid
\[
\mathcal B(G_0) = \{ S  \in \mathcal F(G_0) \mid \sigma (S) =0\} \subset \mathcal F (G_0)
\]
is called the  {\it monoid of zero-sum sequences}  over  $G_0$ (these objects are also referred to in the literature as block monoids).
The embedding $\mathcal B (G_0) \hookrightarrow \mathcal F (G_0)$ is a divisor homomorphism and $\mathcal B (G_0)$ is a Krull monoid. The monoid $\mathcal B (G)$ is factorial if and only if $|G|\le 2$. If $|G| \ne 2$, then $\mathcal B (G)$ is a Krull monoid with class group isomorphic to $G$ and every class contains precisely one prime divisor.  For every arithmetical invariant \ $*(H)$ \ defined for a monoid
$H$, it is usual to  write $*(G_0)$ instead of $*(\mathcal B(G_0))$. In particular, we set  $\mathcal A (G_0) = \mathcal
A (\mathcal B (G_0))$ and $\mathcal L (G_0) = \mathcal L ( \mathcal B
(G_0) )$. Similarly, arithmetical properties of $\mathcal B (G_0)$ are attributed to $G_0$. Thus,  $G_0$ is said to be
\begin{itemize}
\item {\it $($in$)$decomposable} \ if $\mathcal B (G_0)$ is (in)decomposable,

\item {\it$($non-$)$ half-factorial} \  if $\mathcal B (G_0)$ is (non-)half-factorial.
\end{itemize}

\begin{proposition} \label{2.1}
Let $H$ be a Krull monoid with  class group $G$, and suppose   that  each class contains a prime divisor. Then there is a transfer homomorphism $\boldsymbol \beta  \colon H \to \mathcal B (G)$ such that the following hold.
\begin{enumerate}
\item $\mathsf L_H (a) = \mathsf L_{\mathcal B (G)} \big( \boldsymbol \beta (a) \big)$ for each $a \in H$ and $\mathcal L (H) = \mathcal L (G)$.

\item If $|G| \ge 3$, then $\mathsf c (H) = \mathsf c \big( \mathcal B (G) \big)$.
\end{enumerate}
\end{proposition}

\begin{proof}
See \cite[Section 3.4]{Ge-HK06a}.
\end{proof}

The above result generalizes to  {\it transfer Krull monoids} $H$ over  abelian groups $G$, which also satisfy the relationship $\mathcal L (H)=\mathcal L (G)$.  Hence all characterization results, such as Theorem \ref{1.1}, also apply to them (\cite{Ge16c}).

\smallskip
\noindent {\bf Zero-Sum Theory.} Let $G$ be an additive abelian group,  $G_0 \subset G$ a subset, and $G_0^{\bullet} = G_0 \setminus \{0\}$.  Then $[G_0] \subset G$ denotes the subsemigroup and $\langle G_0 \rangle \subset G$  the subgroup generated by $G_0$.  For a sequence
\[
S = g_1 \cdot \ldots \cdot g_l = \prod_{g \in G_0} g^{\mathsf v_g
(S)} \in \mathcal F (G_0) \,,
\]
we set $\varphi (S) = \varphi (g_1) \cdot \ldots \cdot \varphi
(g_l)$ for any homomorphism $\varphi \colon G \to G'$, and in
particular, we have $-S = (-g_1) \cdot \ldots \cdot (-g_l)$. We call
\[
\begin{aligned}
\supp (S)&  = \{g \in
G \mid \mathsf v_g (S) > 0 \} \subset G \ \text{the \ {\it support}
\ of \
$S$} \,, \quad \mathsf v_g(S) \ \text{ the {\it multiplicity} of $g$ in $S$} \,, \\
|S| & = l = \sum_{g \in G} \mathsf v_g (S) \in \mathbb N_0 \
\text{the \ {\it length} \ of \ $S$} \,,  \quad \text{and} \quad \mathsf k (S) = \sum_{g \in G} \frac{1}{\ord (g)} \in \Q \ \text{the \ {\it cross number }
\ of \
$S$} \,.
\end{aligned}
\]
Moreover, $\Sigma (S) = \Big\{ \sum_{i \in I} g_i
\mid \emptyset \ne I \subset [1,l] \Big\} \ \text{is the \ {\it set of
subsequence sums} \ of \ $S$}$. The sequence $S$ is said to be
\begin{itemize}
\item {\it zero-sum free} \ if \ $0 \notin \Sigma (S)$,

\item a {\it zero-sum sequence} \ if \ $\sigma (S) = 0$,

\item a {\it minimal zero-sum sequence} \ if it is a nontrivial zero-sum
      sequence and every proper  subsequence is zero-sum free.
\end{itemize}
Both Davenport constants, namely
\begin{itemize}
\item the (small) {\it Davenport constant} \ $\mathsf d (G_0) = \sup \bigl\{ |S| \, \bigm| \; S \in \mathcal F (G_0) \ \text{is zero-sum free} \bigr\} \in \N_0 \cup \{\infty\}$ and

\item the (large) {\it Davenport constant} \ $\mathsf  D (G_0) = \sup \bigl\{ |U| \, \bigm| \; U \in \mathcal A (G_0) \bigr\} \in \N_0 \cup \{\infty\}$
\end{itemize}
are classical invariants in zero-sum theory.    For $n \in \mathbb N$, let $C_n$ denote a cyclic group with $n$ elements. Suppose that $G$ is finite. A tuple $(e_i)_{i \in I}$ is called a {\it basis} of $G$ if all elements are nonzero and $G = \oplus_{i \in I} \langle e_i \rangle$. For $p \in \P$, let $\mathsf r_p (G)$ denote the $p$-rank of
$G$, $\mathsf r (G) = \max \{ \mathsf r_p (G) \mid p \in \P \}$ denote the
{\it rank} of $G$, and let $\mathsf r^* (G) = \sum_{p \in \P} \mathsf r_p
(G)$ be the {\it total rank} of $G$. If  $|G| > 1$, then
$G \cong C_{n_1} \oplus \ldots \oplus C_{n_r}  \,, \ \text{and we
set} \quad \mathsf d^* (G) = \sum_{i=1}^r (n_i-1)   \,,
$
where  $r, n_1, \ldots , n_r \in
\mathbb N$  with  $1 < n_1 \mid \ldots \mid n_r$, $r = \mathsf r (G)$, and
$n_r = \exp (G)$ is the exponent of $G$.
If $|G| = 1$, then $\mathsf r (G) = \mathsf r^* (G) =
0$, $\exp (G) = 1$, and $\mathsf d^* (G)  = 0$.
We will use the following well-known results  (see  \cite[Chapter 5]{Ge-HK06a}).

\begin{proposition} \label{2.3}
Let $G$ be a finite abelian group.
Then $1+\mathsf d^* (G) \le 1+\mathsf d (G)=\mathsf D (G) \le |G|$. If $G$ is a $p$-group or $\mathsf r (G) \le 2$, then $\mathsf d (G)=\mathsf d^* (G)$.
\end{proposition}

We will make substantial use of the following result (see \cite[Theorem 6.4.7]{Ge-HK06a} and \cite[Theorem 1.1]{Ge-Zh15b}).

\begin{proposition} \label{2.4}
Let $H$ be a Krull monoid with finite class group $G$ where $|G| \ge 3$ and every class contains a prime divisor. Then $\mathsf c (H) \in [3, \mathsf D (G)]$, and we have
\begin{enumerate}
\item $\mathsf c (H) = \mathsf D (G)$ if and only if $G$ is either cyclic or an elementary $2$-group.

\item $\mathsf c (H) = \mathsf D (G)-1$ if and only if $G$ is isomorphic either to $C_2^{r-1} \oplus C_4$ for some $r \ge 2$ or  to $C_2 \oplus C_{2n}$ for some $n \ge 2$.
\end{enumerate}
\end{proposition}

Let $A \in \mathcal B (G_0)$ and $d = \min \{ |U| \mid U \in \mathcal A (G_0) \}$. If $A = BC$ with $B, C \in \mathcal B (G_0)$, then
\[
\mathsf L (B) + \mathsf L (C) \subset \mathsf L (A) \,.
\]
If $A = U_1 \cdot \ldots \cdot U_k = V_1 \cdot \ldots \cdot V_l$ with $U_1, \ldots, U_k, V_1, \ldots, V_l \in \mathcal A (G_0)$ and $k < l$, then
\[
l d \le \sum_{\nu=1}^l |V_{\nu}| = |A| = \sum_{\nu=1}^k |U_{\nu}| \le k \mathsf D (G_0) \,,
\ \text{whence}
\
\frac{|A|}{\mathsf D (G_0)} \le \min \mathsf L (A) \le \max \mathsf L (A) \le \frac{|A|}{d} \,.
\]

We  need the concept of relative block monoids (as introduced by  Halter-Koch in \cite{HK92e}, and recently studied by  Baeth et al. in \cite{Ba-Ho09a}).
Let $G$ be an abelian group. For a subgroup $K \subset G$ let
\[
\mathcal B_K (G) = \{ S \in \mathcal F (G) \mid \sigma (S) \in K \}
\subset \mathcal F (G) \,,
\]
and let $\mathsf D_K (G)$ denote the smallest $l \in \N \cup \{\infty\}$ with the following property:

      \smallskip
      \begin{itemize}
      \item Every sequence $S \in \mathcal F (G)$ of length $|S| \ge l$
            has a  subsequence $T$ with $\sigma (T) \in K$.
      \end{itemize}
Clearly, $\mathcal B_K (G) \subset \mathcal F (G)$ is a submonoid
with
\[
\mathcal B (G) = \mathcal B_{\{0\}} (G) \subset \mathcal B_K (G)
\subset \mathcal B_G (G) = \mathcal F (G)
\]
and $\mathsf D_{\{0\}} (G) = \mathsf D (G)$. The following result is well-known (\cite[Theorem 2.2]{Ba-Ho09a}). Since there seems to be no proof  in the literature and we make substantial use of it, we provide the simple arguments.

\begin{proposition} \label{2.5}
Let $G$ be an abelian group and $K \subset G$  a subgroup.
\begin{enumerate}
\item $\mathcal B_K(G)$ is a Krull monoid. If $|G|=2$ and $K=\{0\}$, then $\mathcal B_K(G)= \mathcal B (G)$ is factorial. In all other cases the embedding $\mathcal B_K (G) \hookrightarrow \mathcal F
      (G)$ is a divisor theory with class group isomorphic to $G/K$ and every class contains precisely $|K|$ prime divisors.

\item The monoid homomorphism $\theta \colon \mathcal B_K (G) \to
      \mathcal B (G/K)$, defined by  $\theta ( g_1 \cdot \ldots \cdot g_l)
      =(g_1+K) \cdot \ldots \cdot (g_l+K)$ is a transfer homomorphism. If $|G/K|\ge 3$, then $\mathsf c \big( \mathcal B_K (G) \big) = \mathsf c \big( \mathcal B (G/K)  \big)$.

\item $\mathsf D_K (G) = \sup \{|U| \mid U \ \text{is an atom of} \
      \mathcal B_K (G) \} = \mathsf D (G/K)$.
\end{enumerate}
\end{proposition}

\begin{proof}
1. If $|G|=1$, then $\mathcal B (G)=\mathcal F (G)$ is factorial, the class group is trivial, and there is precisely one prime divisor. If $|G|=|K|=2$, then $\mathcal B_K (G)=\mathcal F (G)$ is factorial, the class group is trivial, and there are precisely two prime divisors. If  $|G|=2$ and $|K|=1$, then $\mathcal B_K (G) = \mathcal B
(G)$ is factorial, and hence a Krull monoid with trivial  class group. Suppose that
$|G| \ge 3$. Clearly, the embedding $\mathcal B_K (G) \hookrightarrow \mathcal F (G)$ is a divisor homomorphism.  To verify that it is a divisor theory,
let $g \in G$ be given. If $\ord (g) = n \ge 3$, then $g = \gcd \bigl( g^n, g(-g) \bigr)$. If $\ord (g) = 2$, then there is an element $h \in G \setminus \{0,g\}$ and $g = \gcd \bigl( g^2, g h (g-h) \bigr)$. If $\ord (g) = \infty$, then $g = \gcd \bigl( (-g)g, g (2g) (-3g) \bigr)$.

The map $\varphi \colon \mathcal F (G) \to G/K$, defined by $\varphi (S) = \sigma (S) + K$ for every $S \in \mathcal F (G)$, is  a monoid epimorphism. If $S, S' \in \mathcal F (G)$, then $\sigma (S) + K = \sigma (S') + K$ if and only if $S' \in [S] = S \mathsf q \bigl( \mathcal B_K (G) \bigr)$. Thus $\varphi$ induces a group isomorphism $\Phi \colon \mathsf q (\mathcal F (G))/ \mathsf q (\mathcal B_K (G)) \to G/K$, defined by $\Phi ([S]) = \sigma (S) + K$, and we have $[S] \cap G = \sigma (S) + K$. Thus the class $[S]$ contains precisely $|K|$ prime divisors.

2. This follows from 1.,  from Proposition \ref{2.1}, and from the reference given there.

3. Since a sequence $S \in \mathcal F (G)$ is an atom of $\mathcal B_K (G)$ if and only if $S \ne 1$, $\sigma (S) \in K$ and $\sigma (T) \notin K$ for all proper subsequences $T$ of $S$, it follows that $\mathsf D_K (G) = \sup \{|U| \mid U \ \text{is an atom of} \ \mathcal B_K (G) \}$. Since $\theta$ is a transfer homomorphism, we get $\theta \bigl( \mathcal A (\mathcal B_K (G)) \bigr) = \mathcal A (G/K)$ and $\theta^{-1} \bigl( \mathcal A (G/K) \bigr) =  \mathcal A (\mathcal B_K (G)) $. Therefore $|U| = |\theta (U)|$ for all $U \in \mathcal B_K (G)$, and it follows that
\[
\sup \{ |U| \mid U \in \mathcal A (\mathcal B_K (G)) \} = \sup \{ |V| \mid V \in \mathcal A (G/K) \} = \mathsf D(G/K) \,. \qedhere
\]
\end{proof}

\section{Structural results on $\mathcal L (G)$ and first basic constructions} \label{3}

Let $G$ be an abelian group.  If $G$ is infinite, then every finite subset $L \subset \N_{\ge 2}$ is contained in $\mathcal L (G)$ (\cite[Theorem 7.4.1]{Ge-HK06a}). If $G$ is finite, then sets of lengths have a well-studied structure. In order to describe it, we recall the concept of an AAMP.
Let $d \in \N$, \ $l,\, M \in \N_0$ \ and \ $\{0,d\} \subset
\mathcal D \subset [0,d]$. A subset $L \subset \Z$ is called an
{\it almost arithmetical multiprogression} \ ({\rm AAMP} \ for
      short) \ with \ {\it difference} \ $d$, \ {\it period} \ $\mathcal D$,
      \ {\it length} \ $l$ and \ {\it bound} \ $M$, \ if
      \begin{equation} \label{AAMP}
      L = y + (L' \cup L^* \cup L'') \, \subset \, y + \mathcal D + d \Z
      \end{equation}
      where $\min L^* = 0$, $L^*$ is an interval of \ $\mathcal D + d \Z$ \ (this means that $L^*$ is finite nonempty and $L^* = (\mathcal
D + d \mathbb Z) \cap [0, \max L^*]$), $l$ \ is maximal such that \ $l d \in L^*$, \ $L' \subset [-M, -1]$, \ $L'' \subset \max L^* + [1,M]$ \ and \
      $y \in \Z$.
The {\it set of minimal distances} $\Delta^* (G) \subset \Delta (G)$ is  defined as
\[
\Delta^* (G) = \{\min \Delta (G_0) \mid G_0 \subset G \ \text{with} \ \Delta (G_0) \ne \emptyset \} \subset \Delta (G) \,.
\]

\begin{proposition}[Structural results on $\mathcal L (G)$] \label{3.2}~
Let $G$ be a finite abelian group with $|G|\ge 3$.
\begin{enumerate}
\item There exists some $M \in \N_0$  such that every set of lengths  $L \in \mathcal L(G)$ is an {\rm AAMP} with some difference $d \in \Delta^* (G)$ and bound $M$.

\item For every $M \in \mathbb N_0$ and every finite nonempty set $\Delta^* \subset \mathbb N$, there is a finite abelian group $G^*$ such that  for every {\rm AAMP} $L$ with difference $d\in \Delta^*$ and bound $M$ there is a $y^{\phantom{g}}_{\! L} \in \mathbb N$ with
      \[
      y+L \in \mathcal L (G^*) \quad \text{ for all } \quad y \ge y^{\phantom{g}}_{\! L}\,.
      \]

\item If $A \in \mathcal B (G)$ such that $\supp (A) \cup \{0\}$ is a subgroup of $G$, then $\mathsf L (A)$ is an arithmetical progression with difference $1$.
\end{enumerate}
\end{proposition}

\begin{proof}
We refer to \cite[Theorems 4.4.11 and 7.6.8]{Ge-HK06a}) and to \cite{Sc09a}.
\end{proof}

\begin{proposition}[Structural results on $\Delta (G)$ and on $\Delta^* (G)$] \label{3.3}~

Let $G = C_{n_1} \oplus \ldots \oplus C_{n_r}$ where $r, n_1, \ldots, n_r \in \N$ with $r=\mathsf r (G)$, $1 < n_1 \t \ldots \t n_r$, and $|G|\ge 3$.
\begin{enumerate}
\item $\Delta (G)$ is an interval with
      \[
      \bigl[1,\, \max \{\exp (G) -2, \, k-1 \} \bigr] \subset \Delta (G) \subset \bigl[1, \mathsf D (G) - 2\bigr] \ \text{ where } \  k =\sum_{i=1}^{\mathsf r (G)} \Bigl\lfloor \frac{n_i}{2} \Bigr\rfloor  \,.
      \]

\item $1 \in \Delta^* (G) \subset \Delta (G)$,  $[1, \mathsf r (G)-1] \subset \Delta^* (G)$, and $\max \Delta^* (G) = \max \{\exp (G)-2, \mathsf r (G)-1 \}$.

\item If $G$ is cyclic of order $|G| = n \ge 4$, then $\max \big( \Delta^* (G) \setminus \{n-2\}\big) = \lfloor \frac{n}{2} \rfloor - 1$.
\end{enumerate}
\end{proposition}

\begin{proof}
We refer to   \cite[Section 6.8]{Ge-HK06a}, to \cite{Ge-Yu12b}, and to \cite{Ge-Zh16a}.
\end{proof}

\begin{proposition}[Results on $\rho_k (G)$ and on $\rho (G)$] \label{3.4}~

Let $G$ be a finite abelian group with $|G| \ge 3$, and let $k \in \N$.
\begin{enumerate}
\item $\rho (G) = \mathsf D (G)/2$ and $\rho_{2k} (G) = k \mathsf D (G)$.

\item $1 + k \mathsf D (G) \le \rho_{2k+1} (G) \le k \mathsf D (G) + \mathsf D (G)/2$. If $G$ is cyclic, then equality holds on the left side.
\end{enumerate}
\end{proposition}

\begin{proof}
We refer to \cite[Chapter 6.3]{Ge-HK06a}, \cite[Theorem 5.3.1]{Ge09a}, and to \cite{Ge-Gr-Yu15} for recent progress.
\end{proof}

In the next propositions we provide examples of sets of lengths over cyclic groups, over groups of rank two, and over groups of the form $C_2^{r-1} \oplus C_n$ with $r, n \in \N_{\ge 2}$. All examples will have difference $d=\max \Delta^* (G)$ and period $\mathcal D$ with $\{0,d\} \subset \mathcal D \subset [0,d]$ and $|\mathcal D|=3$, and we write them down in a form used in Equation \eqref{AAMP} in order to highlight their periods. It will be crucial for our approach (see Proposition \ref{6.5}) that the sets given in Proposition \ref{3.5}.2 do not occur over cyclic groups (Proposition \ref{3.6}). It is well-known that sets of lengths over cyclic groups and over elementary $2$-groups have many features in common, and this carries over to rank two groups and groups of the form $C_2^{r-1} \oplus C_n$  (see Propositions \ref{3.5}.2, \ref{3.7}.2, and \ref{6.5}). For a set $L \in \mathcal L (G)$ there is a $B \in \mathcal B (G)$ such that $L=\mathsf L (B)$ and hence $m+L = \mathsf L (0^mB) \in \mathcal L (G)$ for all $m \in \N_0$. Therefore the interesting sets of lengths $L \in \mathcal L (G)$ are those which do not stem from such a shift. These are those sets $L \in \mathcal L (G)$ with $-m+L \notin \mathcal L (G)$ for every $m \in \N$.

\begin{proposition} \label{3.5}
Let $G = C_{n_1} \oplus C_{n_2}$ where $n_1, n_2 \in \N$ with $2 < n_1 \t n_2$, and let $d \in [3, n_1]$.
\begin{enumerate}
\item For each $k \in \N$, we have
      \[
      \begin{aligned}
      (2k+2) & + \{0, d-2, n_2-2 \} + \{\nu (n_2-2) \mid \nu \in [0,k-1] \} \cup \{(kn+2)+(d-2)\} = \\
      (2k+2) & + \{0, d-2 \} + \{\nu (n_2-2) \mid \nu \in [0,k] \}  \in \mathcal L (G) \,.
      \end{aligned}
      \]

\item For each $k \in \N$, we have
      \[
      \Big((2k+3) + \{0, n_1-2, n_2-2 \} + \{\nu (n_2-2) \mid \nu \in [0,k] \}  \Big) \cup \Big\{(kn_2+3) + (n_1-2)+(n_2-2) \Big\}   \in \mathcal L (G) \,.
      \]
\end{enumerate}
\end{proposition}

\begin{proof}
Let $(e_1, e_2)$ be a basis of $G$ with $\ord (e_i) = n_i$ for $i \in [1,2]$, and let $k \in \N$. For $i \in [1,2]$, we set
$U_i = e_i^{n_i}$ and $V_i = (-e_i)e_i$. Then
\[
(-U_i)^k U_i^k = (-U_i)^{k - \nu} U_i^{k-\nu} V_i^{\nu n_i} \quad \text{for all} \quad \nu \in [0,k] \,,
\]
and hence
\[
\mathsf L \big( (-U_i)^k U_i^k \big) =  2k+ \{\nu(n_i-2) \mid \nu \in [0,k] \} \,.
\]

1. We set $h = (d-1)e_1$, $W_1 = (-e_1)^{d-1}h$, and $W_2= e_1^{n_1-(d-1)}h$. Then $\mathsf Z (U_1W_1)=\{ U_1W_1, V_1^{d-1}W_2\}$ and $\mathsf L (U_1W_1)=\{2,d\}$. Therefore
\[
\begin{aligned}
\mathsf L \big( (-U_2)^k U_2^k U_1 W_1 \big) & = \mathsf L \big( (-U_2)^k U_2^k \big) + \mathsf L \big( U_1 W_1 \big) \\
 & = \{2k+\nu(n_2-2) \mid \nu \in [0,k] \} + \{2, d \} \\
 & = (2k+2) + \{\nu (n_2-2) \mid \nu \in [0,k] \} + \{0, d-2\} \,.
\end{aligned}
\]

2. We define
\[
\begin{aligned}
W_1 = e_1^{n_1-1}e_2^{n_2-1}(e_1+e_2), \ W_2 & = (-e_1)e_2^{n_2-1}(e_1+e_2), \ W_3 = e_1^{n_1-1}(-e_2)(e_1+e_2), \ W_4 = (-e_1)(-e_2)(e_1+e_2) \,, \\
\text{and} \qquad B_k & = W_1(-U_1)(-U_2)U_2^k (-U_2)^k \,.
\end{aligned}
\]
Then any factorization of $B_k$ is divisible by precisely one of $W_1, \ldots, W_4$, and we obtain that
\[
\begin{aligned}
B_k & = W_1(-U_1)(-U_2)U_2^k (-U_2)^k = W_2 V_1^{n_1-1}(-U_2) U_2^k (-U_2)^k \\
  & = W_3 (-U_1) V_2^{n_2-1} U_2^k (-U_2)^k  = W_4 V_1^{n_1-1} V_2^{n_2-1} U_2^k (-U_2)^k \,.
\end{aligned}
\]
Thus it follows  that
\[
\begin{aligned}
\mathsf L (B_k) & =  \{3, n_1+1,n_2+1,n_1+n_2-1\} + \mathsf L \big( U_2^k (-U_2)^k \big)  \\ & = (2k+3) + \{\nu (n_2-2) \mid \nu \in [0,k] \} \ \cup \\
    &  \qquad (2k+3)  + (n_1-2) + \{\nu (n_2-2) \mid \nu \in [0,k] \} \ \cup \\
    &  \qquad (2k+3)  + (n_2-2) + \{\nu (n_2-2) \mid \nu \in [0,k] \} \ \cup \\
    & \qquad   (2k+3)  +  (n_1-2) + (n_2-2) + \{\nu (n_2-2) \mid \nu \in [0,k] \} \,.
\end{aligned}
\]
Thus $\max \mathsf L (B_k) = (kn_2+3) + (n_1-2)+(n_2-2)$ and
\[
\mathsf L (B_k) =  \Big(
(2k+3) + \{0, n_1-2, n_2-2 \} + \{\nu (n_2-2) \mid \nu \in [0,k] \}  \Big) \cup \{\max \mathsf L (B_k)\}  \,. \qedhere
\]
\end{proof}

\begin{proposition} \label{3.6}
Let $G$ be a cyclic group of order $|G|=n \ge 4$, and let $d \in [3, n-1]$.
\begin{enumerate}
\item  For each $k \in \N_0$, we have
       \[
       (2k+2) + \{ 0 , d - 2 \} + \{\nu (n-2) \mid \nu \in [0,k] \}  \in \mathcal L (G)  \,.
       \]

\item For each $k \in \N_0$, we set
      \[
      L_k =   \Big(
      (2k+3) + \{0, d-2, n-2 \} + \{\nu (n-2) \mid \nu \in [0,k] \}  \Big) \cup \Big\{(kn+3) + (d-2)+(n-2) \Big\} \,.
      \]
      Then for each $k \in \N_0$ and each $m \in \N_0$, we have $-m+L_k \notin \mathcal L (G)$.
\end{enumerate}
\end{proposition}

\begin{proof}
Let $k \in \N_0$.

1. Let  $g \in G$ with $\ord (g) = n$, $U=g^n$, $V=(-g)g$, $W_1= \big( (d-1)g \big) (- g)^{d-1}$, $W_2= \big( (d-1)g \big) g^{n-(d-1)}$, and
$B_k = \big( (-U)U \big)^k U W_1$.
Then $\mathsf Z (UW_1) = \{ UW_1, W_2 V^{d-1}\}$ and $\mathsf L (UW_1)= \{2,d\}$. Since every factorization of $B_k$ is divisible either by $W_1$ or by $W_2$, it follows that
\[
\begin{aligned}
\mathsf L ( B_k ) & = \mathsf L \big( (-U)^k U^k \big) + \mathsf L \big( U W_1 \big)  = \{2k+\nu(n-2) \mid \nu \in [0,k] \} + \{2, d \} \\
 & = (2k+2) + \{\nu (n-2) \mid \nu \in [0,k] \} + \{0, d-2\} \,.
\end{aligned}
\]

2. Note that $\max L_k = (kn+3) + (d-2)+(n-2) = (k+1)n+(d-1)$. Assume to the contrary that there is a $B_k \in \mathcal B (G)$ such that $\mathsf L (B_k)=L_k$. Then $\min \mathsf L (B_k) = 2k+3$ and, by Proposition \ref{3.4},
\[
(k+1)n+(d-1) = \max \mathsf L (B_k) \le \rho_{2k+3} (G) = (k+1)n+1 \,,
\]
a contradiction. If $m \in \N_0$ and $B_{m,k} \in \mathcal B (G)$ such that $\mathsf L (B_{m,k})=-m+L_k$, then $\mathsf L (0^m B_{m,k})=L_k \in \mathcal L (G)$. Thus $-m+L_k \notin \mathcal L (G)$ for any $m \in \N_0$.
\end{proof}

\begin{proposition} \label{3.7}
Let $G = C_2^{r-1} \oplus C_n$ where $r, n \in \N_{\ge 2}$ and $n$ is even.
\begin{enumerate}
\item For each $k \in \N_0$, we have
      \[
      L_k = (2k+2) + \{0, n-2, n+r-3\} + \{\nu (n-2) \mid \nu \in [0,k] \} \in \mathcal L (G) \quad \text{yet} \quad L_k \notin \mathcal L (C_n) \,.
      \]

\item For each $k \in \N_0$, we have
      \[
      \Big(
      (2k+3) + \{0, r-1, n-2 \} + \{\nu (n-2) \mid \nu \in [0,k] \}  \Big) \cup \Big\{(kn+3) + (r-1)+(n-2) \Big\}   \in \mathcal L (G) \,.
       \]
\end{enumerate}
\end{proposition}

\begin{proof}
Let $k \in \N_0$, $(e_1, \ldots, e_{r-1}, e_r)$ be a basis of $G$ with $\ord (e_1)= \ldots = \ord (e_{r-1})=2$ and $\ord (e_r)=n$.  We set $e_0=e_1+ \ldots + e_{r-1}$, $U_i = e_i^{\ord (e_i)}$ for each $i \in [1, r]$, $U_0 = (e_0+e_r)(e_0-e_r)$, $V_r = (-e_r)e_r$,
\[
V = e_1 \cdot \ldots \cdot e_{r-1}(e_0+e_r)(-e_r) \,, \quad \text{and} \quad W =  e_1 \cdot \ldots \cdot e_{r-1}(e_0+e_r) e_r^{n-1} \,.
\]

1.  Obviously, $\mathsf L \big( (-W)W \big) = \{2, n, n+r-1\}$ and
\[
\begin{aligned}
\mathsf L \big( (-W)W (-U_r)^{k} U_r^{k} \big) & = \mathsf L \big( (-W)W \big) + \mathsf L \big( (-U_r)^{k} U_r^{k}  \big) \\
 & = \{2, n, n+r-1\} + \{2k+\nu (n-2) \mid \nu \in [0,k] \} \\
 & = (2k+2) + \{0, n-2, n+r-3\} + \{\nu (n-2) \mid \nu \in [0,k] \}
\end{aligned}
\]
Since $\min L_k = 2k+2$, $\max L_k = (k+1)n+r-1$, and $\rho_{2k+2}(C_n) = (k+1)n$ by  Proposition \ref{3.4}, $r \ge 2$ implies that $L_k \notin \mathcal L (C_n)$.

2. Let $L_k$ denote the set in the statement. We define
\[
B_k = U_0U_1 \cdot \ldots \cdot U_{r-1} (-U_r)^{k+1} U_r^{k+1}
\]
and assert that $\mathsf L (B_k)=L_k$. Let $z$ be a factorization of $B_k$. We distinguish two cases.

\noindent
CASE 1: \ $U_1 \t z$.

Then $U_0U_1 \cdot \ldots U_{r-1} \t z$ which implies that $z = U_0U_1 \cdot \ldots \cdot U_{r-1} \big( (-U_r)U_r \big)^{k+1 - \nu} V_r^{\nu n} $ for some $\nu \in [0, k+1]$ and hence
$|z| \in r + (2k+2) + \{ \nu (n-2) \mid \nu \in [0, k+1] \}$.

\noindent
CASE 2: \ $U_1 \nmid z$.

Then either $V \t z$ or $W \t z$. If $V \t z$, then $z = (-V)V  V_r^{n-1} \big( (-U_r)U_r \big)^{k - \nu} V_r^{\nu n} $ for some $\nu \in [0, k]$ and hence $|z| \in (n+1) + 2k + \{ \nu (n-2) \mid \nu \in [0, k] \}$.
If $W \t z$, then $z = (-W)W V_r \big( (-U_r)U_r \big)^{k - \nu} V_r^{\nu n} $ for some $\nu \in [0, k]$ and hence $|z| \in 3 + 2k + \{ \nu (n-2) \mid \nu \in [0, k] \}$.
Putting all together the assertion follows.
\end{proof}

\begin{proposition} \label{3.8}
Let $G$ be a finite abelian group, $g \in G$ with $\ord (g)=n\ge 5$, and  $B \in \mathcal{B}(G)$ such that  $\big( (-g)g \big)^{2n} \t B$.
Suppose $\mathsf{L}(B)$ is an {\rm AAMP} with period $\{0,d, n-2 \}$ for some $d \in [1, n-3] \setminus \{(n-2)/2\}$.
\begin{enumerate}
\item  If $S \in \mathcal A \big(\mathcal B_{\langle g \rangle}(G) \big)$ with $S \t B$, then  $\sigma(S) \in \{0,g,-g, (d+1)g, -(d+1)g \}$.

\item  If $S_1, S_2 \in \mathcal A \big(\mathcal B_{\langle g \rangle}(G) \big)$ with $S_1S_2 \t B$,  then $\sigma(S_i) \in \{0,g,-g\}$ for at least one $i \in [1,2]$.
\end{enumerate}
\end{proposition}

\begin{proof}
By definition, there is a $y \in \Z$ such that $\mathsf L (B) \subset y + \{0,d,n-2\} + (n-2) \Z$. We set $U=g^n$ and $V=(-g)g$.

1. Let $S \in \mathcal A \big(\mathcal B_{\langle g \rangle}(G) \big)$ with $S \t B$ and set $\sigma(S)= kg$ with $k \in [0,n-1]$. If $k \in \{0,1,n-1\}$, then we are done. Suppose that $k \in [2, n-2]$. Since $S$ is an atom in $\mathcal B_{\langle g \rangle}(G)$, it follows that  $W_1=S(-g)^{k} \in \mathcal{A}(G)$ and $W_1'= S g^{n-k} \in \mathcal A (G)$. We consider a factorization $z \in \mathsf{Z}(B)$ with $U W_1 \t z$, say $z= U W_1 y$.
Then  $z' = W_1' V^k y$ is a factorization of $B$ of length $|z'|= |z| + k-1$.  Since $\mathsf{L}(B)$ is an AAMP with period $\{0,d, n-2 \}$ for some $d \in [1, n-3] \setminus \{(n-2)/2\}$
it follows that $k-1 \in \{d, n-2 -d\}$.

2. Let $S_1, S_2 \in \mathcal A \big(\mathcal B_{\langle e \rangle}(G) \big)$ with $S_1S_2 \t B$, and assume to the contrary $\sigma(S_i)= k_i e$ with $k_i \in [2, n-2]$ for each $i \in [1,2]$. As in 1. it follows that
\[
W_1=S_1(-g)^{k_1} , \ W_1'= S_1 g^{n-k_1} , \ W_2=S_2(-g)^{k_2} , \quad \text{and} \quad  W_2'= S_2 g^{n-k_2}
\]
are in $\mathcal A (G)$. We consider a factorization $z \in \mathsf Z (B)$ with $UW_1U W_2 \t z$, say $z=UW_1UW_2y$. Then $z_1=W_1'V^{k_1-1}UW_2y \in \mathsf Z (B)$ with $|z_1|=|z|+k_1-1$ and hence $k_1-1\in \{d, n-2-d\}$. Similarly, $z_2=UW_1W_2'V^{k_2-1}y \in \mathsf Z (B)$, hence $k_2-1\in \{d, n-2-d\}$, and furthermore it follows that $k_1=k_2$. Now $z_3=W_1'V^{k_1-1}W_2'V^{k_2-1}y \in \mathsf Z (B)$ is a factorization of length $|z_3|=|z|+k_1+k_2-2$.
Thus, if $k_1 -1 =d$, then $2d \in  \{n - 2, n -2 + d\}$, a contradiction, and if $k_1 -1 =n -2 - d$, then $2(n -2 - d) \in  \{n - 2, n -2 + (n -2 - d)\}$, a contradiction.
\end{proof}

\section{A set of lengths not contained in  $\mathcal L (C_{n_1} \oplus C_{n_2})$} \label{5}

The aim of this section is to prove the following proposition.

\begin{proposition} \label{5.1}
Let $G = C_{n_1} \oplus C_{n_2}$ where $n_1, n_2 \in \N$ with  $n_1 \t n_2$ and $6 \le n_1 < n_2$. \newline Then $\{2, n_2, n_1+n_2-2\} \notin \mathcal L (G)$.
\end{proposition}

 Let $G = C_{n_1} \oplus C_{n_2}$ where $n_1, n_2 \in \N$ with $n_1 \t n_2$.  If $6 \le n_1 < n_2$ does not hold, then $\{2, n_2, n_1+n_2-2\}$ may or may not be a set of lengths (e.g., if $2 = n_1 \le n_2$, then $\{2, n_2 \} \in \mathcal L (G)$). By Proposition \ref{3.7}, $\{2, n_2, n_1+n_2-2\} \in \mathcal L (C_2^{n_1-2} \oplus C_{n_2})$, whence Proposition \ref{5.1} implies that $\mathcal L (C_{n_1} \oplus C_{n_2}) \ne \mathcal L (C_2^{n_1-2} \oplus C_{n_2})$.
Its proof is based on the characterization of all minimal zero-sum sequences of maximal length over groups of rank two. This characterization is due to Gao, Grynkiewicz, Reiher, and the present authors (\cite{Ga-Ge03b, Ga-Ge-Gr10a, Sc10b, Re10c}). We repeat the formulation given in \cite[Theorem 3.1]{B-G-G-P13a} and then derive a corollary.

\begin{lemma} \label{structure}
Let  $G = C_{n_1} \oplus C_{n_2}$ where $n_1, n_2 \in \N$ with $1 < n_1 \t n_2$.  A sequence $U$
over $G$ of length $\mathsf D (G) = n_1+n_2-1$ is a minimal zero-sum
sequence if and only if it has one of the following two forms{\rm \,:}
\begin{itemize}
\smallskip
\item \[
      U = e_j^{\ord (e_j)-1} \prod_{\nu=1}^{\ord (e_i)} (x_{\nu}e_j+e_i) \,, \quad \text{where}
      \]
      \begin{itemize}
      \item[(a)] $\{i,j\}= \{1, 2\}$ and $(e_1, e_2)$ is a basis of $G$ with $\ord(e_1)= n_1$ and $\ord(e_2)= n_2$,
      \item[(b)] $x_1, \ldots, x_{\ord (e_i)}  \in [0, \ord (e_j)-1]$ and $x_1 + \ldots + x_{\ord (e_i)} \equiv 1 \mod \ord (e_j)$.
      \end{itemize}
      In this case, we say that $U$ is of type I with respect to the basis $(e_i, e_j)$;

\smallskip
\item \[
      U = (e_1+ye_2)^{s n_1 - 1} e_2^{n_2 - s n_1 +\epsilon}\prod_{\nu=1}^{n_1 -\epsilon} ( -x_{\nu} e_1 +(-x_{\nu}y+1) e_2) \,, \quad \text{where}
      \]
      \begin{itemize}
      \item[(a)] $(e_1, e_2)$ is a basis of $G$ with $\ord(e_1)=n_1$ and $\ord (e_2) = n_2$,
      \item[(b)] $y\in [0,n_2-1]$, \ $\epsilon\in [1,n_1-1]$,  and $s \in [1, n_2/n_1 -1]$,
      \item[(c)] $x_1, \ldots, x_{n_1-\epsilon} \in [1, n_1-1]$ with $x_1 + \ldots + x_{n_1-\epsilon} = n_1-1$,
      \item[(d)] $n_1 y e_2\neq 0$, and
      \item[(e)] either  $s=1$ or $n_1 ye_2 = n_1e_2$.
      \end{itemize}
      In this case, we say that $U$ is of type II with respect to the basis $(e_1, e_2)$.
\end{itemize}
\end{lemma}

We record some observations on this result.
If $n_1 = n_2$, then sequences of type II do not exist as the condition $n_1ye_2 \neq 0$ cannot hold.
Assume that $n_1 \neq n_2$. There are examples of sequences that are both of type I and of type II. However, such sequences are of a rather special form.

If a sequence $U$ is of type I with $j=2$, then it contains an element with multiplicity $n_2 - 1$.
Thus, $U$ is of type II  only when $s = 1$ and $\epsilon = n_1 - 1$ and consequently $x_1 = n_1 - 1$, that is,
$U = e_2^{n_2 - 1}(e_1 + ye_2)^{n_1 -  1}(e_1 + (-(n_1 -1)y+1)e_2)$ with $y \in [0,n_2 -  1]$ and  $n_1 y e_2\neq 0$. Such a sequence is indeed of type I.

If a sequence $U$ is of type I with $j=1$, then it contains an element of order $n_1$ with multiplicity $n_1 -1$.
Since the order of $e_1 + y e_2$ cannot be $n_1$, as $n_1 y e_2\neq 0$,  and the order of $e_2$ is not  $n_1$ either, this is only possible when $\epsilon = 1$ and consequently $x_1 = \dots = x_{n_1 - 1}  =  1$, that is, $U = (e_1+ye_2)^{s n_1 - 1} e_2^{n_2 - s n_1 +1} ( - e_1 +(-y+1) e_2)^{n_1 - 1}$.
If $n_1 ye_2 = n_1e_2$, then indeed   $e_1'=  -e_1 +(-y+1) e_2$ is an element of order $n_1$, we get that $(e_1',e_2)$ is a basis of $G$ and  $U$ is of type I with respect to the basis $(e_1', e_2)$, indeed, $U = e_1'^{n_1-1} ((n_1 -1)e_1' + e_2)^{s n_1 - 1} e_2^{n_2 - s n_1 +1} $ for some $s \in [1, n_2/n_1 -1]$.

\begin{corollary} \label{5.3}
Let  $G = C_{n_1} \oplus C_{n_2}$ where $n_1, n_2 \in \N$   with $n_1 \t n_2$ and $6 \le n_1 < n_2$, and let $U \in \mathcal A (G)$ with  $|U|=\mathsf D (G) = n_1+n_2-1$.
\begin{enumerate}
\item If \ $\mathsf h (U)=n_2-1$, then $U$ is of type I with respect to  a basis  $(e_1, e_2)$  with $\ord (e_1)= n_1$ and $\ord(e_2)= n_2$, that is
\[
U=  e_2^{\ord (e_2)-1} \prod_{\nu=1}^{\ord (e_1)} (x_{\nu}e_2+e_1)
\]
where $x_1, \ldots, x_{n_1}  \in [0, n_2-1]$ with $x_1 + \ldots + x_{n_1} \equiv 1 \mod n_2$.

\smallskip
\item If \ $\mathsf h (U)=n_2-2$, then
      \[
      U = (e_1+ye_2)^{n_1-1} e_2^{n_2-2} \big(-xe_1+(-xy+1)e_2 \big) \ \big(-(n_1-1-x)e_1+(-(n_1-1-x)y+1)e_2 \big) \,,
      \]
      where $(e_1,e_2)$ is a basis with $\ord (e_1)= n_1$, $\ord(e_2)= n_2$, $y \in [0, n_2-1]$, and $x \in [1,(n_1-1)/2]$.

\smallskip
\item If \ $\mathsf h (U)=n_2-3$, then
      \[
      U = (e_1+ye_2)^{n_1-1} e_2^{n_2-3} \prod_{\nu=1}^{3} ( -x_{\nu} e_1 +(-x_{\nu}y+1) e_2) \,,
      \]
      where $(e_1,e_2)$ is a basis  with $\ord (e_1)= n_1$, $\ord(e_2)= n_2$, $y \in [0, n_2-1]$, and $x_1, x_2, x_3 \in [1, n_1-1]$ with $x_1+x_2+x_3 \equiv n_1-1 \mod n_1$ (if $y\ne 0$, then $x_1+x_2+x_3 = n_1-1$).

\smallskip
\item There is at most one element $g \in G$ with $\mathsf v_g (U) \ge n_2-3$. In particular, if $\mathsf h (U) \ge n_2-3$, then there is precisely one element $g \in G$ with $\mathsf v_g (U) = \mathsf h (U)$.
\end{enumerate}
\end{corollary}

\begin{proof}
We use all notation as in Lemma \ref{structure}.

\smallskip
1. If $U$ is of type II with respect to the basis $(e_1, e_2)$, then as observed above $s=1$, $\epsilon=n_1-1$, and
\[
U = (e_1+ye_2)^{n_1-1}e_2^{n_2-1} \big(e_1 + ((-n_1+1)y+1)e_2 \big) \,,
\]
which shows that $U$ is also of type I with respect to the basis $(e_1, e_2)$. If $U$ is of type I with respect to the basis $(e_2,e_1)$ then $\mathsf h(U)= n_2 - 1$ implies that $U$ is also of type I with respect to the basis $(e_1, e_2)$.

\smallskip
2. Suppose that $U$ is of type I with respect to the basis $(f_2, f_1)$. Then $U$ has the form
\[
U= f_1^{n_1-1} (x_1f_1+f_2)^{n_2-2} (x_2f_1+f_2)(x_3f_1+f_2) \,.
\]
Thus $U$ has the asserted form with $y=0$, $e_1=f_1$, and with $e_2=x_1f_1+f_2$.
In this case we only have two summands the congruence condition modulo $n_2$, and hence we obtain an equality in the integers.
Suppose that $U$ is of type II with respect to the basis $(e_1,e_2)$. Then $s=1$, $\epsilon = n_1-2$, and thus the assertion follows.

\smallskip
3. Suppose that $U$ is of type I with respect to the basis $(f_1, f_2)$. Then $U$ has the form
\[
U= f_1^{n_1-1} (x_1f_1+f_2)^{n_2-3} (x_2f_1+f_2)(x_3f_1+f_2) (x_4f_1+f_2) \,.
\]
Thus $U$ has the asserted form with $y=0$, $e_1=f_1$, and with $e_2=x_1f_1+f_2$.
Suppose that $U$ is of type II with respect to the basis $(e_1, e_2)$. Then $s=1$, $\epsilon = n_1-3$, and thus the assertion follows.

\smallskip
4. Assume to the contrary that there are two distinct elements $g_1, g_2 \in G$ with $\mathsf v_{g_1} (U) \ge n_2-3$ and  $\mathsf v_{g_2} (U) \ge n_2-3$. Then
\[
(n_2-3) + (n_2-3) \le \mathsf v_{g_1} (U) + \mathsf v_{g_2} (U) \le |U| = n_1+n_2-1 \,,
\]
which implies that $n_2 \le n_1+5$. Hence $2n_1 \le n_2 \le n_1+5$ and $n_1 \le 5$, a contradiction.
\end{proof}

We recall a technique frequently used in \cite{Ge-Gr-Sc11a} and then provide   a  minor modification of \cite[Lemma 5.3]{Ge-Gr-Sc11a}.

\begin{lemma} \label{5.4}
Let $G$ be a finite abelian group and let $S \in \mathcal F (G)$ be a  zero-sum free sequence. \newline If \ $\prod_{g \in \supp(S)}( 1 + \mathsf v_g (S) ) > |G|$, then there is an $A \in \mathcal A (G)$ with $|A|\ge 3$ such that $(-A) A \mid (-S)S$.
\end{lemma}

\begin{proof}
We observe that
\[
|\{T \in \mathcal F (G) \mid  T \ \text{is a subsequence of} \  S\}|= \prod_{g \in \supp(S)}( 1 + \mathsf v_g (S) ) \,.
\]
Thus, if $\prod_{g \in \supp(S)}( 1 + \mathsf v_g (S) ) > |G|$, then there exist distinct sequences $T_1', T_2' \in \mathcal F (G)$ such that  $T_1' \mid S$, $T_2' \mid S$, and $\sigma (T_1')= \sigma (T _2')$. We set $T_1' = TT_1$ and $T_2' = T T_2$ where  $T = \gcd(T_1',T_2')$ and $T_1, T_2 \in \mathcal F (G)$. Then $\sigma (T_1)=\sigma (T_2)$ and
$(-T_1)T_2$ is a zero-sum subsequence of $(-S)S$.
Let $A \in \mathcal A (G)$ with $A \mid (-T_1)T_2$. Assume to the contrary that  $|A|=2$. Then $A=(-g)g$ for some $g \in G$. Since $S$ is zero-sum free, we infer (after renumbering if necessary) that $(-g)\mid (-T_1)$ and $g\mid T_2$, a contradiction to $\gcd(T_1,T_2)=1$.
Therefore we obtain that $|A|\ge 3$, which implies  that $|\gcd(A,(-g)g)| \le 1$ for each $g \in G$,  and thus $(-A)A \mid (-S)S$.
\end{proof}

\begin{lemma} \label{5.5}
Let $t \in \N$ and $\alpha, \alpha_1, \ldots, \alpha_t , \alpha_1', \ldots, \alpha_t' \in \R$ with
$\alpha_1 \ge \ldots \ge \alpha_t \ge 0$, $\alpha_1' \ge \ldots \ge \alpha_t' \ge 0$, $\alpha_i' \le \alpha_i$ for each $i \in [1,t]$, and $\sum_{i=1}^{t}\alpha_i\geq \alpha \geq \sum_{i=1}^{t}\alpha_i'$. Then
\[
\prod_{\nu=1}^t (1 + x_{\nu})  \qquad \text{is minimal}
\]
over all $(x_1,\ldots,x_t) \in \R^t$ with $\alpha_i' \le x_i \le \alpha_i$ for each $i \in [1,t]$
and $\sum_{i=1}^t x_i = \alpha$, if
\[
x_i = \alpha_i \text{ for each } i \in [1,s] \quad \text{and} \quad x_i = \alpha_i' \quad \text{ for each } i \in [s+2,t]
\]
where $s \in [0, t]$ is maximal with $\sum_{i=1}^s \alpha_i \le \alpha$.
\end{lemma}

\begin{proof}
Since continuous functions attain minima on compact sets, the above function has a minimum at some point $(m_1, \dots, m_t)\in \R^t$. Suppose there are $i, j \in [1,t]$ such that $i < j$  and $m_i< m_j$. Then  $ \alpha_j' \le \alpha_i' \le m_i < m_j \le \alpha_j  \le \alpha_i$, and thus we can exchange  $m_i$ and $m_j$.
Therefore, after renumbering if necessary, we may suppose that $m_1 \ge \dots \ge m_t$.
Since for $x\ge y \ge  0$ and $\delta > 0 $ we have
\[
(1+x+\delta)(1+y-\delta) =  (1+x)(1+y) - \delta (x-y) - \delta^2 < (1+x)(1+y) \,,
\]
it follows that all but at most one of the $m_i$ is equal to $\alpha_i$ or $\alpha_i'$.  It remains to show that there is an $s \in [1,t]$ such that $m_i = \alpha_i$ for $i \in [1,s]$ and $m_i = \alpha_i'$ for each $i \in [s+2,t]$. Assume to the contrary that this is not the case. Then there are $i, j \in [1,t]$ with
$ i < j$ such that $m_i < \alpha_i$ and $\alpha_j' < m_j$.  Using again the just mentioned inequality and that $m_i \ge m_j$,  we obtain a contradiction to the minimum being attained at $(m_1, \dots, m_t)$.
\end{proof}

\begin{proof}[Proof of Proposition \ref{5.1}]
Assume to the contrary that there is an $A \in \mathcal B (G)$ such that $\mathsf L (A) = \{2, n_2, n_1+n_2-2\}$. Then there are $U, V \in \mathcal A (G)$ with $|U| \ge |V|$ such that $A = UV$. We set
\[
U = SU' \quad \text{and} \quad V = (-S)V' \,, \ \text{where $U', V' \in \mathcal F (G)$, and $S = \gcd (U, -V)$}.
\]
Since $2(n_1+n_2-2) \le |A|=|U|+|V| \le 2 \mathsf D (G) = 2(n_1+n_2-1)$, there are the following three cases:
\begin{itemize}
\item[(I)] $|A|= 2(n_1+n_2-2)$. Then, a factorization of $A$ of length $n_1 +n_2 -2$ must contain only minimal zero-sum sequences of length $2$ and thus  $U'=V'=1$.

\item [(II)] $|A|=2(n_1+n_2-2)+1$. Then, a factorization of $A$ of length $n_1 +n_2 -2$ must contain one minimal zero-sum  of length $3$ and otherwise only minimal zero-sum sequences of length $2$, thus $U' = g_1 g_2$ and  $V' = (-g_1-g_2)$ for some $g_1, g_2 \in G$.

\item[(III)] $|A|= 2(n_1+n_2-1)$. Then a factorization of $A$ of length $n_1 +n_2 -2$ must contain either one minimal zero-sum subsequence of length $4$ and otherwise  minimal zero-sum sequences of length $2$, or two minimal zero-sum sequences of length $3$ and otherwise only minimal zero-sum sequences of length $2$. Thus,  there are the following two subcases.
             \begin{itemize}
             \item $U'=g_1g_2$, $V'=h_1h_2$ where $g_1,g_2, h_1,h_2 \in G$ such that $g_1g_2h_1h_2 \in \mathcal A (G)$.

             \item $U'=g_1g_2(-h_1-h_2)$ and $V'=h_1h_2(-g_1-g_2)$ where $g_1,g_2, h_1,h_2 \in G$.

             \end{itemize}
\end{itemize}

\smallskip
We start with the following two assertions.

\begin{enumerate}
\item[{\bf A1.}\,] Let $W \in \mathcal A (G)$ with $|W| < |U|$ and $W \t (-S)S$. Then $|W| \in \{2, n_1\}$.

\smallskip
\item[{\bf A2.}\,] Let $W_1, W_2 \in \mathcal A (G)$ such that $W_1(-W_1)W_2 (-W_2) \t S(-S)$. Then $\{|W_1|, |W_2|\} \ne \{n_1\}$.
\end{enumerate}

{\it Proof of \,{\bf A1.}}
Suppose $|W|>2$. Then $(-W)W \t (-S)S$ and we set $(-S)S=(-W)WT(-T)$ with $T \in \mathcal F (G)$ and obtain that
\[
UV = (-W)WT(-T)(U'V') \,.
\]
Let $z$ be a factorization of $U'V'$. Then $|z| \in [0,2]$. If $T=1$, then $UV$ has a factorization of length $2+|z| \in \{2, n_2, n_1+n_2-2\}$ which implies $|z|=0$ and hence $|W|=|U|$, a contradiction. Thus $T \ne 1$. Since $T(-T)$ has a factorization of length $|T|=|S|-|W|$, the above decomposition gives rise to a factorization of $UV$ of length $t$ where
\[
3 \le t = 2+|T|+|z|=2+|S|-|W|+|z| \in \{n_2, n_1+n_2-2\} \,.
\]
We distinguish four cases.

Suppose that $U'=V'=1$. Then $z=1$, $|z|=0$, and $|S|=|U|=n_1+n_2-2$. Thus $t=n_1+n_2-|W| \in \{n_2, n_1+n_2-2\}$, and the assertion follows.

Suppose that $U' = g_1 g_2$ and  $V' = (-g_1-g_2)$ for some $g_1, g_2 \in S$. Then $|z|=1$ and $|S|=n_1+n_2-3$. Thus $t=n_1+n_2-|W| \in \{n_2, n_1+n_2-2\}$, and the assertion follows.

Suppose that $U'=g_1g_2$, $V'=h_1h_2$ where $g_1,g_2,h_1,h_2 \in G$ such that $g_1g_2h_1h_2 \in \mathcal A (G)$. Then $|z|=1$ and $|S|=n_1+n_2-3$. Thus $t=n_1+n_2-|W| \in \{n_2, n_1+n_2-2\}$, and the assertion follows.

Suppose that $U'=g_1g_2(-h_1-h_2)$ and $V'=h_1h_2(-g_1-g_2)$ where $g_1,g_2, h_1,h_2 \in G$. Then $|z|=2$ and $|S|=n_1+n_2-4$. Thus $t=n_1+n_2-|W| \in \{n_2, n_1+n_2-2\}$, and the assertion follows. \qed ({\bf A1})

\smallskip
{\it Proof of \,{\bf A2}}.\, Assume to the contrary that $|W_1|=|W_2|=n_1$. Then there are $W_5, \ldots, W_k \in \mathcal A (G)$ such that
\[
UV = W_1(-W_1)W_2(-W_2)W_5 \cdot \ldots \cdot W_k \,,
\]
where $k \in \mathsf L (UV) = \{2, n_2, n_1+n_2-2\}$ and hence $k=n_2$. Since
\[
|W_5 \cdot \ldots \cdot W_k| = |UV|-4n_1 \le 2(n_1+n_2-1)-4n_1= 2(n_2-n_1-1) \,,
\]
it follows that
\[
k-4\le \max \mathsf L ( W_5 \cdot \ldots \cdot W_k ) \le |W_5 \cdot \ldots \cdot W_k|/2 \le n_2-n_1-1 < n_2-4 \,,
\]
a contradiction. \qed ({\bf A2})

\smallskip
We distinguish two cases depending on the size of $\mathsf{h}(S)$.

\bigskip
\noindent
CASE 1: \  $\mathsf{h}(S) \ge n_2/2$.

We set $S = g^vS'$ where $g \in G$, $v = \mathsf h (S)$, and $S' \in \mathcal F (G)$. Then
\[
U_1 = (-g)^{n_2-v}S'U' \in \mathcal B (G), \ V_1= g^{n_2-v} (-S') V' \in \mathcal B (G) \,.
\]
Clearly, we have
\[
(U')^{-1}U_1 = (-g)^{n_2-v} S' = - \big( (V')^{-1}V_1 \big) \,.
\]
We will often use that if some $W \in \mathcal A (G)$ divides $(U')^{-1}U_1$, then $(-W)$ divides $(V')^{-1}V_1$ and hence $(-W)W \t (-S)S$.
Now we choose  factorizations $x_1 \in \mathsf Z (U_1)$ and  $y_1 \in \mathsf Z (V_1)$. Note that $|x_1| \le n_2-v$ and $|y_1| \le n_2-v$ as each minimal zero-sum sequence in $x_1$ and $y_1$ contains $(-g)$ and $g$, respectively. Then $UV = U_1 V_1 \big( (-g)g \big)^{2v-n_2}$
has a factorization of length $t$ where
\[
2+(2v-n_2) \le t = |x_1|+|y_1|+(2v-n_2) \le 2(n_2-v)+(2v-n_2)=n_2 \,.
\]
Assume to the contrary that $t=2$. Then $v=n_2/2$ and both, $U= g^{n_2/2}S'U'$ and $U'= (-g)^{n_2/2}S'U'$, are minimal zero-sum sequences, a contradiction, as $SU' \notin \mathcal{A} \big(  \mathcal B_{\langle g \rangle} (G) \big)$ as its length is greater than $n_1 = \mathsf D_{\langle g \rangle} (G) = \mathsf D (G/\langle g \rangle)$. Thus $t=n_2$, $|x_1|=|y_1|=n_2-v$, and hence $\mathsf L (U_1)=\mathsf L (V_1) = \{n_2-v\}$. If $W \in \mathcal A (G)$ with $|W|=2$ and $W \t U_1$, then $W=(-g)g$. Similarly, if $W' \in \mathcal A (G)$ with $|W'|=2$ and $W' \t V_1$, then $W=(-g)g$. By definition of $v$, not both $U_1$ and $V_1$ are divisible by an atom of length $2$.  Now  we distinguish four cases depending on the form of $U'$ and $V'$, which we determined above.

\bigskip
\noindent
CASE 1.1: \ $U'=V'=1$.

Then $V_1=-U_1$, say $V_1 = W_1 \cdot \ldots \cdot W_{n_2-v}$. Since none of the $W_i$ has length $2$, it follows that $|W_1|= \ldots = |W_{n_2-v}|=n_1$ and hence
\[
2(n_2+n_1-2)=2|V|= 2(2v-n_2) + 2|V_1|= 2(2v-n_2)+ 2n_1(n_2-v) \,,
\]
which implies that $v=n_2-1$.
Consequently $|S| = n_1-1$ and $S \in \mathcal{A} \big(  \mathcal B_{\langle g \rangle} (G) \big)$. This implies that  (use an elementary direct argument or \cite[Theorem 5.1.8]{Ge09a}),
\[
S  = (2e_1 + a_1 e_2)\prod_{\nu=2}^{n_1 -1}(e_1 + a_{\nu} e_2) \,,
\]
where $(e_1,e_2)$ is a basis of $G$.
Let $r \in [0, n_2-1]$ such that   $r \equiv -a_1 +a_2 + a_3 \mod n_2$. Then
\[
W_1= (2e_1 + a_1 e_2)(-e_1 - a_2 e_2)(-e_1 - a_3 e_2)e_2^r \quad \text{and} \quad W_2=
(2e_1 + a_1 e_2)(-e_1 - a_2 e_2)(-e_1 - a_3 e_2)(-e_2)^{n_2 - r}
\]
are minimal zero-sum sequences dividing $(-V)V$. Since $|W_1|=3+r$, $|W_2|=3+n_2-r$, and $|W_1W_2|=n_2+6>2n_1$, at least one of them does not have length $n_1$, a contradiction.

\bigskip
\noindent
CASE 1.2: \ $U' = g_1 g_2$ and  $V' = (-g_1-g_2)$ for some $g_1, g_2 \in G$.

We set
\[
x_1=X_1 \cdot \ldots \cdot X_{n_2-v} \quad \text{and} \quad y_1 = Y_1 \cdot \ldots \cdot Y_{n_2-v}
\]
where all $X_i, Y_j \in \mathcal A (G)$, $g_1g_2 \t X_1X_2$ (or even $g_1g_2 \t X_1$), and $(-g_1-g_2)\t Y_1$. We distinguish three subcases.

\smallskip
\noindent
CASE 1.2.1: \ $v=n_2-1$.

By Corollary \ref{5.3}, with all notations as introduced there, we get $U  = e_2^{n_2-1} \prod_{\nu=1}^{n_1} (e_1+ x_{\nu} e_2)$.
Thus $g=e_2$ and  $U' \t \prod_{\nu=1}^{n_1} (e_1+ x_{\nu} e_2)$, whence after renumbering if necessary we have $g_i=x_ie_1+e_2$ for each $i \in [1,2]$. Therefore we have
\[
V = \big(-2e_1 -(x_1+x_2)e_2 \big) (-e_2)^{n_2-1} \prod_{\nu=3}^{n_1} (-e_1-x_{\nu}e_2)  = (-g_1-g_2)(- S) \,.
\]
Assume to the contrary that there are $i, j \in [3, n_1]$ distinct with $x_i \ne x_j$. If $q \in [1, n_2-1]$ with $q \equiv -(x_i-x_j) \mod n_2$, then
\[
W_1' = (e_1+x_ie_2)(-e_1-x_je_2)e_2^q \quad \text{and} \quad W_2' = (e_1+x_ie_2)(-e_1-x_je_2)(-e_2)^{n_2-q}
\]
are atoms dividing $(-S)S$, both have length greater than two but  not both have length $n_1$, a contradiction to {\bf A1}. Therefore we have $x_3= \ldots = x_{n_1}$.
Since $x_1+ \ldots + x_{n_1} \equiv 1 \mod n_2$, it follows that $x_1 \ne x_3$ or $x_2 \ne x_3$, say $x_2 \ne x_3$.
Therefore there is an $r \in [1, n_2-1]$ with $r \equiv x_2-x_3$ such that
\[
W_1= \big( -(x_1+x_2)e_2-2e_1 \big) (x_1e_2+e_1)(x_3e_2+e_1)e_2^{r} \in \mathcal A (G)
\]
and
\[
W_2 = (x_2e_2+e_1)(-x_3e_2-e_1)(-e_2)^{r} \in \mathcal A (G) \,.
\]
Thus it follows  that
\[
UV = W_1W_2 \prod_{\nu=4}^{n_1} \big( (x_{\nu} e_2+e_1)(-x_{\nu} e_2-e_1) \big) \big( (-e_2)e_2 \big)^{n_2-1-r}
\]
has a factorization of length
\[
2+(n_1-3)+(n_2-1-r) = n_1+n_2-2-r \in \{2, n_2, n_1+n_2-2\} \,,
\]
and hence $r= n_1-2$. Now we define
\[
W_1' = \big( -(x_1+x_2)e_2-2e_1 \big) (x_1e_2+e_1)(x_3e_2+e_1)(-e_2)^{n_2-r}
\]
and
\[
W_2' = (x_2e_2+e_1)(-x_3e_2-e_1) e_2^{n_2-r} \in \mathcal A (G) \,.
\]
Thus it follows  that
\[
UV = W_1'W_2' \prod_{\nu=4}^{n_1} \big( (x_{\nu}e_2+e_1)(-x_{\nu}e_2-e_1) \big) \big( (-e_2)e_2 \big)^{r-1}
\]
has a factorization of length
\[
2+(n_1-3)+(r-1) = n_1-2+r = 2n_1-4 \notin \{2, n_2, n_1+n_2-2\} \,, \ \text{a contradiction.}
\]

\noindent
CASE 1.2.2: \ $v=n_2-2$.

Since $V' \mid Y_1$ it follows that $Y_2 \t (-S)S$ and we thus may assume that $X_2  = -Y_2$. Furthermore, $|Y_2|$ cannot have length $2$, since $-g \nmid S'$. Hence {\bf A1} gives that  $|Y_2|$ has length $n_1$. It follows that $|Y_1|=  2$ and $|X_1|= 3$. This implies that $-g_1-g_2 = -g$ whence $\mathsf v_{-g}(V)=n_2-1$ and $\mathsf v_{g}(U)=n_2-2$.  By Corollary \ref{5.3}, with all notations as introduced there, we obtain that
\[
\begin{aligned}
U & = (e_1+ye_2)^{n_1-1} e_2^{n_2-2} \big(-xe_1+(-xy+1)e_2 \big) \ \big(-(n_1-1-x)e_1+(-(n_1-1-x)y+1)e_2 \big) \,.
\end{aligned}
\]
Since $g_1+g_2=g=e_2$, it follows that $x=1$ and that
\[
V =   (-(e_1+ye_2))^{ n_1 - 2} (-e_2)^{n_2 - 1} ( (n_1 -2) e_1 +((n_1 -2)y+1) e_2) \,.
\]
Then $W = (e_1+ye_2)^2  ( (n_1 -2) e_1 +((n_1 -2)y+1) e_2) (-e_2)^r$, where  $r \in [0, n_2 - 1]$ such that  $r \equiv n_1 y +1 \mod n_2$, is a minimal zero-sum sequence. Since $r \equiv 1 \mod n_1$, it follows that $r \in [1, n_2 -n_1 +1]$. Thus,  $W \t (-S)S$,
hence  $|W|= n_1$, and thus $r = n_1 -3$. We consider $W'= (e_1+ye_2)^2  ( (n_1 -2) e_1 +((n_1 -2)y+1) e_2) e_2^{n_2-(n_1 - 3)}$.
Again, $W' \t (-S)S$. Yet $|W'|= 3 + (n_2 - (n_1 - 3)) = n_2 - n_1 + 6$, a contradiction.

\smallskip
\noindent
CASE 1.2.3: \ $v \le n_2-3$.

Then $Y_2 Y_3 \t (V')^{-1}V_1$, and since $|Y_2|\ne2\ne|Y_3|$, we infer that $|Y_2|=|Y_3|=n_1$. Thus $(-Y_2)(-Y_3) \t (U')^{-1}U_1$ and $Y_2(-Y_2)Y_3(-Y_3) \t S(-S)$, a contradiction to {\bf A2}.

\bigskip
\noindent
CASE 1.3: \ $U'=g_1g_2$ and $V'=h_1h_2$, where $g_1,g_2,h_1,h_2 \in G$ such that $g_1g_2h_1h_2 \in \mathcal A (G)$.

We set
\[
x_1=X_1 \cdot \ldots \cdot X_{n_2-v} \quad \text{and} \quad y_1 = Y_1 \cdot \ldots \cdot Y_{n_2-v}
\]
where all $X_i, Y_j \in \mathcal A (G)$, $g_1g_2 \t X_1X_2$ (or even $g_1g_2 \t X_1$ ), and $h_1h_2 \t Y_1 Y_2$ (or even  $h_1h_2 \t Y_1$). We distinguish four cases.

\smallskip
\noindent
CASE 1.3.1: \ $v = n_2-1$.

By Corollary \ref{5.3},  we have
\[
U  = e_2^{n_2-1} \prod_{\nu=1}^{n_1} (e_1+ x_{\nu} e_2)  \ \text{and} \
-V  = e_2^{n_2-1} \prod_{\nu=1}^{n_1} (e_1' + x_{\nu}' e_2) \,,
\]
where $(e_1, e_2)$ and $(e_1', e_2)$ are both bases with $\ord(e_2)= n_2$ and $x_i, x_i' \in [0, n_2-1]$ for each $i \in [1, n_1]$. Since $|S|=|U|-2$, it follows that, after renumbering if necessary,  $\prod_{\nu=3}^{n_1} (e_1+x_{\nu}e_2) \t (-V)$ and hence, after a further renumbering if necessary, $e_1+x_ie_2=e_1'+x_i' e_2$ for each $i \in [3, n_1]$. Thus, if we write $-V$ with respect to the basis $(e_1,e_2)$, it still has the above structure. Therefore we may assume that $e_1=e_1'$ and $x_i=x_i'$ for each $i \in [3, n_1]$. Therefore
\[
g_i = e_1+x_ie_2 \quad \text{and} \quad h_i = -e_1-x_i'e_2 \quad \text{for each} \ i \in [1,2] \,.
\]
Since $g_1+g_2=-h_1-h_2$, it follows that $x_1+x_2 \equiv -x_1'-x_2' \mod n_2$ and hence $x_1-x_1' \equiv x_2'-x_2 \mod n_2$. Let $r \in [0, n_2-1$ such that $r \equiv x_1 - x_1' \mod n_2$. Then
\[
W_1 = g_1h_1 (-e_2)^r \,, \ W_1' = g_1h_1 e_2^{n_2-r} \,, \ W_2=g_2h_2e_2^r \,, \quad \text{and} \quad W_2' = g_2h_2(-e_2)^{n_2-r}
\]
are minimal zero-sum sequences which give rise to the  factorizations
\[
\begin{aligned}
UV & = W_1W_2 \big( (-e_2)e_2 \big)^{n_2-r-1} \prod_{\nu=3}^{n_1} \big( (e_1+x_{\nu}e_2)(-e_1-x_{\nu}e_2) \big) \\
 & = W_1' W_2' \big( (-e_2)e_2 \big)^{r-1} \prod_{\nu=3}^{n_1} \big( (e_1+x_{\nu}e_2)(-e_1-x_{\nu}e_2) \big) \,.
\end{aligned}
\]
These factorizations have length $2+(n_2-r-1)+(n_1-2) = n_1+n_2-2-(r-1)$ and $2+(r-1)+(n_1-2)=n_1+r-1$. Since not both of them can be in $\{2, n_2, n_1+n_2-2\}$, a contradiction.

\smallskip
\noindent
CASE 1.3.2: \ $v = n_2-2$.

Assume to the contrary that $\mathsf h (U)=\mathsf h (V)=n_2-1$. Since, by Corollary \ref{5.3}, the elements $g', g'' \in G$ with $\mathsf v_{g'}(U)=n_2-1$ and $\mathsf v_{g''}(V)=n_2-1$ are uniquely determined, it follows that $g=g'=-g''$ and hence $v=\mathsf h (S)=n_2-1$, a contradiction.
Thus, after exchanging $U$ and $V$ if necessary, we may assume that $\mathsf h (U) = n_2-2$ and it remains to consider the two cases $\mathsf h (V) = n_2-1$ and $\mathsf h (V) = n_2-2$.

\smallskip
\noindent
CASE 1.3.2.1: \ $\mathsf h (U) = n_2-2$ and $\mathsf h (V) = n_2-1$.

By Corollary \ref{5.3}, we infer that
\[
\begin{aligned}
 U & = (e_1+ye_2)^{n_1-1} e_2^{n_2-2} \big(-xe_1+(-xy+1)e_2 \big) \ \big(-(n_1-1-x)e_1+(-(n_1-1-x)y+1)e_2 \big) \,, \\
-V & = e_2^{n_2-1} \prod_{\nu=1}^{n_1}(e_1'+x_{\nu} e_2)  \,,
\end{aligned}
\]
where $(e_1, e_2)$ and $(e_1', e_2)$ are bases and all parameters are as in Corollary \ref{5.3}. Since $|S|=|U|-2$, it follows that $(e_1+ye_2)^{n_1-3} \t (-V)$ and hence, after renumbering if necessary, $e_1'+x_1e_2= \ldots = e_1'+x_{n_1-3}e_2=e_1+y e_2$. Thus, if we write $-V$ with respect to the basis $(e_1,e_2)$, it still has the above structure. Therefore we may assume that $e_1=e_1'$ and $y=x_1= \ldots = x_{n_1-3}$. Thus we obtain that
\[
-V = e_2^{n_2-1} (e_1+ye_2)^{n_1-3} \prod_{\nu=n_1-2}^{n_1}(e_1+x_{\nu} e_2) \,.
\]
Note that  $e_2 \in \{-h_1, -h_2\}$, $e_2 \notin \{g_1, g_2\}$,  say $-h_1=e_2$ and $-h_2=e_1+x_{n_1}e_2$, and $g_1+g_2=-(h_1+h_2)=e_1+(x_{n_1}+1)e_2$. This condition on the sum shows that
\[
\{g_1, g_2\} = \{ -xe_1+(-xy+1)e_2, -(n_1-1-x)e_1+(-(n_1-1-x)y+1)e_2  \} \,.
\]
This implies that $(e_1+ye_2)^{n_1-1} \t (-V)$ and hence, after renumbering if necessary,
\[
-V = e_2^{n_2-1} (e_1+ye_2)^{n_1-1} (e_1+x_{n_1} e_2) \,.
\]
Since $g_1+g_2 = -(n_1-1)e_1 - (y(n_1-1)-2)e_2 $, it follows that the sequence
\[
W_1 = g_1g_2 (-e_1-ye_2)(-e_2)^r \,,
\]
where $r \in [0, n_2-1]$ and $r \equiv -yn_1+2 \mod n_2$, is a minimal zero-sum sequence. Since $r \equiv 2 \mod n_1$, we infer that $r \in [2, n_2-2]$ and that $W_1 \t UV$. Since $g_1+g_2=-(h_1+h_2)$, we obtain that
\[
W_2 = h_1h_2 (e_1+ye_2)e_2^r \in \mathcal B (G) \ , \ \mathsf L (W_2) = \{2\} \quad \text{and} \quad W_2 \t UV \,.
\]
Therefore it follows that
\[
UV = W_1 W_2 \big( (-e_2)e_2 \big)^{n_2-2-r} \big( (e_1+ye_2)(-e_1-ye_2) \big)^{n_1-2} \,,
\]
and hence $n_1+n_2-1-r \in \mathsf L (UV) = \{2, n_2, n_1+n_2-2\}$, a contradiction, since $r \equiv 2 \mod n_1$.

\smallskip
\noindent
CASE 1.3.2.2: \ $\mathsf h (U) = \mathsf h (V) = n_2-2$.

By Corollary \ref{5.3}, we infer that
\[
U  =  (e_1+ye_2)^{n_1-1}e_2^{n_2-2} U'' \ \text{and} \
-V  = (e_1'+y'e_2)^{n_1-1} e_2^{n_2-2} (-V'') \,,
\]
where $(e_1, e_2)$ and $(e_1',e_2)$ are bases, $U'', V'' \in \mathcal F (G)$ with $|U''|=|V''|=2$, and $y, y' \in [0, n_2-1]$. Since $|S|=|U|-2$, it follows that $(e_1'+y'e_2)^{n_1-3} \t U$ and hence $e_1'+y'e_2=e_1+ye_2$. Thus, if we write $-V$ with respect to the basis $(e_1,e_2)$, it still has the above structure. Therefore we may assume that $e_1=e_1'$ and $y=y'$. Therefore it follows that
\[
U  =  (e_1+ye_2)^{n_1-1}e_2^{n_2-2} g_1 g_2 \quad \text{and} \quad V = (-e_1-ye_2)^{n_1-1}(-e_2)^{n_2-2} h_1 h_2 \,,
\]
and hence $g_1+g_2 = -(n_1-1)e_1 - (y(n_1-1)-2)e_2 $. Thus
\[
W_1 = g_1g_2 (-e_1-ye_2)(-e_2)^r \,,
\]
where $r \in [0, n_2-1]$ and $r \equiv -yn_1+2 \mod n_2$, is a minimal zero-sum sequence. Since $r \equiv 2 \mod n_1$, we infer that $r \in [2, n_2-2]$ and that $W_1 \t UV$. Since $g_1+g_2=-(h_1+h_2)$, we obtain that
\[
W_2 = h_1h_2 (e_1+ye_2)e_2^r \in \mathcal A (G) \quad \text{and} \quad W_2 \t UV \,.
\]
Therefore it follows that
\[
UV = W_1 W_2 \big( (-e_2)e_2 \big)^{n_2-2-r} \big( (e_1+ye_2)(-e_1-ye_2) \big)^{n_1-2} \,,
\]
and hence $n_1+n_2-2-r \in \mathsf L (UV) = \{2, n_2, n_1+n_2-2\}$, a contradiction, since $r \equiv 2 \mod n_1$.

\smallskip
\noindent
CASE 1.3.3: \ $v = n_2-3$.

Then $X_3 \t (U')^{-1}U_1$ and hence $|X_3|=n_1$. Since
\[
|X_1X_2X_3|=|U_1|=|U|-v+(n_2-v)=n_1+5 \,,
\]
it follows that $|X_1X_2|=5$, and hence $X_1$ or $X_2$ has length two. Similarly, we obtain that $Y_1$ or $Y_2$ has length two, a contradiction to the earlier mentioned fact that not both, $U_1$ and $V_1$ are divisible by an atom of length two.

\smallskip
\noindent
CASE 1.3.4: \ $v \le n_2-4$.

Then $Y_3 Y_4 \t (V')^{-1}V_1$, and since $|Y_3|\ne2\ne|Y_4|$, we infer that $|Y_3|=|Y_4|=n_1$. Thus $(-Y_3)(-Y_4) \t (U')^{-1}U_1$ and $Y_3(-Y_3)Y_4(-Y_4) \t S(-S)$, a contradiction to {\bf A2}.

\bigskip
\noindent
CASE 1.4: \ $U'=g_1g_2(-h_1-h_2)$ and $V'=h_1h_2(-g_1-g_2)$, where $g_1,g_2, h_1,h_2 \in G$.

We set
\[
x_1=X_1 \cdot \ldots \cdot X_{n_2-v} \quad \text{and} \quad y_1 = Y_1 \cdot \ldots \cdot Y_{n_2-v}
\]
where all $X_i, Y_j \in \mathcal A (G)$, $g_1g_2 (-h_1-h_2) \t X_1X_2 X_3$ (or even $g_1g_2 (-h_1-h_2) \t X_1X_2 $ or $g_1g_2 (-h_1-h_2) \t X_1$), and $h_1h_2 (-g_1-g_2) \t Y_1 Y_2 Y_3$ (or even $h_1h_2 (-g_1-g_2) \t Y_1 Y_2$ or $h_1h_2 (-g_1-g_2) \t Y_1$). We distinguish five subcases 1.4.1 -- 1.4.5.

\smallskip
\noindent
CASE 1.4.1: \ $v = n_2-1$.

By Corollary \ref{5.3} we have
\[
U  = e_2^{n_2-1} \prod_{\nu=1}^{n_1} (e_1+ x_{\nu} e_2)  \ \ \text{and} \
-V  = e_2^{n_2-1} \prod_{\nu=1}^{n_1} (e_1' + x_{\nu}' e_2) \,,
\]
where $(e_1, e_2)$ and $(e_1', e_2)$ are both bases with $\ord (e_2)= n_2$ and $x_{\nu}, x_{\nu}' \in [0, n_2-1]$ for each $\nu \in [1, n_1]$. Since $|S|=|U|-3$, it follows that, after renumbering if necessary,  $\prod_{\nu=4}^{n_1} (e_1+x_{\nu}e_2) \t (-V)$ and hence, after a further renumbering if necessary,  $e_1+x_{\nu}e_2=e_1'+x_{\nu}' e_2$ for each $\nu \in [4, n_1]$. Thus, if we write $-V$ with respect to the basis $(e_1,e_2)$, it still has the above structure. Therefore we may assume that $e_1=e_1'$ and $x_{\nu}=x_{\nu}'$ for each $\nu \in [4, n_1]$. Furthermore, we obtain that
\[
\begin{aligned}
g_1=e_1+x_1e_2, \ \ & g_2=e_1+x_2e_2, \ \ -h_1-h_2=e_1+x_3e_2 \\
h_1 = -e_1-x_1'e_2, \ \ & h_2=-e_1-x_2'e_2, \quad \text{and} \quad -g_1-g_2=-e_1-x_3'e_2 \,,
\end{aligned}
\]
a contradiction, since $\ord(e_1)= n_1 > 3$

\smallskip
\noindent
CASE 1.4.2: \ $v = n_2-2$.

Arguing as at the beginning of CASE 1.3.2 we may assume $\mathsf h (U) = n_2-2$ and it is sufficient to consider the two subcases $\mathsf h (V) = n_2-1$ and $\mathsf h (V) = n_2-2$.

\smallskip
\noindent
CASE 1.4.2.1: \ $\mathsf h (U) = n_2-2$ and $\mathsf h (V) = n_2-1$.

By Corollary \ref{5.3}, we infer that $-V  = e_2^{n_2-1} \prod_{\nu=1}^{n_1}(e_1'+x_{\nu} e_2) $ and
\[
 U  = (e_1+ye_2)^{n_1-1} e_2^{n_2-2} \big(-xe_1+(-xy+1)e_2 \big) \ \big(-(n_1-1-x)e_1+(-(n_1-1-x)y+1)e_2 \big) \,,
 \,,
\]
where $(e_1, e_2)$ and $(e_1', e_2)$ are bases and all parameters are as in Corollary \ref{5.3}. Since $|S|=|U|-3$, it follows that $(e_1+ye_2)^{n_1-4} \t (-V)$ and hence, after renumbering if necessary, $e_1'+x_1e_2= \ldots = e_1'+x_{n_1-4}e_2=e_1+y e_2$. Thus, if we write $-V$ with respect to the basis $(e_1,e_2)$, it still has the above structure. Therefore we may assume that $e_1=e_1'$ and $y=x_5= \ldots = x_{n_1}$. Thus we obtain that
\[
-V = e_2^{n_2-1} (e_1+ye_2)^{n_1-4} \prod_{\nu=1}^{4}(e_1+x_{\nu} e_2) \,.
\]
Since
\[
\gcd \Big( \big(-xe_1+(-xy+1)e_2 \big) \ \big(-(n_1-1-x)e_1+(-(n_1-1-x)y+1)e_2 \big), \ -V \Big) = 1 \,,
\]
it follows that
\[
g_1g_2(-h_1-h_2) = (e_1+ye_2) \big(-xe_1+(-xy+1)e_2 \big) \ \big(-(n_1-1-x)e_1+(-(n_1-1-x)y+1)e_2 \big) \,.
\]
Thus $\mathsf v_{e_1+ye_2}(S) = n_1-2$ and hence, after renumbering if necessary,
\[
-V = e_2^{n_2-1} (e_1+ye_2)^{n_1-2} (e_1+x_{1} e_2) (e_1+x_{2} e_2) \,.
\]
We observe that
\[
\begin{aligned}
\big(-xe_1+(-xy+1)e_2 \big) + \big(-(n_1-1-x)e_1+(-(n_1-1-x)y+1)e_2 \big) = e_1 + (-(n_1-1)y+2)e_2 \,, \\
(-e_1-x_{1} e_2) + (-e_1-x_{2} e_2) = (n_1-2)(e_1+ye_2)+(n_2-1)e_2= -2e_1 + ((n_1-2)y-1)e_2 \,.
\end{aligned}
\]
Consequently, there are $r, r' \in [0, n_2-1]$ such that
\[
\begin{aligned}
W_1 & = \big(-xe_1+(-xy+1)e_2 \big) \big(-(n_1-1-x)e_1+(-(n_1-1-x)y+1)e_2 \big) (-e_1-ye_2) (-e_2)^r \in \mathcal A (G) \ \text{and} \\
W_2 & = (-e_1-x_1e_2)(-e_1-x_2e_2)(e_1+ye_2)^2 e_2^{r'} \in \mathcal A (G)
\end{aligned}
\]
(note that $y \notin \{-x_1, -x_2\}$). Clearly we have
\[
r \equiv 2 - n_1 y \mod n_2 \quad \text{and} \quad r' \equiv   1 - n_1y \mod n_2 \,.
\]
This implies  that  $r \equiv 2 \mod n_1$ and $r' = r-1$. Therefore, we obtain that
\[
UV = W_1 W_2 \big( (e_1+ye_2)(-e_1-ye_2) \big)^{n_1-3} \big( (-e_2)e_2 \big)^{n_2-r-1} \,,
\]
and thus $n_1+n_2-2-r \in \mathsf L (UV) = \{2, n_2, n_1+n_2-2\}$, a contradiction, since $r \equiv 2 \mod n_1$.

\smallskip
\noindent
CASE 1.4.2.2: \ $\mathsf h (U) = \mathsf h (V) = n_2-2$.

By Corollary \ref{5.3}, we infer that
\[
U  =  (e_1+ye_2)^{n_1-1}e_2^{n_2-2} U'' \ \text{and} \
-V  = (e_1'+y'e_2)^{n_1-1} e_2^{n_2-2} (-V'') \,,
\]
where $(e_1, e_2)$ and $(e_1',e_2)$ are bases, $U'', V'' \in \mathcal F (G)$ with $|U''|=|V''|=2$, and $y, y' \in [0, n_2-1]$. Since $|S|=|U|-3$, it follows that $(e_1'+y'e_2)^{n_1-4} \t U$. If $n_1 > 6$, it follows that  $e_1'+y'e_2=e_1+ye_2$. If $n_1=6$, so $n_1$ even, then
\[
U'' = \big(-xe_1+(-xy+1)e_2 \big) \ \big(-(n_1-1-x)e_1+(-(n_1-1-x)y+1)e_2 \big)
\]
is not a square whence $(e_1'+y'e_2)^{2} \ne  U''$ and it follows again that $e_1'+y'e_2=e_1+ye_2$. Thus, if we write $-V$ with respect to the basis $(e_1,e_2)$, it still has the above structure. Therefore we may assume that $e_1=e_1'$ and $y=y'$. Therefore it follows that
\[
U  =  (e_1+ye_2)^{n_1-1}e_2^{n_2-2} U'' \quad \text{and} \quad V = (-e_1-ye_2)^{n_1-1}(-e_2)^{n_2-2} V'' \,,
\]
a contradiction to $|S|=|U|-3$.

\smallskip
\noindent
CASE 1.4.3: \ $v = n_2-3$.

Note that
\[
|X_1X_2X_3|=|U_1|=|U|-v+(n_2-v)=n_1+n_2-1+n_2-(2n_2-6)=n_1+5 \,.
\]

Suppose that $U'$ divides a product of two of the $X_1, X_2, X_3$, say $U' \t X_1X_2$. Then $X_3 \t (U')^{-1}U_1$ and hence $|X_3|=n_1$. Thus $|X_1X_2|=5$ and either $X_1$ or $X_2$ has length two.
Since $X_3$ divides $(U')^{-1}U_1$, it follows that $-X_3$ divides $V_1(V')^{-1}$. After considering a new factorization of $V_1$ if necessary we may suppose without restriction that $Y_3=-X_3$. Arguing as above we infer that $Y_1$ or $Y_2$ has length two, a contradiction to the earlier mentioned fact that not both, $U_1$ and $V_1$ are divisible by an atom of length two.

Thus from now on we may assume that for every  $X \in \mathcal A (G)$ dividing $U_1$ we have $|\gcd (X, U')|=1$, and  for every  $Y \in \mathcal A (G)$ dividing $V_1$ we have $|\gcd (Y, V')|=1$.
Arguing as at the beginning of CASE 1.3.2 we obtain that $\mathsf h (U) = n_2-3$ or $\mathsf h (V)=n_2-3$, say $\mathsf h (U)=n_2-3$.
By Corollary \ref{5.3}, we infer that
 \[
      U = (e_1+ye_2)^{n_1-1} e_2^{n_2-3} \prod_{\nu=1}^{3} ( -x_{i} e_1 +(-x_{\nu}y+1) e_2) \,,
 \]
with all parameters as described there. Since $|S|=|U|-3$, it follows that $(e_1+ye_2)^{n_1-4} \t (-V)$ and thus
\[
V = (-e_1-ye_2)^{n_1-4}(-e_2)^{n_2-3} V'' \quad \text{where} \quad V'' \in \mathcal F (G) \ \text{with} \ |V''|=6 \,.
\]
Since
\[
U_1 =(-e_2)^{3} (e_1+ye_2)^{n_1-1}  \prod_{\nu=1}^{3} ( -x_{\nu} e_1 +(-x_{\nu}y+1) e_2) = X_1X_2X_3 \,,
\]
it follows that $(-e_2) \t X_{\nu}$ for each $\nu \in [1,3]$. Since $(e_1+ye_2)^{n_1-1} (-e_2)^{3}$ is zero-sum free, each of the $X_{\nu}$ is divisible by at least one of the elements from $\prod_{\nu=1}^{3} ( -x_{\nu} e_1 +(-x_{\nu}y+1) e_2)$. Thus, after renumbering if necessary, it follows that for each $\nu \in [1,3]$
\[
X_{\nu} = (-e_2) ( -x_{\nu} e_1 +(-x_{\nu}y+1) e_2) (e_1+ye_2)^{x_{\nu}} \,.
\]
This implies that $x_1+x_2+x_3=n_1-1$. Since $|\gcd (X_{\nu}, U')|=1$ for each $\nu \in [1,3]$, it follows that $U' = \prod_{\nu=1}^{3} ( -x_{\nu} e_1 +(-x_{\nu}y+1) e_2)$, and hence
\[
V = (-e_1-ye_2)^{n_1-1}(-e_2)^{n_2-3} V'   \,.
\]
Since
\[
V_1  = e_2^3(-e_1-ye_2)^{n_1-1}  V' = Y_1Y_2Y_3 \,,
\]
it follows that $e_2 \t Y_{\nu}$ for each $\nu \in [1,3]$. Since $(-e_1-ye_2)^{n_1-1} e_2^3$ is zero-sum free, each of the $Y_{\nu}$ is divisible by at least one of the elements from $V'$. Setting $h_3=-g_1-g_2$ and renumbering if necessary, it follows that for each $\nu \in [1,3]$
\[
Y_{\nu} = e_2 h_{\nu} (-e_1-ye_2)^{y_{\nu}} \,,
\]
where $y_1, y_2, y_3 \in \N_0$ with $y_1+y_2+y_3=n_1-1$. For each $\nu \in [1,3]$ it follows that $h_{\nu} = y_{\nu}e_1 + (y y_{\nu}-1)e_2$. Therefore we obtain that
\[
\begin{aligned}
0 & = g_1+g_2+h_3 = \Big( -x_{1} e_1 +(-x_1 y+1) e_2 \Big) + \Big( -x_{2} e_1 +(-x_2 y+1) e_2 \Big) + \Big( y_3e_1 + (y y_3-1)e_2 \Big) \\
 & = \Big( -x_1-x_2+y_3 \Big) e_1 + \Big( (-x_1-x_2+y_3) y + 1 \Big) e_2 \,,
\end{aligned}
\]
a contradiction, as not both, $-x_1-x_2+y_3$ and  $(-x_1-x_2+y_3) y + 1$, can be multiples of $n_1$.

\smallskip
\noindent
CASE 1.4.4: \ $v = n_2-4$.

Then $X_4 \t (U')^{-1}U_1$ and hence $|X_4|=n_1$. Since
\[
|X_1X_2X_3X_4|=|U_1|=|U|-v+(n_2-v)=n_1+7 \,,
\]
it follows that $|X_1X_2X_3|=7$, and hence $X_1$, $X_2$, or $X_3$ has length two. Similarly, we obtain that $Y_1$, $Y_2$, or $Y_3$ has length two, a contradiction to the earlier mentioned fact that not both, $U_1$ and $V_1$ are divisible by an atom of length two.

\smallskip
\noindent
CASE 1.4.5: \ $v \le n_2-5$.

Then $Y_4 Y_5 \t (V')^{-1}V_1$, and since $|Y_4|\ne2\ne|Y_5|$, we infer that $|Y_4|=|Y_5|=n_1$. Thus $(-Y_4)(-Y_5) \t (U')^{-1}U_1$ and $Y_4(-Y_4)Y_5(-Y_5) \t S(-S)$, a contradiction to {\bf A2}.

\bigskip
\noindent
CASE 2: \  $\mathsf{h}(S) < n_2/2$.

We distinguish two subcases, depending on the parity of $n_2$.

\bigskip
\noindent
CASE 2.1: \  $n_2$ is odd.

Since $n_2$ is odd, we have $3n_1\le n_2$ and $n_1 \ge 7$.
We write $S = \prod_{\nu=1}^k a_{\nu}^{{\alpha_{\nu}}}$ with $a_1, \ldots, a_k \in G$ pairwise distinct and $\alpha_1  \ge \ldots \ge \alpha_{k} \ge 1$.
Since $|S|= \sum_{\nu=1}^k {\alpha_{\nu}} \ge n_2 + n_1  -4$ and $\alpha_1 \le (n_2 - 1)/2$, it follows that $k \ge 3$.
We define $T_1 = S(a_1 a_2 a_3)^{-1}$ and set $T_1= \prod_{i=1}^l b_i^{{\beta_i}}$ with $b_1, \ldots, b_l \in G$ pairwise distinct and $\beta_1  \ge \ldots \ge \beta_{l}\ge 1$.
Since  $\beta_1 \le \alpha_1 \le  (n_2 - 1)/2$,  $\sum_{i=1}^k {\beta_i} = |S|-3 \ge n_2 + n_1  -7$, and $n_2-1 < n_1+n_2-7$,  it follows that $l \ge 3$.
Applying Lemma \ref{5.5} (with parameters $t=l$, $\alpha = |T_1|$, $\alpha_1'= \ldots = \alpha_t'=0$ and $\alpha_1= \ldots = \alpha_t=(n_2-1)/2$; note that we have $s+1=3$) we infer that
\[
\begin{aligned}
\prod_{\nu=1}^l (1 + \beta_{\nu} ) & \ge \left( 1 + \frac{n_2 - 1}{2} \right)^2 \big( 1+ (|T_1| - (n_2-1) \big) \\
 & \ge \left( 1 + \frac{n_2 - 1}{2} \right)^2 (1 + 1)= \frac{n_2^2 + 2n_2 + 1}{2} > n_1 n_2 \,.
\end{aligned}
\]
Thus Lemma \ref{5.4} implies that there is a $W_1 \in \mathcal A (G)$ with $|W_1| \ge 3$ such that
$(-W_1) W_1 \t (-T_1)T_1 $. Since $|W_1| < |U|$, {\bf A1} implies that $W_1 = n_1$. We write $W_1 = W_1' (-W_1'')$ with $W_1'W_1''\mid T_1$.

We define $T_2 = S(W_1'W_1'')^{-1}$ and  note that $(-S)S = W_1(-W_1)T_2(-T_2)$. Furthermore,
\[
|T_2|=|S|-n_1 \ge n_2 - 4  \quad \text{and} \quad |\supp (T_2)|\ge 3 \,.
\]
We set $T_2=\prod_{\nu=1}^m c_{\nu}^{{\gamma_{\nu}}}$ with $c_1, \ldots, c_m \in G$ pairwise distinct and $\gamma_1 \ge \ldots \ge  \gamma_m \ge 1$.
Applying Lemma \ref{5.5} (with parameters $t=m$, $\alpha = |T_2|$, $\alpha_1'= \ldots = \alpha_3'=1$, $\alpha_4' = \ldots = \alpha_t'=0$ and $\alpha_1= \ldots = \alpha_t=(n_2-1)/2$; note that we have $s+1=3$) we infer that
\[
\begin{aligned}
\prod_{\nu=1}^m (1+{\gamma_{\nu}}) & \ge \left( 1+ \frac{n_2-1}{2} \right) \left( 1+|T_2|-\left(\frac{n_2-1}{2}+1\right) \right) (1+1) \\
 & = \frac{n_2+1}{2} \frac{n_2-7}{2} 2 = \frac{1}{2}(n_2^2 - 6n_2-7) > n_1n_2 \,.
\end{aligned}
\]
Thus Lemma \ref{5.4} implies that there is a $W_2 \in \mathcal A (G)$ with $|W_2| \ge 3$ such that
$(-W_2)W_2 \t (-T_2)T_2 $. Since $|W_2| < |U|$, {\bf A1} implies that $W_2 = n_1$. Therefore we obtain that $W_1(-W_1)W_2 (-W_2) \t (-S)S$, a contradiction to {\bf A2}.

\bigskip
\noindent
CASE 2.2: \  $n_2$ is even.

We distinguish three cases; the first one is that $|U|=|V|=  n_1 + n_2 -2$ and the two others deal with the case $|U| = n_1 + n_2 -2$, further distinguishing based on the  description recalled in Lemma \ref{structure}.

\smallskip
\noindent
CASE 2.2.1: \ $|U|=|V|=  n_1 + n_2 -2$.

First we handle the case $n_2>12$. The special case $n_2=12$ will follow by the same strategy but the details will be different.

We write $S= \prod_{\nu=1}^k a_{\nu}^{{\alpha_{\nu}}}$ with ${\alpha_1} \ge \ldots \ge \alpha_{k}$.
Since $\sum_{\nu=1}^k {\alpha_{\nu}} = n_2 + n_1  -2$ and $\alpha_1 \le (n_2 - 2)/2$, it follows that $k \ge 3$.

We define $T_1= \prod_{\nu=1}^l b_{\nu}^{{\beta_{\nu}}}$ with ${\beta_1} \ge \ldots \ge \beta_{l}$
to be a subsequence of $S$ of length  $n_2 - 2$ such that $\beta_2 \le n_2/2 - 3$ and such that $T_1^{-1}S$ contains at least $4$ distinct elements
or $3$ elements with multiplicity at least $2$.
Applying Lemma \ref{5.5} (with parameters $t=l$, $\alpha= |T_1|=n_2-2$, $\alpha_1'=\ldots=\alpha_t'=0$, and $\alpha_1=(n_2-2)/2$, $\alpha_2= \ldots = \alpha_t=(n_2-6)/2$; note that we have $s+1=3$) we infer that
\[
\prod_{\nu=1}^l (1 + \beta_{\nu} ) \ge \left( 1 + \frac{n_2 - 2}{2} \right) \left(1 + \frac{n_2 - 6}{2} \right) (1 + 2)= 3\left( \frac{n_2^2}{4} - n_2 \right) > n_1n_2 \,,
\]
where the last inequality holds because $n_2 > 12$.
Thus Lemma \ref{5.4} implies that there is a $W_1 \in \mathcal A (G)$ with $|W_1| \ge 3$ such that $(-W_1)W_1 \t (-T_1)T_1 $. Since  $|W_1|< |U|$, {\bf A2} implies that
$|W_1|=n_1$. We write $W_1 = W_1' (-W_1'')$ with $W_1'W_1''\mid T_1$.

We define $T_2 = S(W_1'W_1'')^{-1} = \prod_{\nu=1}^m c_{\nu}^{\gamma_{\nu}}$ with $\gamma_1 \ge \ldots \ge \gamma_m$. We note that $T_1^{-1}S \t T_2$,  $(-S)S=W_1(-W_1)T_2(-T_2)$, and $|T_2|=n_2 - 2$. By construction of $T_1^{-1}S$, we obtain that either ($\gamma_3 \ge 2$) or ($\gamma_3 \ge 1$ and $\gamma_4 \ge 1$).
Applying  Lemma \ref{5.5} (with parameters $t=m$, $\alpha=|T_2|$, $\alpha_1=\ldots=\alpha_t=(n_2-2)/2$, and either ($\alpha_1'= \ldots = \alpha_3'=2$, $\alpha_4'=\ldots=\alpha_t'=0$) or ($\alpha_1'= \ldots = \alpha_4'=1$, $\alpha_5'=\ldots=\alpha_t'=0$) we infer that either
\[
\prod_{\nu=1}^m(1+\gamma_{\nu}) \ge \left( 1 + \frac{n_2}{2}-1 \right) \big( 1+ (|T_2|- \left(\frac{n_2}{2}-1\right)-2) \big)(1+2) = \frac{n_2}{2}\left(\frac{n_2}{2}-2 \right) 3 > n_1n_2 \,,
\]
or
\[
\prod_{\nu=1}^m(1+\gamma_{\nu}) \ge \left(1 + \frac{n_2}{2}-1\right) \left( 1+ |T_2|-\left(\frac{n_2}{2}-1\right)-2 \right) \big)(1+1)(1+1) = \frac{n_2}{2} \left( \frac{n_2}{2}-2 \right) 4 > n_1n_2 \,.
\]
Thus Lemma \ref{5.4} implies that there is a $W_2 \in \mathcal A (G)$ with $|W_2| \ge 3$ such that  $(-W_2)W_2 \t (-T_2)T_2$. Since $|W_2| < |U|$, {\bf A1} implies that $|W_2|=n_1$. Therefore we obtain that $W_1(-W_1)W_2 (-W_2) \t (-S)S$, a contradiction to {\bf A2}.

Now suppose that  $n_2 = 12$. Then $n_1 = 6$ and $|S| = 16$.
Again we set $S= \prod_{\nu=1}^k a_{\nu}^{{\alpha_{\nu}}}$ with ${\alpha_1} \ge \ldots \ge \alpha_{k}$. Since $\mathsf h (S) \le 5$, we infer  that $k \ge 4$.
We define $T_1 = S(a_1 a_2 a_3 a_4)^{-1}$ and set $T_1= \prod_{\nu=1}^l b_{\nu}^{\beta_{\nu}}$ with $\beta_1 \ge \ldots \ge \beta_{l}$. Observe that  $\beta_1 \le 4$. Applying Lemma \ref{5.5} (with parameters $t=l$, $\alpha= |T_1|=12$, $\alpha_1'=\ldots=\alpha_t'=0$, and $\alpha_1 = \ldots = \alpha_t=4$) we infer that
\[
\prod_{\nu=1}^l (1 + {\beta_{\nu}} ) \ge ( 1 + 4 )^3 >  n_1n_2 \,.
\]
Thus Lemma \ref{5.4} implies that there is a $W_1 \in \mathcal A (G)$ with $|W_1| \ge 3$ such that   $(-W_1)W_1 \t (-T_1)T_1 $. Since $|W_1|<|U|$, {\bf A1} implies that
 $|W_1|=n_1=6$. We write $W_1 = W_1' (-W_1'')$ with $W_1'W_1''\mid T_1$.

We define $T_2 = S(W_1'W_1'')^{-1} = \prod_{\nu=1}^m c_{\nu}^{\gamma_{\nu}}$ with $\gamma_1 \ge \ldots \ge \gamma_m$. We note that $|T_2|=n_2 - 2= 10$ and $m \ge 4$.
Applying Lemma \ref{5.5} (with parameters $t=m$, $\alpha= |T_2|=10$, $\alpha_1'=\ldots=\alpha_4'=1$, $\alpha_5'= \ldots = \alpha_t'=0$, and $\alpha_1 = \ldots = \alpha_t=5$) we infer that
\[
\prod_{\nu=1}^m(1+\gamma_{\nu}) \ge (1+5)(1 + 3)(1+ 1)(1+ 1) > n_1n_2 \,,
\]
and we obtain a contradiction as above.

\smallskip
\noindent
CASE 2.2.2: \ $U$ is of type I, as given in Lemma \ref{structure}.

Then
\[
\frac{n_2}{2}-1 \ge \mathsf h (S) \ge \mathsf h (U)-3 \ge \ord (e_j)-4 \,,
\]
which implies that $\ord(e_j)=n_1$ so $j=1$. We assert that
\[
\mathsf v_{e_1} (UV) + \mathsf v_{-e_1} (UV) \ge n_1+1 \,.
\]
If this holds, then Lemma 5.2 in \cite{Ge-Gr-Sc11a} implies that $\mathsf L (UV) \cap [3, n_1] \ne \emptyset$, a contradiction. Since $\mathsf v_{-e_1}(V)\ge \mathsf v_{-e_1}(-S)\ge \mathsf v_{e_1}(U)-3$, we obtain that
\[
\mathsf v_{e_1}(UV) + \mathsf v_{-e_1}(UV) \ge (n_1-1)+(n_1-1)-3 = 2n_1-5 \ge n_1+1 \,.
\]

\smallskip
\noindent
CASE 2.2.3: \ $U$ is of type II, as given in Lemma \ref{structure}.

We observe that
\[
\frac{n_2}{2}-1 \ge \mathsf h (S) \ge \mathsf h (U)-3 = \max \{ sn_1-1, n_2-sn_1+ \epsilon\} - 3\,.
\]
This implies that $s=\frac{n_2}{2n_1}$, hence $\mathsf v_{e_2} (U)=\frac{n_2}{2}+\epsilon$. Thus $\epsilon \in [1,2]$ and $e_2^{\epsilon+1} \t U'$.

Assume to the contrary that  $U'=g_1g_2(-h_1-h_2)=e_2^3$. Then $V'=(-2e_2)h_1h_2$, $|V|=\mathsf D (G)$, and
\[
\mathsf v_{-e_2}(V) \ge \mathsf v_{-e_2}(-S) = \mathsf v_{e_2}(S) = \mathsf v_{e_2}(U)-3 \ge \frac{n_2}{2}+\epsilon-3 \,.
\]
Thus $\mathsf v_{-e_2}(V) \ge 1$, which implies that $V^* = (-2e_2)^{-1}(-e_2)^{2}V \in \mathcal A (G)$, but $|V^*|=|V|+1= \mathsf D (G)+1$, a contradiction.

Since $\epsilon=2$ implies that $U'=e_2^3$, we obtain that $\epsilon=1$, $\mathsf v_{e_2}(U')=2$, $\mathsf v_{-e_2}(V')=0$, and
\[
\mathsf v_{-e_2}(V)=\mathsf v_{-e_2}(-S)=\mathsf v_{e_2}(S)=\mathsf v_{e_2}(U)-2=\frac{n_2}{2}-1\,.
\]
We consider
\[
U_1= e_2^{-\mathsf v_{e_2}(U)}(-e_2)^{\mathsf v_{-e_2}(V)}U \quad \text{and} \quad V_1= (-e_2)^{-\mathsf v_{-e_2}(V)} e_2^{\mathsf v_{e_2}(U)}V \,.
\]
Neither $U_1$ nor $V_1$ is divisible by an atom of length $2$,  and since $|V_1|=|V|+2 > \mathsf D (G)$, $V_1 \notin \mathcal A (G)$. Therefore we obtain that
\[
2 < \max \mathsf L (U_1)+\max \mathsf L (V_1) \le \frac{|U_1|}{3} + \frac{|V_1|}{3} = \frac{|UV|}{3} \le  \frac{2n_1+2n_2-2}{3} < n_2 \,,
\]
a contradiction.
\end{proof}

\section{Characterization of the system $\mathcal L (C_{n_1} \oplus C_{n_2})$} \label{6}

In this section we  prove Theorem \ref{1.1}. We start with two propositions which gather various special cases which have been settled before. The first groups, for which the Characterization Problem has been solved, are cyclic groups and elementary $2$-groups (\cite{Ge90c}). We use the characterization of groups $C_n \oplus C_n$ (\cite{Sc09c}) and \cite{B-G-G-P13a}).   The core of this section is Proposition \ref{6.5}.

\begin{proposition} \label{6.1}
Let $G$ be an abelian group such that $\mathcal L (G) = \mathcal L (C_{n_1} \oplus C_{n_2})$ where $n_1, n_2 \in \N$ with $n_1 \t n_2$ and $n_1+n_2>4$. Then $G$ is finite, and we have
\begin{enumerate}
\item $\mathsf d (G) = \mathsf d (C_{n_1}\oplus C_{n_2}) = n_1+n_2-2$ and $\exp (G)= n_2$;

\smallskip
\item If $n_1 = n_2$, then $G \cong C_{n_1} \oplus C_{n_2}$.
\end{enumerate}
\end{proposition}

\begin{proof}
1. The finiteness of $G$ and the equality of the Davenport constants follows from \cite[Proposition 7.3.1]{Ge-HK06a}. The statement on the exponents follows from  \cite[Proposition 5.2]{Sc11b} or from \cite[Proposition 5.4]{B-G-G-P13a}.

\smallskip
2. This follows from \cite[Theorem 4.1]{Sc09c}.
\end{proof}

\begin{proposition} \label{6.2}
Let   $n_1, n_2 \in \N$ with $n_1 \t n_2$ and $n_1+n_2>4$, and let $G = H \oplus C_{n_2}$   where $H \subset G$ is a subgroup with $\exp (H) \t n_2$. Suppose that $\mathcal L (G)  = \mathcal L (C_{n_1} \oplus C_{n_2})$.
\begin{enumerate}
\item $\mathsf d (H) \le n_1-1$, and if $\mathsf d^* (H) = n_1-1$, then $\mathsf d (G) = \mathsf d^* (G)$.

\smallskip
\item If $\mathsf d (G) = \mathsf d^* (G)$, then $G \cong C_{n_1} \oplus C_{n_2}$.

\smallskip
\item If $n_1 \in [1,5]$, then $G \cong C_{n_1} \oplus C_{n_2}$.

\end{enumerate}
\end{proposition}

\begin{proof}
1. Proposition \ref{6.1} implies that
\[
n_1+n_2-2 = \mathsf d (G) \ge \mathsf d (H) + (n_2-1) \quad \text{and hence} \quad \mathsf d (H) \le n_1-1 \,.
\]
If $\mathsf d^* (H) = n_1-1$, then
\[
n_1+n_2-2 = \mathsf d (G) \ge \mathsf d^* (G) = \mathsf d^* (H) + (n_2-1) = n_1+n_2-2 \,.
\]

2. This follows from \cite[Theorem 5.6]{B-G-G-P13a}.

\smallskip
3. By 2., it is sufficient to show that $\mathsf d (G)=\mathsf d^* (G)$. Suppose that $H$ is cyclic. Then $\mathsf r (G) \le 2$ and  Proposition \ref{2.3} implies that $\mathsf d (G) = \mathsf d^* (G)$. Suppose  that $H$ is noncyclic. Then $2 \le \mathsf r (H) \le \mathsf d (H) \le n_1-1$, and hence $n_1 \in [3,5]$.
Suppose that $n_1=3$. Then $\mathsf d (H)=2$ and $H \cong C_2 \oplus C_2$. Thus $\mathsf d^* (H)=2=n_1-1$, and the assertion follows from 1.
Suppose that $n_1=4$. Then $\mathsf d (H) \in [2,3]$ and $H$ is isomorphic to $C_2 \oplus C_2$ or to $C_2^3$. If $H \cong C_2^3$, then $\mathsf d^* (H)=n_1-1$, and the assertion follows from 1. Suppose that $H \cong C_2 \oplus C_2$ and set $n_2=2m$. If $m$ is even,  then $\mathsf d (G) = \mathsf d^* (G)$ by \cite[Corollary 4.2.13]{Ge09a}. If $m$ is odd, then $\mathsf d (G) = \mathsf d^* (G)$  by \cite[Theorem 3.13]{Sc11b}.
Suppose that $n_1= 5$. Then $\mathsf d (H) \in [2,4]$ and $H$ is isomorphic to one of the following groups:  $C_2^2, C_2^3, C_2^4, C_2 \oplus C_4, C_3 \oplus C_3$. If $H$ is isomorphic to one of the groups in $\{C_2^4, C_2 \oplus C_4, C_3 \oplus C_3 \}$, then $\mathsf d^* (H) = n_1-1$. If $H \cong C_2 \oplus C_2$, then $\mathsf d (G) = \mathsf d^* (G)$ as outlined above.
Suppose that  $H \cong C_2^3$. Then $G = C_2^3 \oplus C_{n_2}$ and we set $n_2=2m$. If $m$ is even, then again \cite[Corollary 4.2.13]{Ge09a} implies that $\mathsf d (G) = \mathsf d^* (G)$. If $m$ is odd, then this follows from \cite{Ba69a}.
\end{proof}

\smallskip
We need the following characterization of decomposable subsets.

\smallskip
\begin{lemma} \label{6.3}
Let $G$ be a finite abelian group and $G_0 \subset G$  a subset.
\begin{enumerate}
\item The following statements are equivalent.
      \begin{enumerate}
      \item $G_0$ is decomposable.

      \item There are nonempty subsets $G_1, G_2 \subset G_0$  such that $G_0=G_1 \uplus G_2$ and $\mathcal B(G_0)=\mathcal B(G_1)\time \mathcal B(G_2)$.

      \item There are nonempty subsets $G_1, G_2 \subset G_0$ such that $G_0=G_1 \uplus G_2$ and $\mathcal A(G_0) = \mathcal A(G_1) \uplus \mathcal A(G_2)$.

      \item There are nonempty subsets $G_1, G_2 \subset G_0$ such that $\langle G_0 \rangle = \langle G_1 \rangle \oplus \langle G_2 \rangle$.
      \end{enumerate}

\smallskip
\item There exist a uniquely determined $t \in \N$ and (up to order) uniquely determined nonempty indecomposable sets $G_1, \ldots, G_t \subset G_0$ such that
      \[
      G_0 = \biguplus_{\nu=1}^t G_{\nu} \quad \text{and} \quad \langle G_0 \rangle = \bigoplus_{\nu=1}^t \langle G_{\nu} \rangle \,.
      \]
\end{enumerate}
\end{lemma}

\begin{proof}
For 1. see \cite[Lemma 3.7]{Sc04a} and \cite[Lemma 3.2]{Ba-Ge14b}, and for 2. we refer to \cite[Proposition 3.10]{Sc04a}.
\end{proof}

\smallskip
We  need the invariant
\[
\mathsf m (G) = \max \{ \min \Delta (G_0) \mid G_0 \subset G \
\text{is a non-half-factorial subset with}  \ \mathsf k (A)\ge 1 \ \text{for all} \ A \in \mathcal A (G_0) \} \,.
\]

\begin{lemma} \label{6.4}
Let $G$ be a finite abelian group, $G_0 \subset G$ a subset with $\min \Delta (G_0) = \max \Delta^* (G)$, and let $G_0 = \bigcup_{\nu=1}^t G_{\nu}$ be the decomposition into indecomposable components. If  $\exp (G) > \mathsf m (G)+2$,   then each component $G_{\nu}$ is either half-factorial or equal to $\{-g_{\nu}, g_{\nu}\}$ for some $g_{\nu} \in G$ with $\ord (g_{\nu}) = \exp (G)$. Moreover, at least one of the components $G_{\nu}$ is not half-factorial.
\end{lemma}

\begin{proof}
See \cite[Corollary 5.2]{Sc08d}.
\end{proof}

\begin{proposition} \label{6.5}
Let $n_1, n_2 \in \N$ with  $n_1 \t n_2$ and $6 \le n_1 < n_2$, and let $G$ be a finite abelian group with $\exp (G)=n_2$ and $\mathsf d (G) = n_1+n_2-2$. Suppose that $\mathcal L (G)$ contains, for all $k \in \mathbb N$, the sets
\[
L_k = \Big\{(kn_2+3) + (n_1-2)+(n_2-2) \Big\} \cup \Big(
(2k+3) + \{0, n_1-2, n_2-2 \} + \{\nu (n_2-2) \mid \nu \in [0,k] \}  \Big)  \,.
\]
Then $G$ is isomorphic to one of the following groups{\rm \,:}
\[
C_{n_1}\oplus C_{n_2} \,, \  C_2^{s} \oplus C_{n_2} \ \text{with} \ s \in \{n_1-2,n_1-1\} \,, C_2^{n_1-4} \oplus C_4 \oplus C_{n_2} \,,
\ \text{or} \ C_2 \oplus C_{n_1-1} \oplus C_{n_2} \ \text{ with} \  2 \t (n_1-1) \t n_2   \,.
\]
\end{proposition}

\begin{proof}
We set $G = H \oplus C_{n_2}$  where  $H \subset G$ is a subgroup with $\exp (H) \t n_2$.  If $H$ is cyclic, then $\mathsf d (G) = |H|+n_2-2$ whence $|H|=n_1$ and $G \cong C_{n_1} \oplus C_{n_2}$. For the remainder of the proof  we suppose that $H$ is non-cyclic. Since $\mathsf d (H) + n_2-1 \le \mathsf d (G)=n_1+n_2-2$, it follows that $\mathsf d (H) \le n_1-1$, and hence $\exp (H) \le \mathsf D (H) \le n_1$. Since $\exp (H)=n_1$ would imply that $H$ is cyclic of order $n_1$, it follows that $\exp (H) \le n_1-1$. We have  $\mathsf r (H) \le \mathsf d (H) \le n_1-1$. If $\mathsf r (H)=n_1-1$, then $H \cong C_2^{n_1-1}$ and hence $G \cong C_2^{n_1-1} \oplus C_{n_2}$.  Thus for the remainder of the proof we suppose that $\mathsf r (H) \in [2, n_1-2]$.

We start with the following two assertions.

\begin{enumerate}
\item[{\bf A1.}\,] $\exp (G) > \mathsf m (G) + 2$.

\smallskip
\item[{\bf A2.}\,] Let $G_0 \subset G$ with $\min \Delta (G_0) = n_2-2$. Then $G_0 = \{g, -g\} \cup G_1$ where $\ord (g) = n_2$, $G_1 \subset G$ is half-factorial, and $\langle G_1 \rangle \cap \langle g \rangle  = \{0\}$.
\end{enumerate}

\smallskip
{\it Proof of} \,{\bf A1}.\, Assume to the contrary that $n_2 \le \mathsf m (G) + 2$. By \cite[Proposition 3.6]{Sc09c}, we have
\[
\mathsf m (G) \le \max \{ \mathsf r^* (G) - 1, \mathsf K (G)-1\} \,, \ \text{where $\mathsf K (G)$ is the cross number of $G$}\,.
\]
So $\mathsf r^* (G) \le \log_2 |G|$,  $\mathsf K (G) \le \frac{1}{2} + \log|G| \le \frac{1}{2} + \log_2 |G|$ by \cite[Theorem 5.5.5]{Ge-HK06a} whence $\mathsf m (G) \le - \frac{1}{2} + \log_2 |G|$.
If $H = C_{m_1} \oplus \ldots \oplus C_{m_s}$, with $s = \mathsf r (H) \ge 2$, $m_1, \ldots, m_s \in \N$,   and $1 < m_1 \t \ldots \t m_s \t n_2$, then
\[
\log_2 |H| = \sum_{i=1}^s \log_2 m_i \le \sum_{i=1}^s (m_i-1) = \mathsf d^* (H) \le \mathsf d (H) \,.
\]
Therefore we obtain that
\[
\begin{aligned}
n_2-2 & \le \mathsf m (G) \le -\frac{1}{2} + \log_2|G|=  -\frac{1}{2} + \log_2 n_2 + \log_2|H| \le -\frac{1}{2} + \log_2 n_2 + \mathsf d (H) \\
      & \le - \frac{3}{2} + \log_2 n_2 + n_1  \le - \frac{3}{2} + \log_2 n_2 + \frac{n_2}{2}
\end{aligned}
\]
and hence
\[
\frac{n_2}{2} \le \log_2 n_2 + \frac{1}{2} \,, \ \text{a contradiction to $n_2 \ge 7$. \qquad \qed{(Proof of {\bf A1})}}
\]

\smallskip
{\it Proof of} \,{\bf A2}.\,
By Lemma \ref{6.3}, $G_0$ has a decomposition into indecomposable subsets, say $G_0 = \cup_{\nu=1}^t G_{\nu}$. Proposition \ref{3.3}.2 implies that $\max \Delta^* (G)=n_2-2= \min \Delta (G_0)$. By {\bf A1} and Lemma \ref{6.4}, the sets $G_{\nu}$ have the following structure: there is an $s \in [1,t]$ such that $G_{\nu} = \{-g_{\nu}, g_{\nu}\}$ with $\ord (g_{\nu}) = n_2$ for each $\nu \in [1,s]$, and $G_{s+1}, \ldots, G_t$ are half-factorial. Now it follows that $s=1$ because, by Lemma \ref{6.3}.2,
\[
\langle G_0 \rangle = \bigoplus_{\nu=1}^t \langle G_{\nu} \rangle \subset G = H \oplus C_{n_2} \quad \text{and} \quad \exp (H) < n_2  \,. \qed{\text{\rm (Proof of {\bf A2})}}
\]

\smallskip
By assumption, for every $k \in \N$ the sets $L_k =$
\[
\Big\{(kn_2+3) + (n_1-2)+(n_2-2) \Big\} \cup \Big(
(2k+3) + \{0, n_1-2, n_2-2 \} + \{\nu (n_2-2) \mid \nu \in [0,k] \}  \Big) \in \mathcal L (G) \,.
\]
Clearly, these sets  are AAMPs with difference $n_2-2$ and period $\{0, n_1-2, n_2-2\}$ and, for all sufficiently large $k \in \mathbb N$, $L_k$ is not an AAMP with some difference $d$ which is not a multiple of $n_2-2$ (\cite[Theorem 4.2.7]{Ge-HK06a}).
Let $k \in \N$ be sufficiently large. In the course of the proof we will meet certain bounds and will assume that $k$ exceeds all of them.

We choose  $B_k \in\mathcal B (G)$ such that $\mathsf L (B_k) = L_k$.
By \cite[Proposition 9.4.9]{Ge-HK06a}, there exists an $M_1 \in \N$ (not depending on $k$) such that  $B_k = V_k S_k$, where $V_k$ and $S_k$ are zero-sum sequences with  the following properties:
\[
\min \Delta \big( \supp (V_k) \big) = n_2 -2 \quad \text{and} \quad |S_k| \le M_1 \,,
\]
(indeed in the terminology of \cite[Proposition 9.4.9]{Ge-HK06a}, we have
$V_k \in V\LK V \RK$ and $S_k \in \mathcal B (G)[\mathcal U,  V]$ for a given full almost generating set $\mathcal U$; but we do not need these additional properties).
By {\bf A2}, we obtain that
\[
\supp (V_k) = \{-g_k, g_k\} \cup A_k, \quad \text{where} \ \ord (g_k) = n_2 \ \text{and} \ A_k \subset G \ \text{is half-factorial with} \ \langle g_k \rangle \cap \langle A_k \rangle = \{0\} \,.
\]
Since for each two elements $g, g' \in G$ with $\ord (g) = \ord (g') = n_2$, there is a group automorphism $\varphi \colon G \to G$ with $\varphi (g) = g'$, and since $\mathsf L (B) = \mathsf L \big( \varphi (B) \big)$ for all $B \in \mathcal B (G)$, we may assume without restriction that there is a $g \in G$ such that $g_k = g$ for every $k \in \N$. Applying a further automorphism if necessary we may suppose that $G = H \oplus \langle g \rangle$.
We will require an additional   assertion

\begin{enumerate}
\item[{\bf A3.}\,] There exist a constant $M_2 \in \N$ (not depending on $k$), $C_k \in \mathcal B (\supp (V_k))$, and $D_k \in \mathcal B (G^{\bullet})$ with the following properties:
\begin{itemize}
\item $B_k = C_k D_k$ with $|D_k| \le M_2$,
\item For any factorization $z = W_1 \cdot \ldots \cdot W_{\gamma} \in \mathsf Z (B_k)$ with $W_1, \ldots, W_{\gamma} \in \mathcal A (G)$ there are $I, J$ such that $[1, \gamma] = I \uplus J$, $\prod_{i \in I} W_i = C_k$ and $\prod_{j \in J} W_j = D_k$.
\end{itemize}
\end{enumerate}

\smallskip
{\it Proof of} \,{\bf A3}.\, Let $z = X_1 \cdot \ldots \cdot X_{\alpha}Y_1 \cdot \ldots \cdot Y_{\beta}$ be a factorization of $B_k$, where $X_1, \ldots, X_{\alpha}, Y_1, \ldots, Y_{\beta}$ are atoms, and $Y_1, \ldots, Y_{\beta}$ are precisely those atoms which contain some element from $S_k$. Then $\beta \le |S_k| \le M_1$ and $X_1 \cdot \ldots \cdot X_{\alpha}$ divides $V_k$ (in $\mathcal B (G)$). For any element $a \in \supp (V_k)$ let $m_a (z) \in \N_0$ be maximal such that $a^{\ord (a) m_a(z)}$ divides $X_1 \cdot \ldots \cdot X_{\alpha}$. Since $\beta \le M_1$, there is a constant $M_3(z) \in \N$ (not depending on $k$) such that $\mathsf v_a (B_k) - \ord (a) m_a(z) \le M_3 (z)$.
Now we define, for each $a \in \supp (V_k)$,
\[
m_a = \min \{ m_a (z) \mid z \in \mathsf Z (B_k) \} \,, \
C_k = \prod_{a \in \supp (V_k)} a^{\ord (a) m_a}, \quad \text{and} \quad D_k = C_k^{-1} B_k \,.
\]
Since there is a constant $M_3 \in \N$ (not depending on $k$) such that $\mathsf v_a (B_k) - \ord (a) m_a \le M_3$ for all $a \in \supp (V_k)$, there is a constant $M_2 \in \N$ (not depending on $k$) such that
$|D_k| = |B_k| - |C_k| \le M_2$. \qed{\text{\rm (Proof of {\bf A3})}}

\smallskip
Since $C_k \in \mathcal B (\supp (V_k))$, $\mathsf L (C_k)$ is an arithmetical progression with difference $n_2-2$ and by {\bf A3} we have
\[
\mathsf L (B_k) = \mathsf L (C_k) + \mathsf L (D_k) = \bigcup_{m \in \mathsf L (D_k)} \big(m + \mathsf L (C_k) \big) \,.
\]
Assume to the contrary that $\mathsf L (D_k)=\{m\}$. Then $-m+L_k=-m+\mathsf L (B_k) = \mathsf L (C_k) \in \mathcal L (C_{n_2})$, a contradiction to  Proposition \ref{3.6}.2. This implies that $|\mathsf L (D_k)| > 1$. Since $\supp (C_k) \subset \supp (V_k) \subset \{-g, g\} \cup A_k$, where $A_k$ is half-factorial and $\langle g \rangle \cap \langle A_k \rangle = \{0\}$, it follows that $C_k = C_k' C_k''$, with $C_k' \in \mathcal B (\{ g, -g\} )$, $C_k'' \in \mathcal B (A_k)$, $\mathsf L (C_k) = \mathsf L (C_k') + \mathsf L (C_k'')$, and $|\mathsf L (C_k'')| = 1$. Thus, if $\mathsf L (C_k'') = \{m_k\}$, then
\[
\mathsf L (C_k) = \mathsf L (C_k' 0^{m_k}) \quad \text{and} \quad \mathsf L (B_k) = \mathsf L (C_k D_k) = \mathsf L (C_k' 0^{m_k}D_k) \,.
\]
Therefore, after changing notation if necessary, we suppose from now on that
\[
B_k = C_k D_k, \ \mathsf L (B_k) = \mathsf L (C_k) + \mathsf L (D_k), \ \text{where} \ \supp (C_k) \subset \{0, g, -g\} \ \text{and} \ D_k \in \mathcal B (G) \ \text{with} \ |D_k| \le M_2 \,.
\]

We continue with the following assertion, whose proof follows from Proposition \ref{3.8}.

\begin{enumerate}
\item[{\bf A4.}\,]
\begin{itemize}
\item Let $T \in \mathcal F (G)$ with $T \t D_k$. If $T \in \mathcal A ( \mathcal B_{\langle g \rangle} (G))$, then $\sigma (T) \in \{0,g,-g, (n_1-1)g, -(n_1-1)g \}$.

\item If $z = T_1 \cdot \ldots \cdot T_{\gamma} \in \mathsf Z_{\mathcal B_{\langle g \rangle} (G)} (D_k)$ with $T_1, \ldots, T_{\gamma} \in \mathcal A ( \mathcal B_{\langle g \rangle} (G))$, then at most one of the elements $\sigma (T_1), \ldots, \sigma (T_{\gamma})$ does not lie in $\{0, g, -g\}$.
\end{itemize}
\end{enumerate}

\smallskip
\noindent
We shall use the following notation. If $z = T_1 \cdot \ldots \cdot T_{\gamma}$ is as above, then we set
\[
\sigma (z) = \sigma (T_1) \cdot \ldots \cdot \sigma (T_{\gamma}) \in \mathcal F ( \langle g \rangle ) \,.
\]
We continue with an additional  assertion.

\begin{enumerate}
\item[{\bf A5.}\,]
\[
\mathsf L_{\mathcal B (G)} (B_k) \ \ = \bigcup_{z \in \mathsf Z_{\mathcal B_{\langle g \rangle} (G)} (D_k)} \mathsf L_{\mathcal B (\langle g \rangle)}  \big( C_k \sigma (z) \big) \,, \quad \text{where the union on}
\]
the right hand side consists of at least two distinct sets which are not contained in each other.
\end{enumerate}

\smallskip
{\it Proof of} \,{\bf A5}.\, Assume to the contrary that all sets of lengths on the right hand side are contained in one fixed set $L_1 = L_{\mathcal B (\langle g \rangle)}  \big( C_k \sigma (z^*) \big)$ with $z^* \in \mathsf Z_{\mathcal B_{\langle g \rangle} (G)} (D_k)$. Then $\mathsf L (B_k) \in \mathcal L (C_{n_2})$, a contradiction to Proposition \ref{3.6}.
To show that the set on the left side is in the union on the right side, we choose a factorization $z^* = W_1 \cdot \ldots \cdot W_{\gamma} \in \mathsf Z_{\mathcal B (G)} (B_k)$, where $W_1, \ldots, W_{\gamma} \in \mathcal A (G)$. For each $\nu \in [1, \gamma]$, we set $W_{\nu} = X_{\nu}Y_{\nu}$ where $X_{\nu}, Y_{\nu} \in \mathcal F (G)$ such that
\[
C_k = X_1 \cdot \ldots \cdot X_{\gamma} \quad \text{and} \quad D_k = Y_1 \cdot \ldots \cdot Y_{\gamma} \,.
\]
For each $\nu \in [1, \gamma]$, we have $\sigma (W_{\nu}) = 0 \in G$, hence $\sigma (Y_{\nu}) = - \sigma (X_{\nu}) \in \langle g \rangle$,  $Y_{\nu} \in \mathcal B_{\langle g \rangle} (G)$, and we choose a factorization $z_{\nu} \in \mathsf Z_{\mathcal B_{\langle g \rangle} (G)} (Y_{\nu})$. Then
\[
z = z_1 \cdot \ldots \cdot z_{\gamma}  \in \mathsf Z_{\mathcal B_{\langle g \rangle} (G)} (D_k) \,.
\]
Then, for each $\nu \in [1, \gamma]$, $W_{\nu}' = X_{\nu} \sigma (z_{\nu}) \in \mathcal A (\langle g \rangle)$ and $W_1' \cdot \ldots \cdot W_{\gamma}' = C_k \sigma (z) \in \mathcal F (\langle g \rangle)$. Therefore $z' = W_1' \cdot \ldots \cdot W_{\gamma}' \in \mathsf Z_{\mathcal B (\langle g \rangle)} \big( C_k \sigma (z) \big) $ and
\[
|z^*| = \gamma = |z'| \in \mathsf L_{\mathcal B (\langle g \rangle)}  \big( C_k \sigma (z) \big) \,.
\]
Conversely, let $z = S_1 \cdot \ldots \cdot S_{\beta} \in \mathsf Z_{\mathcal B_{\langle g \rangle} (G)} (D_k)$ and $z' = W_1' \cdot \ldots \cdot W_{\gamma}' \in \mathsf Z_{\mathcal B (\langle g \rangle)} \big( C_k \sigma (z) \big) $ be given, where $S_1, \ldots, S_{\beta} \in \mathcal A (\mathcal B_{\langle g \rangle} (G))$, $W_1', \ldots, W_{\gamma}' \in \mathcal A (\langle g \rangle)$, and we write
\[
\sigma (z) = s_1 \cdot \ldots \cdot s_{\beta}, \ \text{where} \ s_1 = \sigma (S_1), \ldots, s_{\beta} = \sigma (S_{\beta})  \,.
\]
Note that $s_1, \ldots, s_{\beta}$ satisfy the properties given in {\bf A4}.

\smallskip
\noindent
{\it Claim:} We can find a renumbering such that
\[
W_{\nu}' = s_{\nu}T_{\nu} \quad \text{with} \quad T_{\nu} \in \mathcal F (\{-g, g\}) \quad \text{for all} \ \nu \in [1, \beta] \,.
\]
\smallskip
\noindent
{\it Proof of the Claim.} We proceed in three steps.

First, we may assume without restriction that $s_1= \ldots = s_{\delta}=0$ and $0 \notin \{s_{\delta+1}, \ldots, s_{\beta}\}$. Then at least $\delta$ of the $W_1', \ldots, W_{\gamma}'$ are equal to $0$. After renumbering if necessary, we may suppose that $W_1'= \ldots = W_{\delta}'=0$, and we set $T_1= \ldots = T_{\delta} = 1 \in \mathcal F (\{-g, g\})$.

Second, suppose there is a $\nu \in [\delta+1, \beta]$ such that $s_{\nu} \in \{(n_1-1)g, n_2-(n_1-1)g\}$, say $\nu = \delta+1$. Then $s_{\delta+1}$ divides (in $\mathcal F (G)$) one element of $\{W_{\delta+1}', \ldots, W_{\gamma}'\}$, say $W_{\delta+1}'$. Then we set $T_{\delta+1} = s_{\delta+1}^{-1} W_{\delta+1}' \in \mathcal F (\{-g, g\})$.

To handle the last step, we observe that, by {\bf A4}, all remaining $s_{\nu}$ lie in $\{-g, g\}$. Since $\beta \le |D_k| \le M_2$ and the multiplicities of $g$ and of $-g$ in $C_k$ are growing with $k$, and $k$ is sufficiently large, for each $\nu \le \beta$ the product $\prod_{\lambda = \nu}^{\gamma} W_{\lambda}'$ is divisible by $g$ and by $-g$. Thus we can pick a suitable $W_{\nu}'$ and the assertion follows. \qed

Now we define
\[
W_{\nu}'' = \begin{cases} S_{\nu}T_{\nu} & \text{for each } \ \nu \in [1, \beta], \\
 W_{\nu}' & \text{for each } \ \nu \in [\beta+1, \gamma] \,.
 \end{cases}
\]
Then, by construction, we have $B_k = W_1'' \cdot \ldots \cdot W_{\gamma}'' $. Let $\nu \in [1, \beta]$. Since $W_{\nu}' \in \mathcal A (\langle g \rangle)$, it follows that $T_{\nu}' \in \mathcal F (\langle g \rangle)$ is zero-sum free. Since $S_{\nu} \in \mathcal A ( \mathcal B_{\langle g \rangle} (G))$, it follows that $W_{\nu}'' \in \mathcal A (G)$. Thus $W_1'', \ldots, W_{\beta}'' \in \mathcal A (G)$, and we have constructed a factorization of $B_k$ of length $\gamma =|z'|$.
\qed{(Proof of {\bf A5})}

\smallskip
\begin{enumerate}
\item[{\bf A6.}\,]  Let
\[
z = T_1 \cdot \ldots \cdot T_{\gamma} \in \mathsf Z_{\mathcal B_{\langle g \rangle} (G)} (D_k)\quad \text{and} \quad
z' = T_1' \cdot \ldots \cdot T_{\gamma'}' \in \mathsf Z_{\mathcal B_{\langle g \rangle} (G)} (D_k)\,,
\]
where $\gamma, \gamma' \in \N$, $T_1, \ldots , T_{\gamma} ,T_1' , \ldots , T_{\gamma'}' \in \mathcal A ( \mathcal B_{\langle g \rangle} (G) )$, and $z\ne z'$. Furthermore, let
\[
F = C_k \sigma (T_1) \cdot \ldots \cdot \sigma (T_{\gamma}) \in \mathcal B (  \langle g \rangle) \quad \text{and} \quad
F' = C_k \sigma (T_1') \cdot \ldots \cdot \sigma (T_{\gamma'}') \in \mathcal B (  \langle g \rangle) \,,
\]
and  define
\[
F = S F_1 \ \text{and} \ F' = S F_2 , \quad \text{where} \quad S, F_1, F_2 \in \mathcal F ( \langle g \rangle) \ \text{and} \ S = {\gcd}_{\mathcal F ( \langle g \rangle)} (F, F') \,.
\]
Then one of the following statements holds:
\begin{itemize}
\item[(i)] $\mathsf d (z, z') \ge n_1 - 1$.

\item[(ii)] $\{F_1, F_2\}  = \{ \big( (-g)g \big)^{v},  0^{v} \}$ with $v \in \N$. 
\end{itemize}
\end{enumerate}

\smallskip
{\it Proof of} \,{\bf A6}.\,
Note that $\gcd (F_1, F_2)=1$, $\sigma (F_1) = \sigma (F_2) = - \sigma (S)$, $C_k \t S$, and
\[
\mathsf d_{\mathsf Z (\mathcal B_{\langle g \rangle} (G))} (z, z') \ge \mathsf d_{\mathcal F ( \langle g \rangle)} (F, F') = \mathsf d_{\mathcal F ( \langle g \rangle)} (F_1, F_2) = \max \{ |F_1|, |F_2| \} \,. \tag{$*$}
\]
Since $|F|=|C_k|+|z|$, $|F_1|+|S|=|C_k|+|z|$, $|F_2|+|S|=|C_k|+|z'|$, we obtain $|F_2|-|F_1|=|z'|-|z|$ and
\[
\mathsf d (z, z') \ge \big| |z|-|z'| \big| +2 = \big| |F_1|-|F_2| \big|+2 \,. \tag{$**$}
\]
Using $(*)$ and $(**)$  we observe that   $\max \{ |F_1|, |F_2| \} \ge n_1-1$ as well as  $\big| |F_1|-|F_2| \big| \ge n_1-3$ implies (i).
To simplify the discussion, we suppose that $\max \{ |F_1|, |F_2| \} \le n_1-1$ (of course we could also assume that $\max \{ |F_1|, |F_2| \} \le n_1-2$; the slightly weaker assumption allows us to give a more complete description of $(F_1, F_2)$ without additional efforts).
Based on the structural description of $\sigma (z)$ and $\sigma (z')$ given in {\bf A4} we distinguish  four cases.

\medskip
\noindent CASE 1: \ $\sigma (z) \sigma (z') \in \mathcal F ( \{0,g,-g\})$.

We set
\[
S =  \big( g^{n_2} \big)^{k_1} \big( (-g)^{n_2} \big)^{k_2}  \big( (-g)g \big)^{k_3}  \big( \delta g \big)^{k_4} 0^{k_5} \,,
\]
where $\delta \in \{-1, 1\}$, $k_1, \ldots, k_5 \in \N_0$ and $k_3 < n_2$.
We  distinguish two cases.

\medskip
\noindent CASE 1.1: \ $F_1=1$ or $F_2=1$, say $F_2=1$.

Then $\sigma (S)=0$, $\sigma (F_1)=0$, and $k_4=0$. We have $F_1= 0^{\mathsf v_0(F_1)} \big( (-g)g \big)^{\mathsf v_g (F_1)}$ and $|F_1|>0$. Then
\[
\min \mathsf L (SF_2) = \min \mathsf L (S)=k_1+k_2+k_3+k_5 \quad \text{and} \quad \min \mathsf L (S F_1) = k_1+k_2+k_3+k_5 + \mathsf v_0 (F_1) + \mathsf v_g (F_1) - \epsilon (n_2-2)
\]
where
\[
\epsilon = \begin{cases} 0 & k_3+\mathsf v_g (F_1) < n_2 \\ 1 & \text{otherwise} \end{cases} \,.
\]
Thus $\mathsf v_0 (F_1)+\mathsf v_g (F_1)$ is congruent to $\min \mathsf L (S F_1) - \min \mathsf L (S F_2)$ modulo $n_2-2$ and hence congruent either to $0$ or to $n_1-2$ or to $(n_2-2)-(n_1-2) = n_2-n_1$ modulo $n_2-2$. Since $0 < |F_1| = \mathsf v_0 (F_1)+2\mathsf v_g (F_1) \le n_1-1$, it follows that $(\mathsf v_0 (F_1), \mathsf v_g (F_1) ) \in \{ (n_1-2, 0), \{n_1-3,1)\}$, hence $|F_1|-|F_2|= |F_1| \ge n_1-2$, and thus (i) holds.

\medskip
\noindent CASE 1.2: \ $F_1 \ne 1$ and $F_2 \ne 1$.

By symmetry we may suppose that $0 \nmid F_1$. Then $g \t F_1$ or $(-g) \t F_1$, and by symmetry we may suppose that $g \t F_1$. We distinguish two cases.

\smallskip
\noindent CASE 1.2.1: \ $(-g) \t F_1$.

Then $F_2 = 0^{\mathsf v_0 (F_2)}$, and hence $\sigma (S) = 0 = \sigma (F_1)$. This implies $k_4= 0$ and $F_1 = \big( (-g) g \big)^{\mathsf v_g (F_1)}$. Then
\[
\min \mathsf L (S F_2) = k_1+k_2+k_3+\mathsf v_0 (F_2)+k_5 \quad \text{and} \quad
\min \mathsf L (S F_1) = k_1+k_2+k_3+\mathsf v_g (F_1)-\epsilon (n_2-2) + k_5
\]
where $\epsilon \in \{0,1\}$. Thus $\mathsf v_0 (F_2)-\mathsf v_g (F_1)$ is congruent to $\min \mathsf L (S F_1) - \min \mathsf L (S F_2)$ modulo $n_2-2$ and hence congruent either to $0$ or to $n_1-2$ or to $n_2-n_1$ modulo $n_2-2$. This implies that
either
\[
\mathsf v_0 (F_2) = \mathsf v_g (F_1) \quad \text{or} \quad \mathsf v_0 (F_2) = \mathsf v_g (F_1)+n_1-2 \quad \text{or} \quad \mathsf v_g (F_1) = \mathsf v_0 (F_2) + n_1-2 \,.
\]

If $\mathsf v_0 (F_2) = \mathsf v_g (F_1)+n_1-2$, then $\mathsf v_g (F_1) \ge 1$ implies that  $|F_2| \ge \mathsf v_0 (F_2) \ge  n_1-1$, and hence (i) holds.

If $\mathsf v_g (F_1) = \mathsf v_0 (F_2) + n_1-2$, then $\mathsf v_0 (F_2) \ge 1$ implies that $\mathsf v_g (F_1) \ge n_1-1$ whence $|F_1| =2 \mathsf v_g (F_1) \ge 2(n_1-1) > n_1$, a contradiction.
If $\mathsf v_0 (F_2) = \mathsf v_g (F_1)$, then (ii) holds.

\smallskip
\noindent CASE 1.2.2: \ $(-g) \nmid F_1$.

Then $F_1= g^{\mathsf v_g (F_1)}$ and $F_2 = (-g)^{\mathsf v_{-g} (F_2)} 0 ^{\mathsf v_0 (F_2)}$. Note that $\mathsf v_g (F_1)+ \mathsf v_{-g}(F_2)>0$, $\mathsf v_g (F_1), \mathsf v_{-g}(F_2) \in [0, n_1-1]$, and $n_2 \ge 2n_1$. However, $\sigma (F_1)=\sigma (F_2)$ implies that $\mathsf v_g (F_1)+\mathsf v_{-g} (F_2) \equiv 0 \mod n_2$, a contradiction.

\medskip
\noindent CASE 2: \ $\sigma (z)\sigma (z') \in \big((n_1-1)g \big) \mathcal F (\{0,-g,g\})$ \ or \ $\sigma (z)\sigma (z') \in \big(-(n_1-1)g \big) \mathcal F (\{0,-g,g\})$.

After applying the group automorphism which sends each $h \in G$ onto its negative, if necessary, we may suppose that $\sigma (z)\sigma (z') \in \big((n_1-1)g \big) \mathcal F (\{0,-g,g\})$.
After exchanging $z$ and $z'$, if necessary, we may suppose that $\sigma (z) \in  \mathcal F (\{0,-g,g\})$ and $\sigma (z')  \in \big((n_1-1)g \big) \mathcal F (\{0,-g,g\})$.
We set
\[
S =  \big( g^{n_2} \big)^{k_1} \big( (-g)^{n_2} \big)^{k_2}  \big( (-g)g \big)^{k_3}  \big( \delta g \big)^{k_4} 0^{k_5} \,,
\]
where $\delta \in \{-1, 1\}$, $k_1, \ldots, k_5 \in \N_0$ and $k_3 < n_2$. If $\sigma (F_1)=0$, then $\sigma (F_2)=0$ and hence $|F_2| \ge n_1$, a contradiction. Thus it follows that $\sigma (F_1) \ne 0$, and hence there are the following three cases.

\medskip
\noindent CASE 2.1: \ $g \t F_1$ and $(-g) \t F_1$.

It follows that $F_2= \big( (n_1-1)g \big) 0^{\mathsf v_0 (F_2)}$ and hence $\sigma (F_2) = (n_1-1)g = \sigma (F_1) = \big( \mathsf v_g (F_1) - \mathsf v_{-g} (F_1) \big) g$, a contradiction to $|F_1| \le n_1-1$.

\medskip
\noindent CASE 2.2: \ $g \t F_1$ and $(-g) \nmid F_1$.

Then $F_1 = g^{\mathsf v_g (F_1)} 0^{\mathsf v_0 (F_1)}$, $F_2 = \big( (n_1-1)g \big) (-g)^{n_1-1 - \mathsf v_g (F_1)} 0^{\mathsf v_0 (F_2)}$, and we can write $SF_1$ and $SF_2$ as follows:
\[
\begin{aligned}
SF_1 & = \big( g^{n_2} \big)^{l_1} \big( (-g)^{n_2}\big)^{l_2} \big( (-g)g \big)^{l_3} 0^{l_4 + \mathsf v_0 (F_1)} \\
SF_2 & = \big( g^{n_2} \big)^{l_1} \big( (-g)^{n_2}\big)^{l_2} \big( (-g)g \big)^{l_3- \mathsf v_g (F_1)} 0^{l_4 + \mathsf v_0 (F_2)} \Big( \big((n_1-1)g\big) (-g)^{n_1-1} \Big)
\end{aligned}
\]
where $l_1, \ldots, l_4 \in \N_0$, $l_4 = \mathsf v_0 (S)$, and $l_3 \ge \mathsf v_g (F_1)$ (the last inequality holds because $k$ is large enough). Therefore
\[
m_1  = l_1+l_2+l_3+l_4+\mathsf v_0 (F_1) \in L_k \ \ \text{and} \ \
m_2  = l_1+l_2+l_3-\mathsf v_g (F_1) + l_4+\mathsf v_0 (F_2) + 1 \in L_k
\]
which implies that $m_1-m_2 = \mathsf v_0 (F_1)+\mathsf v_g (F_1)-\mathsf v_0 (F_2)-1$ is congruent to either $0$ or to $n_1-2$ or to $n_2-n_1$ modulo $n_2-2$. We distinguish three cases.

\smallskip
\noindent CASE 2.2.1: \  $\mathsf v_0 (F_1)+\mathsf v_g (F_1) \equiv \mathsf v_0 (F_2)+1 \mod n_2-2$.

Since $|F_1| < n_1$ and $|F_2| < n_1$, it follows that $\mathsf v_0 (F_1)+\mathsf v_g (F_1) = \mathsf v_0 (F_2)+1$. Since
\[
|F_2|  = \mathsf v_0 (F_2)+1+ \big( n_1-1-\mathsf v_g (F_1) \big)
  = \mathsf v_0 (F_1) + \mathsf v_g (F_1) + \big( n_1-1-\mathsf v_g (F_1) \big) = n_1-1+\mathsf v_0 (F_1) \,,
\]
it follows that $|F_2| \ge n_1-1$ and hence (i) holds.

\smallskip
\noindent CASE 2.2.2: \ $\mathsf v_0 (F_1)+\mathsf v_g (F_1) \equiv \mathsf v_0 (F_2)+n_1-1 \mod n_2-2$.

Similarly,  we obtain that $\mathsf v_0 (F_1)+\mathsf v_g (F_1) = \mathsf v_0 (F_2)+n_1-1$. Thus $|F_1| \ge n_1-1$ and hence (i) holds.

\smallskip
\noindent CASE 2.2.3: \ $\mathsf v_0 (F_1)+\mathsf v_g (F_1) \equiv n_2-n_1+ \mathsf v_0 (F_2)+1 \mod n_2-2$.

We obtain that $\mathsf v_0 (F_1)+\mathsf v_g (F_1) = -(n_1-3)+\mathsf v_0 (F_2)$ which implies that $\mathsf v_0 (F_2) > n_1-3$. Therefore
\[
|F_2|  = \mathsf v_0 (F_2)+1+n_1-1- \mathsf v_g (F_1) = n_1+\mathsf v_0 (F_1)+(n_1-3)-\mathsf v_0 (F_2)+\mathsf v_0(F_2)
  = 2n_1-3+\mathsf v_0 (F_1) \ge n_1 \,,
\]
a contradiction.

\medskip
\noindent CASE 2.3: \  $g \nmid F_1$ and $(-g) \t F_1$.

Then
\[
F_1= (-g)^{\mathsf v_{-g} (F_1)} 0^{\mathsf v_0 (F_1)} \quad \text{ and} \quad
F_2 = \big( (n_1-1)g \big) g^{n_2-(n_1-1)-\mathsf v_{-g} (F_1)} 0^{\mathsf v_0 (F_2)} \,.
\]
We can write $SF_1$ and $SF_2$ as
\[
\begin{aligned}
SF_1 & = \big( g^{n_2} \big)^{l_1} \big( (-g)^{n_2}\big)^{l_2} \big( (-g)g \big)^{l_3} 0^{l_4 + \mathsf v_0 (F_1)} \quad \text{and}\\
SF_2 & = \big( g^{n_2} \big)^{l_1} \big( (-g)^{n_2}\big)^{l_2} \big( (-g)g \big)^{l_3- \mathsf v_{-g} (F_1)} 0^{l_4 + \mathsf v_0 (F_2)} \Big(   g^{n_2-(n_1-1) } \big( (n_1-1)g \big) \Big)
\end{aligned}
\]
where $l_1, \ldots, l_4 \in \N_0$, $l_4 = \mathsf v_0 (S)$, and $l_3 \ge \mathsf v_{-g} (F_1)$ (the last inequality holds because $k$ is large enough). Therefore
\[
m_1 = l_1+l_2+l_3+l_4+\mathsf v_0 (F_1) \in L_k \ \ \text{and} \ \ m_2 = l_1+l_2+l_3-\mathsf v_{-g} (F_1)+l_4+\mathsf v_0 (F_2)+1 \in L_k \,.
\]
which implies that $m_1-m_2 = \mathsf v_0 (F_1)+\mathsf v_{-g} (F_1)-\mathsf v_0 (F_2)-1$ is congruent to either $0$ or to $n_1-2$ or to $n_2-n_1$ modulo $n_2-2$. We distinguish three cases.

\smallskip
\noindent CASE 2.3.1: \  $\mathsf v_0 (F_1)+\mathsf v_{-g} (F_1) \equiv \mathsf v_0 (F_2)+1 \mod n_2-2$.

We obtain that $\mathsf v_0 (F_1)+\mathsf v_{-g} (F_1) = \mathsf v_0 (F_2)+1$ and hence
\[
|F_2|=1+\mathsf v_0 (F_2)+n_2-(n_1-1)-\mathsf v_{-g}(F_1) = \mathsf v_0 (F_1)+n_2-(n_1-1) \ge n_1 \,, \ \ \text{a contradiction.}
\]

\smallskip
\noindent CASE 2.3.2: \ $\mathsf v_0 (F_1)+\mathsf v_{-g} (F_1) \equiv \mathsf v_0 (F_2)+n_1-1 \mod n_2-2$.

We obtain that $\mathsf v_0 (F_1)+\mathsf v_{-g} (F_1) = \mathsf v_0 (F_2)+n_1-1$. Therefore $|F_1| \ge n_1-1$ and hence (i) holds.

\smallskip
\noindent CASE 2.3.3: \ $\mathsf v_0 (F_1)+\mathsf v_{-g} (F_1) \equiv n_2-n_1+ \mathsf v_0 (F_2)+1 \mod n_2-2$.

We obtain that $\mathsf v_0 (F_1)+\mathsf v_{-g} (F_1) =  \mathsf v_0 (F_2) - n_1+3$ and therefore
\[
|F_2|  = \mathsf v_0 (F_2) +1 + n_2-(n_1-1)-\mathsf v_{-g} (F_1) = \mathsf v_0(F_1)+n_1-3+1+n_2-(n_1-1)
 = \mathsf v_0 (F_1)-1+n_2 \ge n_1 \,,
\]
a contradiction.

\medskip
\noindent CASE 3: \ $\sigma (z)\sigma (z') \in \big((n_1-1)g \big)^2 \mathcal F (\{0,-g,g\})$ \ or \ $\sigma (z)\sigma (z') \in \big(-(n_1-1)g \big)^2 \mathcal F (\{0,-g,g\})$.

After applying the group automorphism which sends each $h \in G$ onto its negative, if necessary, we may suppose that $\sigma (z)\sigma (z') \in \big((n_1-1)g \big)^2 \mathcal F (\{0,-g,g\})$, whence $\sigma (z) \in \big((n_1-1)g \big) \mathcal F (\{0,-g,g\})$ and $\sigma (z') \in \big((n_1-1)g \big) \mathcal F (\{0,-g,g\})$.

We set
\[
S =  \big( (n_1-1)g \big) \big( g^{n_2} \big)^{k_1} \big( (-g)^{n_2} \big)^{k_2}  \big( (-g)g \big)^{k_3}  \big( \delta g \big)^{k_4} 0^{k_5} \,,
\]
where $\delta \in \{-1, 1\}$, $k_1, \ldots, k_5 \in \N_0$ and $k_3 < n_2$. We distinguish two cases.

\medskip
\noindent CASE 3.1: \ $F_1=1$ or $F_2=1$, say $F_2=1$.

Since $F_2=1$, it follows that $\mathsf L (S) \subset L_k$. Let $l_1 \in \mathsf L (F_1)$. Then $l_1+\mathsf L (S) \subset L_k$ and hence $l_1$ is congruent either to $0$ or to $n_1-2$ or to $n_2-n_1$ modulo $n_2-2$. Since $l_1> 0$, it follows that $|F_1|-|F_2|=|F_1| \ge n_1-2$, and hence (i) holds.

\medskip
\noindent CASE 3.2: \ $F_1 \ne 1$ and  $F_2 \ne 1$.

We have  $0 \nmid F_1$ or  $0 \nmid F_2$, say $0 \nmid F_1$. Then $g \t F_1$ or $(-g) \t F_1$. We distinguish three cases.

\smallskip
\noindent CASE 3.2.1: \ $g \t F_1$ and $(-g) \nmid F_1$.

Then $F_1= g^{\mathsf v_g (F_1)}$ and $F_2 = (-g)^{\mathsf v_{-g} (F_2)} 0^{\mathsf v_0(F_2)}$. Since $\sigma (F_1)=\sigma (F_2)$, it follows that $\mathsf v_g (F_1)+\mathsf v_{-g} (F_2) \equiv 0 \mod n_2$, and hence
\[
\max \{ \mathsf v_g (F_1), \mathsf v_{-g} (F_2) \} \ge \frac{n_2}{2} \ge n_1 \,,
\]
a contradiction.

\smallskip
\noindent CASE 3.2.2: \ $g \t F_1$ and $(-g) \t F_1$.

Then $(-g)g \t F_1$ whence $F_2=0^{\mathsf v_0 (F_2)}$. This implies that $0 = \sigma (F_2)=\sigma (F_1)$ and thus $F_1= \big( (-g)g \big)^{\mathsf v_g (F_1)}$. As above it follows that $\mathsf v_g (F_1)-\mathsf v_0 (F_2)$ is congruent either to $0$ or to $n_1-2$ or to $n_2-n_1$ modulo $n_2-2$.

If $\mathsf v_g (F_1)=\mathsf v_0 (F_2)$, then (ii) holds.

If $\mathsf v_g (F_1)=\mathsf v_0 (F_2)+n_1-2$, then $|F_1|=2\mathsf v_g (F_1) \ge 2n_1-4 \ge n_1$, a contradiction.

Suppose that $\mathsf v_g (F_1)-\mathsf v_0 (F_2) \equiv n_2-n_1 \mod n_2-2$. Then $|F_1|<n_1$ implies that $\mathsf v_g (F_1)-\mathsf v_0 (F_2)=-n_1+2$. Since $\mathsf v_g (F_1) \ge 1$, it follows that  $|F_2| \ge \mathsf v_0 (F_2) \ge n_1-1$, and hence (i) holds.

\smallskip
\noindent CASE 3.2.3: \ $g \nmid F_1$ and $(-g) \t F_1$.

Then $F_1= (-g)^{\mathsf v_{-g} (F_1)}$ and $F_2 = g^{\mathsf v_{g} (F_2)} 0^{\mathsf v_0(F_2)}$. Since $\sigma (F_1)=\sigma (F_2)$, it follows that $\mathsf v_{-g} (F_1)+\mathsf v_{g} (F_2) \equiv 0 \mod n_2$, and hence
\[
\max \{ \mathsf v_{-g} (F_1), \mathsf v_{g} (F_2) \} \ge \frac{n_2}{2} \ge n_1 \,,
\]
a contradiction.

\medskip
\noindent CASE 4: \ $\sigma (z)\sigma (z') \in \big((n_1-1)g \big) \big(-(n_1-1)g \big) \mathcal F (\{0,-g,g\})$.

After exchanging $z$ and $z'$ if necessary we may suppose that $\sigma (z) \in \big((n_1-1)g \big)  \mathcal F (\{0,-g,g\})$ and $\sigma (z') \in \big(-(n_1-1)g \big)  \mathcal F (\{0,-g,g\})$.
We set
\[
SF_1 = \big( g^{n_2})^{l_1} \big( (-g)^{n_2} \big)^{l_2} \big( (g(-g) \big)^{l_3} \big( (n_1-1)g (-g)^{n_1-1} \big) 0^{l_4 + \mathsf v_0 (F_1)}
\]
and
\[
SF_2 = \big( g^{n_2})^{l_1'} \big( (-g)^{n_2} \big)^{l_2'} \big( (g(-g) \big)^{l_3'} \Big( \big(-(n_1-1)g\big) g^{n_1-1} \Big) 0^{l_4 + \mathsf v_0 (F_2)}
\]
where $l_1, l_1', \ldots, l_3, l_3', l_4 \in \N_0$.
Since
\[
F_1 = \big((n_1-1)g \big) g^{\mathsf v_g (F_1)} (-g)^{\mathsf v_{-g} (F_1)} 0^{\mathsf v_0 (F_1)} \quad \text{and} \quad F_2 = \big(-(n_1-1)g \big) g^{\mathsf v_g (F_2)} (-g)^{\mathsf v_{-g} (F_2)} 0^{\mathsf v_0 (F_2)} \,,
\]
it follows that
\[
\big( n_1-1+ \mathsf v_g (F_1) - \mathsf v_{-g} (F_1) \big) g = \sigma (F_1) = \sigma (F_2) = \big( -n_1+1+\mathsf v_g (F_2) - \mathsf v_{-g} (F_2) \big) g
\]
and hence
\[
2n_1 - 2 \equiv \big( \mathsf v_g (F_2) - \mathsf v_g (F_1) \big) + \big( \mathsf v_{-g}(F_1) - \mathsf v_{-g} (F_2) \big) \mod n_2 \,.
\]
We distinguish four cases.

\medskip
\noindent CASE 4.1: \ $g \t F_1$ and $(-g) \t F_1$.

Then $\mathsf v_g (F_2)=0=\mathsf v_{-g} (F_2)$ and hence
\[
2n_1-2 \equiv - \mathsf v_g (F_1) + \mathsf v_{-g} (F_1) \mod n_2 \,.
\]
If $n_2 \ge 3n_1$, then $|F_1| \ge n_1$, a contradiction. Thus $n_2 = 2n_1$, $\mathsf v_{-g} (F_1) +2 \equiv \mathsf v_g (F_1) \mod n_2$, and so $\mathsf v_{-g} (F_1) +2 = \mathsf v_g (F_1)$. Therefore we obtain that
\[
SF_2= \big( g^{n_2})^{l_1} \big( (-g)^{n_2} \big)^{l_2} \big( (g(-g) \big)^{l_3 -\mathsf v_{g}(F_1)-(n_1-1)}(-g)^{n_2} \Big( \big(-(n_1-1)g\big) g^{n_1-1} \Big) 0^{l_4 + \mathsf v_0 (F_2)}
\]
Therefore
\[
m_1= l_1+l_2+l_3+l_4+1+\mathsf v_0 (F_1) \in L_k \ \text{and} \
m_2= l_1+l_2+l_3+l_4-(\mathsf v_g (F_1)+n_1-1)+2+\mathsf v_0 (F_2) \in L_k
\]
which implies that $m_1-m_2=\mathsf v_0 (F_1)-\mathsf v_0 (F_2)+\mathsf v_g (F_1)+n_1-2$ is congruent to either $0$ or to $n_1-2$ or to $n_2-n_1$ modulo $n_2-2$. We distinguish three cases.

\smallskip
\noindent CASE 4.1.1: \ $\mathsf v_0 (F_1)+\mathsf v_g (F_1)+n_1 \equiv \mathsf v_0 (F_2) + 2 \mod n_2-2$.

The left and the right hand side cannot be equal, since $\mathsf v_g (F_1) \ge 2$ would imply that $|F_2| \ge \mathsf v_0(F_2) \ge n_1$. Therefore we have
\[
\mathsf v_0 (F_1)+\mathsf v_g (F_1)+n_1 = \mathsf v_0 (F_2) + n_2
\]
and thus $|F_1| \ge \mathsf v_0 (F_1)+\mathsf v_g (F_1) \ge n_2-n_1=n_1$, a contradiction.

\smallskip
\noindent CASE 4.1.2: \ $\mathsf v_0 (F_1)+\mathsf v_g (F_1)+n_1 \equiv \mathsf v_0 (F_2) + n_1 \mod n_2-2$.

This implies that $\mathsf v_0 (F_1)+\mathsf v_g (F_1) = \mathsf v_0 (F_2)$ whence $\mathsf v_0 (F_2) \ge \mathsf v_g (F_1)\ge 2$, $\mathsf v_0 (F_1)=0$, and $\mathsf v_g (F_1) = \mathsf v_0(F_2)$. Therefore we obtain $F_1 = \big((n_1-1)g \big) g^{\mathsf v_0 (F_2)} (-g)^{\mathsf v_{0} (F_2)-2}$ and $F_2 = \big(-(n_1-1)g \big)  0^{\mathsf v_0 (F_2)}$.
Now consider a factorization $z_1$ of $SF_1$ which is divisible by the atom $X =  \big( (n_1-1)g\big) g^{n_1+1}$ and by $\big( g(-g) \big)^{\mathsf v_0 (F_2)-2}$. It gives rise to a factorization
\[
z_2= z_1 X^{-1} \big( g(-g) \big)^{-(\mathsf v_0 (F_2)-2)} \Big( \big( -(n_1-1)g \big)g^{n_1-1}\Big) 0^{\mathsf v_0 (F_2)} \in \mathsf Z (S F_2)
\]
of length $|z_2|=|z_1|-(1+\mathsf v_0(F_2)-2)+1+\mathsf v_0(F_2)=|z_1|+2$. Since $n_1 \ge 5$ and $\min \Delta (L_k) = \min \{n_1-2, n_2-n_1\} \ge 3$, $L_k$ cannot contain the lengths $|z_1|$ and $|z_1|+2=|z_2|$, a contradiction.

\smallskip
\noindent CASE 4.1.3: \ $\mathsf v_0 (F_1)+\mathsf v_g (F_1) \equiv \mathsf v_0 (F_2) + 2 \mod n_2-2$.

This implies that $\mathsf v_0 (F_1)+\mathsf v_g (F_1) = \mathsf v_0 (F_2) + 2$. Since $\mathsf v_g (F_1) \ge 3$, it follows that $\mathsf v_0 (F_2)>0$ and hence $\mathsf v_0 (F_1)=0$.
Therefore we obtain $F_1 = \big((n_1-1)g \big) g^{\mathsf v_0 (F_2)+2} (-g)^{\mathsf v_{0} (F_2)}$ and $F_2 = \big(-(n_1-1)g \big)  0^{\mathsf v_0 (F_2)}$.
Now consider a factorization $z_2$ of $SF_2$ which is divisible by the atom $X= \big(-(n_1-1)g\big) (-g)^{n_1+1}$. It gives rise to a factorization
\[
z_1 = z_2X^{-1}0^{-\mathsf v_0 (F_2)}  \big( (n_1-1)g (-g)^{n_1-1} \big) \big( (-g)g \big)^{\mathsf v_0(F_2)+2}
\]
of length $|z_1|=|z_2|+2$, a contradiction.

\medskip
\noindent CASE 4.2: \ $g \t F_1$ and $(-g) \nmid F_1$.

Then $\mathsf v_g (F_2)=0=\mathsf v_{-g}(F_1)$, hence
\[
n_2-2n_1+2 \equiv \mathsf v_g (F_1)+ \mathsf v_{-g} (F_2) \mod n_2
\]
and thus  $n_2-2n_1+2 = \mathsf v_g (F_1)+ \mathsf v_{-g} (F_2)$. Furthermore, we obtain that
\[
\max \{ \mathsf v_g (F_1), \mathsf v_{-g} (F_2) \} \ge \frac{n_2-2n_1+2}{2} \,,
\]
and hence $n_2 \in \{2n_1, 3n_2\}$. We obtain that
\[
SF_2 = \big( g^{n_2})^{l_1} \big( (-g)^{n_2} \big)^{l_2} \big( (g(-g) \big)^{l_3-\mathsf v_g (F_1)-(n_1-1)} (-g)^{n_2} \Big( \big(-(n_1-1)g\big) g^{n_1-1} \Big) 0^{l_4 + \mathsf v_0 (F_2)}
\]
Therefore $m_1= l_1+l_2+l_3+l_4+1+\mathsf v_0 (F_1) \in L_k$ and
\[
m_2= l_1+l_2+l_3+l_4-(\mathsf v_g (F_1)+n_1-1)+2+\mathsf v_0 (F_2) \in L_k
\]
which implies that $m_1-m_2=\mathsf v_0 (F_1)-\mathsf v_0 (F_2)+\mathsf v_g (F_1)+n_1-2$ is congruent to either $0$ or to $n_1-2$ or to $n_2-n_1$ modulo $n_2-2$. We distinguish three cases.

\smallskip
\noindent CASE 4.2.1: \ $\mathsf v_0 (F_1)+\mathsf v_g (F_1)+n_1 \equiv \mathsf v_0 (F_2) + 2 \mod n_2-2$.

The left and the right hand side cannot be equal, because otherwise we would have $|F_2| \ge \mathsf v_0(F_2) +1 \ge n_1$. Therefore we have
\[
\mathsf v_0 (F_1)+\mathsf v_g (F_1)+n_1 = \mathsf v_0 (F_2) + n_2
\]
and thus $|F_1| \ge \mathsf v_0 (F_1)+\mathsf v_g (F_1) \ge n_2-n_1 \ge n_1$, a contradiction.

\smallskip
\noindent CASE 4.2.2: \ $\mathsf v_0 (F_1)+\mathsf v_g (F_1)+n_1 \equiv \mathsf v_0 (F_2) + n_1 \mod n_2-2$.

This implies that $\mathsf v_0 (F_1)+\mathsf v_g (F_1) = \mathsf v_0 (F_2)$ whence $\mathsf v_0 (F_2) \ge \mathsf v_g (F_1)\ge 1$, $\mathsf v_0 (F_1)=0$, and $\mathsf v_g (F_1) = \mathsf v_0(F_2)$. Therefore we obtain $F_1 = \big((n_1-1)g \big) g^{\mathsf v_0 (F_2)} $ and $F_2 = \big(-(n_1-1)g \big) (-g)^{n_2-2n_1+2-\mathsf v_0(F_2)}  0^{\mathsf v_0 (F_2)}$ and hence
$|F_2|=1+n_2-2n_1+2$ which implies that $n_2=2n_1$ and $|F_2|=3$.  Thus $\mathsf v_0 (F_2) \in \{1,2\}$.

Suppose that $\mathsf v_0 (F_2)=1$. Then $\mathsf v_g (F_1)=1$, $F_1 = \big((n_1-1)g \big) g$,  and  $F_2= \big(-(n_1-1)g \big)(-g)0$. Consider a factorization $z_1$ of $SF_1$ divisible by $X= \big( (n_1-1)g \big) g^{n_1+1}$. This gives rise to a factorization
\[
z_2 = z_1 X^{-1} 0 \big( (-g)g \big) \Big( \big(-(n_1-1)g \big) g^{n_1-1} \Big)
\]
of length $|z_2|=|z_1|+2$, a contradiction.

Suppose that $\mathsf v_0 (F_2)=2$. Then $\mathsf v_g (F_1)=2$, $F_1= \big( (n_1-1)g\big) g^2$, and $F_2 = \big(-(n_1-1)g\big) 0^2$. Consider a factorization $z_1$ of $SF_1$ divisible by $X= \big( (n_1-1)g \big) g^{n_1+1}$. This gives rise to a factorization
\[
z_2 = z_1 X^{-1} 0^2 \Big( \big(-(n_1-1)g \big) g^{n_1-1} \Big)
\]
of length $|z_2|=|z_1|+2$, a contradiction.

\smallskip
\noindent CASE 4.2.3: \ $\mathsf v_0 (F_1)+\mathsf v_g (F_1) \equiv \mathsf v_0 (F_2) + n_2-2n_1+2 \mod n_2-2$.

Suppose that  $n_2=3n_1$. Then $\mathsf v_0 (F_1)+\mathsf v_g (F_1) \equiv \mathsf v_0 (F_2) + n_1+2 \mod n_2-2$, and equality cannot hold because $|F_1| \ge \mathsf v_0 (F_1)+\mathsf v_g (F_1)$. This implies that $(n_2-2) + \mathsf v_0 (F_1)+\mathsf v_g (F_1) = \mathsf v_0 (F_2) + n_1+2$ and hence $2n_1-4+\mathsf v_0 (F_1)+\mathsf v_g (F_1) = \mathsf v_0 (F_2)$, a contradiction to $\mathsf v_0 (F_2) \le |F_2| \le n_1-1$.

This implies that $n_2=2n_1$ and $\mathsf v_0 (F_1)+\mathsf v_g (F_1) = \mathsf v_0 (F_2) + 2$. Since $2=\mathsf v_g (F_1)+\mathsf v_{-g}(F_2)$, we infer that $\mathsf v_g (F_1) \in [1,2]$.

Suppose $\mathsf v_g (F_1)=2$. Then $\mathsf v_{-g}(F_2)=0$ and $\mathsf v_0 (F_1)=\mathsf v_0 (F_2)=0$, and we have $F_1= \big( (n_1-1)g \big) g^2$ and $F_2 = \big( -(n_1-1)g \Big)$.
Consider a factorization $z_2$ of $SF_2$ containing the atom $X = \big( -(n_1-1)g\big) (-g)^{n_1+1}$. This gives rise to a factorization
\[
z_1= z_2 X^{-1} \Big( \big( (n_1-1)g \big) (-g)^{n_1-1} \Big) \big( (-g)g \big)^2
\]
of length $|z_1|=|z_2|+2$, a contradiction.

Suppose $\mathsf v_g (F_1)=1$. Then $\mathsf v_{-g}(F_2)=1$, $\mathsf v_0(F_1)=1$, $\mathsf v_0 (F_2)=0$, and we have $F_1= \big( (n_1-1)g \big) g 0$ and $F_2 = \big( -(n_1-1)g \Big) (-g)$.
Consider a factorization $z_2$ of $SF_2$ containing the atom $X = \big( -(n_1-1)g\big) (-g)^{n_1+1}$. This gives rise to a factorization
\[
z_1= z_2 X^{-1} \Big( \big( (n_1-1)g \big) (-g)^{n_1-1} \Big) \big( (-g)g \big) 0
\]
of length $|z_1|=|z_2|+2$, a contradiction.

\medskip
\noindent CASE 4.3: \  $g \nmid F_1$ and $(-g) \t F_1$.

Then $\mathsf v_g (F_1)=0$ and $\mathsf v_{-g}(F_2)=0$ and hence
\[
2n_1-2 \equiv \mathsf v_g (F_2) + \mathsf v_{-g} (F_1) \mod n_2 \,.
\]
This implies that $\mathsf v_g (F_2) = \mathsf v_{-g} (F_1) =n_1-1$ and hence $|F_1| \ge n_1$ and $|F_2| \ge n_1$, a contradiction.

\medskip
\noindent CASE 4.4: \  $g \nmid F_1$ and $(-g) \nmid F_1$.

Then $\mathsf v_g (F_1)=0=\mathsf v_{-g}(F_1)$ and hence
\[
2n_1-2 \equiv \mathsf v_g (F_2) - \mathsf v_{-g} (F_2) \mod n_2 \,.
\]
If $n_2 \ge 3n_1$, then $|F_2| \ge n_1$, a contradiction. Thus $n_2=2n_1$ and hence $\mathsf v_g (F_2)= \mathsf v_{-g}(F_2)-2$.
Therefore we obtain that
\[
SF_2= \big( g^{n_2})^{l_1} \big( (-g)^{n_2} \big)^{l_2} \big( (g(-g) \big)^{l_3 + \mathsf v_{g}(F_2)-(n_1-1)}(-g)^{n_2} \big( ((-n_1+1)g) g^{n_1-1} \big) 0^{l_4 + \mathsf v_0 (F_2)}
\]
Therefore
\[
m_1= l_1+l_2+l_3+l_4+1+\mathsf v_0 (F_1) \in L_k
\]
and
\[
m_2= l_1+l_2+l_3+l_4+\mathsf v_g (F_2)-(n_1-1)+2+\mathsf v_0 (F_2) \in L_k
\]
which implies that $m_1-m_2=\mathsf v_0 (F_1)-\mathsf v_0 (F_2) - \mathsf v_g (F_2)+(n_1-1)-1$ is congruent to either $0$ or to $n_1-2$ or to $n_2-n_1$ modulo $n_2-2$. We distinguish three cases.

\smallskip
\noindent CASE 4.4.1: \ $\mathsf v_0 (F_2)+\mathsf v_g (F_2) \equiv \mathsf v_0 (F_1) + n_1-2 \mod n_2-2$.

This implies that $\mathsf v_0 (F_2)+\mathsf v_g (F_2) = \mathsf v_0 (F_1) + n_1-2$, and hence $|F_2| \ge \mathsf v_{-g} (F_2) \ge \mathsf v_g (F_2)+2 \ge n_1$, a contradiction.

\smallskip
\noindent CASE 4.4.2: \ $\mathsf v_0 (F_2)+\mathsf v_g (F_2) + (n_1-2) \equiv \mathsf v_0 (F_1) + n_1-2 \mod n_2-2$.

This implies that $\mathsf v_0 (F_2)+\mathsf v_g (F_2) = \mathsf v_0 (F_1)$ and hence $\mathsf v_0 (F_2)=0$. Therefore we obtain that $F_1 = \big( (n_1-1)g \big) 0^{\mathsf v_0 (F_1)}$ and $F_2 = \big( -(n_1-1)g \big) g^{\mathsf v_0 (F_1)} (-g)^{\mathsf v_0 (F_1)+2}$. Now consider a factorization $z_1$ of $SF_1$ containing the atom $X = \big( (n_1-1)g \big) g^{n_1+1}$. This gives rise to a factorization
\[
z_2 = z_1 X^{-1} \Big( \big( -(n_1-1)g\big) g^{n_1-1} \Big) \big( (-g)g \big)^{\mathsf v_0 (F_1)+2} 0^{-\mathsf v_0 (F_1)}
\]
of length $|z_2|=|z_1|+2$, a contradiction.

\smallskip
\noindent CASE 4.4.3: \ $\mathsf v_0 (F_2)+\mathsf v_g (F_2) + (n_2-n_1) \equiv \mathsf v_0 (F_1) + n_1-2 \mod n_2-2$.

Since $n_2=2n_1$,  the congruence simplifies to $\mathsf v_0 (F_2)+\mathsf v_g (F_2) + 2 \equiv \mathsf v_0 (F_1) \mod n_2-2$ which implies that $\mathsf v_0 (F_2)+\mathsf v_g (F_2) + 2 = \mathsf v_0 (F_1)$. Thus $\mathsf v_0 (F_2)=0$, $F_1 = \big( (n_1-1)g \big) 0^{\mathsf v_0 (F_1)}$, and $F_2 = \big( -(n_1-1)g \big) g^{\mathsf v_0 (F_1)-2} (-g)^{\mathsf v_0 (F_1)}$.
Now consider a factorization $z_1$ of $SF_1$ containing the atom $X = \big( (n_1-1)g \big) (-g)^{n_1-1}$. This gives rise to a factorization
\[
z_2 = z_1 X^{-1} \Big( \big( -(n_1-1)g\big) (-g)^{n_1+1} \Big) \big( (-g)g \big)^{\mathsf v_0 (F_1)-2} 0^{-\mathsf v_0 (F_1)}
\]
of length $|z_2|=|z_1|-2$, a contradiction.
\qed{(Proof of {\bf A6})}

\smallskip
We state the final assertion
\begin{enumerate}
\item[{\bf A7.}\,] $n_1-1 \le \mathsf c_{\mathcal B_{\langle g \rangle} (G)} (D_k)$.
\end{enumerate}

\smallskip
{\it Proof of} \,{\bf A7}.\,  By {\bf A5}, we have
\[
\mathsf L_{\mathcal B (G)} (B_k) \ \ = \bigcup_{z \in \mathsf Z_{\mathcal B_{\langle g \rangle} (G)} (D_k)} \mathsf L_{\mathcal B (\langle g \rangle)}  \big( C_k \sigma (z) \big) \,,
\]
and the union on the right hand side consists of at least two distinct sets which are not contained in each other.
Assume to the contrary that $\mathsf c_{\mathcal B_{\langle g \rangle} (G)} (D_k) \le n_1-2$ and choose a factorization $z_0 \in \mathsf Z_{\mathcal B_{\langle g \rangle}(G)} (D_k)$.

We assert that for each $z \in \mathsf Z_{\mathcal B_{\langle g \rangle}(G)} (D_k)$ there exists an $l (z) \in \Z$ such that $\sigma (z) = \sigma (z_0) 0^{-l (z)} \big( (-g)g \big)^{ l (z)}$. Let $z \in \mathsf Z_{\mathcal B_{\langle g \rangle}(G)} (D_k)$ be given, and let $z_0, \ldots, z_k=z$ be an $(n_1-2)$-chain of factorizations concatenating $z_0$ and $z$. Since $\mathsf d (z_{i-1}, z_i) < n_1-1$, it follows that the pair $(z_{i-1},z_i)$ is of type (ii) in {\bf A6} for each $i \in [1,k]$. Therefore $\sigma (z_i) = \sigma (z_{i-1}) 0^{-l_i} \big( (-g)g \big)^{l_i}$ for some $l_i \in \Z$ and each $i \in [1,k]$, and hence the assertion follows with $l (z) = l_1+ \ldots + l_k$.

We choose a factorization $z^* \in \mathsf Z_{\mathcal B_{\langle g \rangle}(G)} (D_k)$ such that
\[
l (z^*) = \max \{ l (z) \mid z \in \mathsf Z_{\mathcal B_{\langle g \rangle}(G)} (D_k) \} \,,
\]
and assert that
\[
\mathsf L_{\mathcal B (\langle g \rangle)} ( C_k \sigma (z)) \subset \mathsf L_{\mathcal B (\langle g \rangle)} (C_k \sigma (z^*)) \quad \text{for each} \quad z \in \mathsf Z_{\mathcal B_{\langle g \rangle}(G)} (D_k) \,.
\]
Let $z \in \mathsf Z_{\mathcal B_{\langle g \rangle}(G)} (D_k)$ be given. Then $\sigma (z^*) = \sigma (z) 0^{-(l(z^*)-l(z))} \big( (-g)g \big)^{l(z^*)-l(z)}$. If
\[
y \in \mathsf Z_{\mathcal B (\langle g \rangle)} (C_k \sigma (z)) \,, \quad \text{ then} \quad y 0^{-(l(z^*)-l(z))} \big( (-g)g \big)^{l(z^*)-l(z)} \in \mathsf Z_{\mathcal B (\langle g \rangle)} (C_k \sigma (z^*))
\]
is a factorization of length $|y|$, and hence $\mathsf L_{\mathcal B (\langle g \rangle)} ( C_k \sigma (z)) \subset \mathsf L_{\mathcal B (\langle g \rangle)} (C_k \sigma (z^*))$.

Therefore we obtain that
\[
\mathsf L_{\mathcal B (G)} (B_k) \ \ = \bigcup_{z \in \mathsf Z_{\mathcal B_{\langle g \rangle} (G)} (D_k)} \mathsf L_{\mathcal B (\langle g \rangle)}  \big( C_k \sigma (z) \big)  \ \ = \ \ \mathsf L (C_k \sigma (z^*)) \,, \ \ \text{a contradiction to the fact}
\]
that the union consists of  least two distinct sets  not contained in each other. \qed{(Proof of {\bf A7})}

\smallskip
Using {\bf A7} and Proposition \ref{2.5}.2, we infer that
\[
n_1-1 \le \mathsf c_{\mathcal B_{\langle g \rangle} (G)} (D_k) \le \mathsf c ( \mathcal B_{\langle g \rangle} (G) ) = \mathsf c (\mathcal B (G/\langle g \rangle)) \le \mathsf D (G/\langle g \rangle) = \mathsf D (H) \le n_1 \,.
\]
We distinguish two cases.

\smallskip
\noindent CASE 1: \ $\mathsf c ( \mathcal B (G/\langle g \rangle)) = n_1$.

Then $\mathsf D (G/\langle g \rangle) = n_1$, Proposition \ref{2.4}.1 implies that $G/\langle g \rangle$ is either cyclic of order $n_1$ or an elementary $2$-group of rank $n_1-1$. Since $H \cong G/\langle g \rangle$, it follows that $G \cong C_{n_1} \oplus C_{n_2}$ or $G \cong C_2^{n_1-1} \oplus C_{n_2}$.

\smallskip
\noindent CASE 2: \ $\mathsf c ( \mathcal B (G/\langle g \rangle)) = n_1-1$.

We distinguish two cases.

\smallskip
\noindent CASE 2.1: \ $\mathsf D (G/\langle g \rangle) = n_1$.

Then Proposition \ref{2.4}.2 implies that $G/\langle g \rangle$ is isomorphic either to $C_2 \oplus C_{n_1-1}$, where $n_1-1$ is even, or to $C_2^{n_1-4} \oplus C_4$.  Since $H \cong G/\langle g \rangle$, it follows that $G \cong C_2 \oplus C_{n_1-1} \oplus C_{n_2}$ or $G \cong C_2^{n_1-4} \oplus C_4 \oplus C_{n_2}$.

\smallskip
\noindent CASE 2.2: \ $\mathsf D (G/\langle g \rangle) = n_1-1$.

Then $\mathsf c ( \mathcal B (G/\langle g \rangle)) =\mathsf D (G/\langle g \rangle) = n_1-1$, and (again by Proposition \ref{2.4}.1) $G/\langle g \rangle$ is cyclic of order $n_1-1$ or an elementary $2$-group of rank $n_1-2$. If $G/\langle g \rangle$ is cyclic, then $G$ has rank two and $\mathsf d (G) = \mathsf d (G/\langle g \rangle) + \mathsf d (\langle g \rangle) = n_1-2+n_2-1 < n_1+n_2-2 = \mathsf d (G)$, a contradiction. Thus $G/\langle g \rangle$ is an elementary $2$-group and  $G \cong C_2^{n_1-2} \oplus C_{n_2}$.
\end{proof}

\begin{proof}[Proof of Theorem \ref{1.1}]
Let  $G$ be an abelian group such that $\mathcal L (G) = \mathcal L (C_{n_1} \oplus C_{n_2})$ where $n_1, n_2 \in \N$ with $n_1 \t n_2$ and $n_1+n_2 > 4$.

Proposition \ref{6.1} implies that $G$ is finite with $\exp (G) = n_2$ and $\mathsf d (G) = \mathsf d (C_{n_1} \oplus C_{n_2})=n_1+n_2-2$. If  $n_1 = n_2$, then $G \cong C_{n_1} \oplus C_{n_2}$ by Proposition \ref{6.1}.2.
Thus we may suppose that $n_1 < n_2$, and we set $G = H \oplus C_{n_2}$ where $H \subset G$ is a subgroup with $\exp (H) \t n_2$. If $n_1 \in [1, 5]$, then the assertion follows from Proposition \ref{6.2}.3, and hence we suppose that   $n_1 \ge 6$. Since $\mathcal L (G) = \mathcal L (C_{n_1} \oplus C_{n_2})$, Proposition \ref{3.5} implies that, for each $k \in \N$, the sets
\[
L_k = \Big\{(kn_2+3) + (n_1-2)+(n_2-2) \Big\} \cup \Big(
(2k+3) + \{0, n_1-2, n_2-2 \} + \{\nu (n_2-2) \mid \nu \in [0,k] \}  \Big)
\]
are in $\mathcal L (G)$. Therefore Proposition \ref{6.5} implies that $G$ is isomorphic to one of the following groups
\[
C_{n_1}\oplus C_{n_2} \,, \  C_2^{s} \oplus C_{n_2} \ \text{with} \  s \in \{n_1-2, n_1-1\} \,, C_2^{n_1-4} \oplus C_4 \oplus C_{n_2} \,,
\ C_2 \oplus C_{n_1-1} \oplus C_{n_2} \ \text{ with} \  2 \t (n_1-1) \t n_2   \,.
\]
Since
\[
\mathsf d^* (C_2^{n_1-4} \oplus C_4 \oplus C_{n_2}) = n_1+n_2-2 = \mathsf d (G)
\quad \text{and} \quad
\mathsf d^* ( C_2 \oplus C_{n_1-1} \oplus C_{n_2}) = n_1+n_2-2 = \mathsf d ( G ) \,,
\]
Proposition \ref{6.2}.2 implies that $G$ cannot be isomorphic to any of these two groups. Proposition \ref{3.7} (with $k=0$, $n=n_2$, and $r=n_1-1$) implies that
\[
\{2, n_2, n_1+n_2-2\} \in \mathcal L (C_2^{n_1-2} \oplus C_{n_2}) \subset \mathcal L (C_2^{n_1-1} \oplus C_{n_2}) \,.
\]
However, Proposition \ref{5.1} shows that $\{2, n_2, n_1+n_2-2\} \notin \mathcal L (C_{n_1} \oplus C_{n_2}) = \mathcal L (G)$ whence $G \cong C_{n_1} \oplus C_{n_2}$.
\end{proof}

\providecommand{\bysame}{\leavevmode\hbox to3em{\hrulefill}\thinspace}
\providecommand{\MR}{\relax\ifhmode\unskip\space\fi MR }
\providecommand{\MRhref}[2]{%
  \href{http://www.ams.org/mathscinet-getitem?mr=#1}{#2}
}
\providecommand{\href}[2]{#2}

\end{document}